\documentclass[10pt,a4paper]{article}
\usepackage{amsmath, amsthm, amssymb, amsfonts}
\usepackage[T1]{fontenc}
\usepackage{graphicx, graphics, wrapfig}
\usepackage{makeidx}
\usepackage{psfrag}
\usepackage{color}
\usepackage{bm}
\usepackage{sidecap}
\usepackage{caption,subcaption}

\usepackage{enumitem}
\numberwithin{equation}{section}

\theoremstyle{plain}

\newtheorem{theorem}{Theorem}[section]

\newtheorem{lemma}[theorem]{Lemma}

\newtheorem{corollary}[theorem]{Corollary}
\newtheorem{conjecture}[theorem]{Conjecture}
\newtheorem{proposition}[theorem]{Proposition}

\theoremstyle{definition}

\newtheorem{definition}[theorem]{Definition}

\newtheorem{challenge}[theorem]{Challenge}
\theoremstyle{remark}
\newtheorem{remark}[theorem]{Remark}
\newtheorem{remarks}[theorem]{Remarks}
\newtheorem{example}[theorem]{Example}

\DeclareSymbolFont{AMSb}{U}{msb}{m}{n}
\DeclareMathSymbol{\N}{\mathalpha}{AMSb}{"4E}
\DeclareMathSymbol{\R}{\mathalpha}{AMSb}{"52}
\DeclareMathSymbol{\Z}{\mathalpha}{AMSb}{"5A}
\DeclareMathSymbol{\C}{\mathalpha}{AMSb}{"43}
\DeclareMathSymbol{\D}{\mathalpha}{AMSb}{"44}
\DeclareMathSymbol{\s}{\mathalpha}{AMSb}{"53}

\DeclareMathOperator{\diam}{diam}

\DeclareMathOperator{\Cpl}{Cpl}
\DeclareMathOperator{\Opt}{Opt}
\DeclareMathOperator{\Par}{Par}
\DeclareMathOperator{\Id}{Id}
\DeclareMathOperator{\Inv}{Inv}
\DeclareMathOperator{\Sym}{Sym}

\DeclareMathOperator{\size}{size}
\DeclareMathOperator{\supp}{supp}

\DeclareMathOperator{\sym}{sym}
\DeclareMathOperator{\Exp}{Exp}

\DeclareMathOperator{\Lip}{Lip}
\DeclareMathOperator{\Ric}{Ric}
\DeclareMathOperator{\dis}{dis}
\DeclareMathOperator{\length}{length}
\DeclareMathOperator{\EExp}{{\mathbb E}xp}
\newcommand{\borel}{\mathcal{B}}
\newcommand{\prob}{\mathcal{P}}

\renewcommand{\d}{{\sf{d}}}
\newcommand{\f}{{\sf{f}}}

\newcommand{\m}{\mathfrak{m}}
\newcommand{\X}{\mathcal{X}}
\newcommand{\Leb}{\mathfrak{L}}
\newcommand{\XX}{\mathbb{X}}
\newcommand{\LL}{\mathbb{L}}

\renewcommand{\Ric}{{\sf{Ric}}}
\newcommand{\Riem}{{\sf{R}}}
\newcommand{\scal}{{\sf{s}}}

\newcommand{\TT}{\mathbb{T}}
\newcommand{\XXn}{\mathbb{X}^{(n)}}
\newcommand{\YY}{\mathbb{Y}}
\newcommand{\YYn}{\mathbb{Y}^{(n)}}
\newcommand{\M}{\mathbb{M}}
\newcommand{\Mn}{\mathbb{M}^{(n)}}
\newcommand{\Mnl}{\mathbb{M}^{(n)}_\leq}
\newcommand{\spec}{:\,}

\newcommand{\bfdelta}{\bm\delta}

\newcommand{\Expf}[1]{\ensuremath{\Exp_{\X_{#1}}}}
\newcommand{\dol}[1]{\ensuremath{\bar{\bar{#1}}}}
\newcommand{\geod}[2]{\ensuremath{(\X_t)_{#1\leq t\leq #2}}}
\newcommand{\DD}{
\Delta\!\!\!\!\Delta}
\newcommand{\Grad}{\nabla\!\!\!\!\nabla}
\newcommand{\ol}[1]{\ensuremath{\bar{#1}}}

\newcommand{\aslope}[1]{\ensuremath{\left|D^+\U({#1})\right|}}

\newcommand{\F}{\mathcal{F}}
\newcommand{\G}{\mathcal{G}}
\renewcommand{\H}{\mathcal{H}}
\newcommand{\U}{\mathcal{U}}
\newcommand{\be}{{(\varepsilon)}}

\newcommand{\auf}{{[\:\!\![}}
\newcommand{\zu}{{]\:\!\!]}}

\newcommand{\PP}{\mathbb{P}}
\newcommand{\base}{\flat}
\newcommand{{\push}}{*}

\newcommand{\wichtig}{\textcolor[rgb]{0.00,0.00,1.00}}
\makeindex

\begin{document}

\title{The space of spaces: \\  curvature bounds and gradient flows on the space of\\ metric measure spaces\\[.6cm]
}

\medskip

	\author{Karl-Theodor Sturm\\[.6cm]}
\date{}

\maketitle


\begin{abstract}
\noindent
Equipped with the $L^{2,q}$-distortion distance $\DD_{2,q}$, the space $\XX_{2q}$ of all 
metric measure spaces $(X,\d,\m)$ is proven to have nonnegative curvature in the sense of Alexandrov.
Geodesics and tangent spaces are characterized in detail.
Moreover, classes of semiconvex functionals 
and their gradient flows on $\ol\XX_{2q}$ are presented.\\[.6cm]
\end{abstract}

\section*{Introduction and Main Results at a Glance}

\subsubsection*{}
{\bf I.}\
The basic object of this paper is the space $\XX_{pq}$ of isomorphism classes 
of metric measure spaces
for given numbers $p,q\in[1,\infty)$. 
A \emph{metric measure space} is a triple $(X,\d,\m)$ consisting of a space $X$, a complete separable metric $\d$ on $X$ and a Borel probability measure on it (more precisely,  a probability measure  on the Borel $\sigma$-field  induced by the metric $\d$ on $X$).
We will always require that its $L^{pq}$-size
$\left( \int_X \int_X \d^{pq}(x,y) d\m(x) d\m(y)\right)^{1/p}$ is finite.
Two metric measure spaces with full supports are \emph{isomorphic} if there exists a measure preserving isometry between them. \\
We will consider $\XX_{pq}$ as a metric space equipped with the so-called $L^{p,q}$-distortion distance $\DD_{p,q}$ to be presented below.
One of our main results is that for each $q\in[1,\infty)$
\begin{itemize}\item[$\wichtig\blacktriangleright$]\wichtig{the metric space $(\XX_{2q},\DD_{2,q})$ has nonnegative curvature in the sense of Alexandrov.}
\end{itemize}
Both the triangle comparison and the quadruple comparison will be verified.

	\subsubsection*{}
{\bf II.}\
The \emph{$L^{p,q}$-distortion distance}  between two  metric measure spaces $(X_0,\d_0,\m_0)$ and $(X_1,\d_1,\m_1)$  is
	defined for $p,q\in[1,\infty)$ as
	\begin{eqnarray*}
	\lefteqn{\DD_{p,q}\Big((X_0,\d_0,\m_0),(X_1,\d_1,\m_1)\Big)}\\
&=& \inf_{\ol{\m}\in \Cpl(\m_0,\m_1)}   \bigg( \int_{X_0\times X_1} \int_{X_0\times X_1}
		\Big| \d_0^q(x_0,y_0)-\d_1^q(x_1,y_1) \Big|^p d\ol{\m}
		(x_0,x_1)d\ol{\m}(y_0,y_1)\bigg)^{1/p}
	\end{eqnarray*}
where the infimum is taken over all \emph{couplings} of $\m_0$ and $\m_1$, i.e. over all probability measures $\ol\m$ on $X_0\times X_1$ with prescribed marginals
$(\pi_0)_{\push}\ol{\m}=\m_0$ and $(\pi_1)_{\push} \ol{\m}=\m_1$. There always exists an \emph{optimal coupling} for which the infimum is attained.
Convergence w.r.t. the $L^{p,q}$-distortion distance can be characterized as convergence w.r.t. the $L^{0,q}$-distortion distance together with convergence of the $L^{pq}$-size. The $L^{0,1}$-distortion distance induces the same topology as the $L^0$-transportation distance (also known as Prohorov-Gromov metric) which in turn is equivalent to Gromov's box metric $\underline\square_\lambda$.
\\
One of our fundamental results -- with far reaching applications --    is a complete, explicit characterization of $\DD_{p,q}$-geodesics in $\XX_{pq}$:
\begin{itemize}\item[$\wichtig\blacktriangleright$]\wichtig{				
For  each  optimal coupling
  $\ol{\m}$, the family of metric
			measure spaces
\[\Big({X_0\times X_1},\big((1-t)\,{\d}_0^q+t\,\d_1^q\big)^{1/q},\ol{\m}\Big)\qquad\text{for }t\in (0,1)\]
				defines a geodesic  in $\XX_{pq}$ connecting $(X_0,\d_0,\m_0)$ and $(X_1,\d_1,\m_1)$.}
\end{itemize}
	\begin{itemize}\item[$\wichtig\blacktriangleright$]\wichtig{	
If $p\in (1,\infty)$, then each geodesic  in $\XX_{pq}$ is of this form.}
\end{itemize}
For each  metric measure space $(X,\d,\m)$, a geodesic ray through it is given by $(X,t\cdot\d,\m)$ for $t\ge0$.
Its initial point is the \emph{one-point space} $\bfdelta$ (= the equivalence class of metric measure spaces whose supports consist of one point).
In the particular case $p=2$,
$(\XX_{pq},\DD_{p,q})$ is a \emph{cone} with apex $\bfdelta$ over its unit sphere.

\subsubsection*{}
{\bf III.}\
$\XX_{pq}$ is quite a huge space: it contains all Riemannian manifolds, GH-limits of Riemannian manifolds (cf. \cite{CC97, CC00a, CC00b}), Finsler spaces (cf. \cite{Sh01}, \cite{OSt09}), finite dimensional Alexandrov spaces (cf.  \cite{BGP92}, \cite{OS94}), groups (cf. \cite{Woe00}),   graphs (cf. \cite{Del}), fractals (cf. \cite{Kig}) as well as many infinite dimensional spaces (cf. \cite{BSC05}) --
provided the respective spaces, manifolds, graphs etc. have finite volume (which then is assumed to be normalized). In particular, it contains all metric measure spaces with generalized lower bounds for the Ricci curvature in the sense of Lott-Sturm-Villani \cite{St06}, \cite{LV09}.

However, $\XX_{pq}$  is not complete w.r.t. $\DD_{p,q}$.
Fortunately,
 each element in its completion $\ol\XX_{pq}$ again can be represented as a triple $(X,\d,\m)$ -- more precisely, as an equivalence class (`homomorphism class') of such triples --  where $X$ is a Polish space, $m$ a Borel probability measure on $X$ and $\d$ a symmetric, $L^{pq}$-integrable Borel  function on $X\times X$ which satisfies the triangle inequality almost everywhere. That is,
\begin{itemize}\item[$\wichtig\blacktriangleright$]\wichtig{
the completion of $\XX_{pq}$ is the space $\ol\XX_{pq}$ of pseudo metric measure spaces.} \end{itemize}
In the particular case $p=2$, the `space of spaces' $(\ol\XX_{2q},\DD_{2,q})$  is a complete, geodesic 
space of nonnegative curvature (infinite dimensional Alexandrov space) and as such
allows for a variety of geometric concepts
including  space of geodesic directions, 
tangent cones, exponential maps,  gradients   of semiconvex functions, and (downward) gradient flows.

\subsubsection*{}
{\bf IV.}\
 A deeper insight into the tangent structure of $\ol\XX_{2q}$
is obtained by embedding $\ol\XX_{2q}$ isometrically as a closed convex subset into a complete metric space $\YY$ which consists of equivalence classes of triples $(X,\f,\m)$ -- called \emph{gauged measure spaces} --
with $X$ being Polish, $\f$ a symmetric $L^2$-function on $X^2$ (no longer required to satisfy the triangle inequality) and $\m$ a Borel probability measure on $X$.
It turns out that
\begin{itemize}\item[$\wichtig\blacktriangleright$]\wichtig{
the metric space $(\YY,\DD_{2,1})$
is isometric to the quotient space
$L^2_s(I^2,\Leb^2)/\Inv(I,\Leb)$}
\end{itemize}
where $L^2_s(I^2,\Leb^2)$ denotes the space of symmetric $L^2$-functions on the unit square 
and $\Inv(I,\Leb)$ denotes the space of measure preserving transformations of the unit interval $I=[0,1]$.
Being isometric to the quotient of a Hilbert space under the action of a semigroup (acting isometrically via pull back), it comes as no surprise that
$(\YY,\DD_{2,1})$ is again a complete, geodesic metric space of nonnegative curvature.

A more detailed analysis of the tangent structure allows to regard $\YY$ as an \emph{infinite dimensional Riemannian orbifold}. In fact, one always may choose a homomorphic representative $(X,\f,\m)$ without atoms.
Then
\begin{itemize}\item[$\wichtig\blacktriangleright$]\wichtig{
the tangent space of the triple $(X,\f,\m)$ is given by
\[\TT_{(X,\f,\m)}\YY=L^2_s(X^2,\m^2)/\Sym(X,\f,\m)\]
where $\Sym(X,\f,\m)$ denotes the  {symmetry group} (or isotropy group) of
$(X,\f,\m)$.}
\end{itemize}
In particular, if the given space $(X,\f,\m)$ has no non-trivial symmetries then its tangent space is Hilbertian and for $\f'\in L^2_s(X^2,\m^2)$
\[\EExp_{(X,\f,\m)}(\f')=(X,\f+\f',\m).\]
These results are very much in the spirit of Otto's Riemannian calculus \cite{Ot01} on the $L^2$-Wasserstein space ${\mathcal P}_2(\R^n)$ which also leads to lower bounds on the sectional curvature (cf. \cite{Lo08}) and quite detailed structural assertions on the tangent space (cf. \cite{ags}). The latter, however, is essentially limited to `regular' points (i.e. absolutely continuous measures) whereas the above results also provide precise assertions on the tangent structure for `non-regular' points (i.e. spaces with non-trivial symmetries).

\subsubsection*{}
{\bf V.}\  
To simplify the presentation, let us now restrict to the case $p=2, q=1$ and put $\XX:= \XX_{2,1}$,  $\DD:= \DD_{2,1}$.
For major classes of functionals on $\ol\XX$ one can explicitly calculate directional derivatives (of any order) and thus obtains sharp bounds for gradients and Hessians.
For each Lipschitz continuous, semiconvex $\U: \ol\XX\to\R$ there exists a unique downward gradient flow in $\ol\XX$. Any lower bound $\kappa$ for the Hessian of $\U$ yields an  
\begin{itemize}\item[$\wichtig\blacktriangleright$]\wichtig{
 Lipschitz estimate for the downward gradient flow
\begin{equation}\label{grad-lip}
\DD\Big((X_t,\d_t,\m_t),(X'_t,\d'_t,\m'_t)\Big)\le e^{-\kappa\,t}\cdot \DD\Big((X_0,\d_0,\m_0),(X'_0,\d'_0,\m'_0)\Big).\end{equation}
}
\end{itemize}
Among these functionals are `polynomials' of order $n\in\N$. They are of the form
\begin{equation*}
\label{U}
\U\Big(( X,\d,\m)\Big)=\int_{X^n}u\bigg( \Big( \d(x^i,x^j)\Big)_{1\le i<j\le n}\bigg)d\m^n(x^1,\ldots,x^n)
\end{equation*}
where $u$ is some   smooth function  on $\R^{\frac{n(n-1)}2}$. Of particular interest will be polynomials of order $n=4$ which allow to determine whether a given curvature bound (either from above or from below) in the sense of Alexandrov is satisfied. For each $K\in\R$, there
 \begin{itemize}\item[$\wichtig\blacktriangleright$]\wichtig{
exist Lipschitz continuous,  semiconvex functionals $\G_K$ and $\H_0: \ \ol\XX\to\,[0,\infty)$ with the property that
for each geodesic metric measure space $(X,\d,\m)$
\begin{eqnarray*}
\G_K\Big((X,\d,\m)\Big)=0\quad&\Longleftrightarrow&\quad (X,\d,\m) \text{  has curvature }\ge K\\
\H_0\Big((X,\d,\m)\Big)=0\quad&\Longleftrightarrow&\quad (X,\d,\m) \text{  has curvature }\le 0.
\end{eqnarray*}}
\end{itemize}

	\subsubsection*{}
{\bf VI.}\
Given any `model space' $(X^\star,\d^\star,\m^\star)$ within $\ol\XX$, we define a functional $\F:\ol\XX\to\R_+$ whose downward gradient flow will push each pseudo metric measure space $(X,\d,\m)$ towards the given model space.
We put
\[\F\Big((X,\d,\m)\Big)=\frac12
 \int_0^\infty \int_X \left[\int_0^r \big(v_t(x)-v^\star_t\big)\,dt\right]^2 d\m(x) \rho_r dr.
\]
Here $v_r(x)=m(B_r(x))$ denotes the volume growth of balls in the  space $(X,\d,\m)$ whereas $r\mapsto v^\star_r$ is the volume growth in $(X^\star,\d^\star,\m^\star)$ and $r\mapsto\rho_r$ is some  positive ('weight') function on  $\R_+$.
\begin{itemize}\item[$\wichtig\blacktriangleright$]\wichtig{
The functional $\F$ is $\lambda$-Lipschitz and $\kappa$-convex}
\end{itemize}
with $\lambda=\int_0^\infty r\rho_r\,dr$ and $\kappa=-\sup_{r>0}[r\rho_r]$.
In particular, the downward gradient flow for $\F$ satisfies a Lipschitz bound (\ref{grad-lip}) with constant $e^{|\kappa|\,t}$.
\begin{itemize}\item[$\wichtig\blacktriangleright$]\wichtig{
The functional $\F$ will vanish if and only if
\[v_r(x)=v^\star_r\qquad\text{for every }r\ge0 \text{ and $\m$-a.e. } x\in X.\]}
\end{itemize}
If $X$ is a Riemannian manifold and $v^\star$ denotes the volume growth of the Riemannian model space $\M^{n,\kappa}$ for $n\le3$ and $\kappa>0$ then the previous property implies that $X$ is the model space $\M^{n,\kappa}$.
\begin{itemize}\item[$\wichtig\blacktriangleright$]\wichtig{
The gradient of $-\F$ at the point $(X,\d,\m)$ is explicitly given as the function
$\f\in L^2_s(X^2,\m^2)$ with
	\[
		\f(x,y) =\int_0^\infty \left( \frac{v_r(x)+v_r(y)}{2}-v^\star_r\right) \bar{\rho}\big(r\vee\d(x,y)\big)dr
	\]
where $\bar{\rho}(a)= \int_a^\infty  \rho_r dr$.}
\end{itemize}
The infinitesimal evolution of $(X,\d,\m)$ under the downward gradient flow for $\F$ on $\ol\XX$ is given by $(X,\d_t,\m)$ with
	\[
		\d_t(x,y)=\d(x,y)+ t\f(x,y)+O(t^2)
	\]
	and $\f$ as above.
	That is, $\d(x,y)$ will be enlarged if -- in average w.r.t. the radius $r$ -- the volume  of balls $B_r(x)$ and $B_r(y)$ in $X$ is too large (compared with the volume $v^\star_r$ of balls in the model space), and
	$\d(x,y)$ will be reduced if the volume of balls is too small.

\medskip

{
In a broader sense, the downward gradient flow for $\F$ is related to Ricci flow. Indeed, on the space of Riemannian manifolds, the functionals $\F^{(\epsilon)}$ for a suitable
sequence of weight functions $\rho^{(\epsilon)}$ (converging to $\delta_0$)
 will converge  to
\[\frac12 \int_X (\scal(x)-\scal^\star)^2 d\m(x),\]
a modification of the Einstein-Hilbert functional which plays a key role in Perelman's program \cite{Per02},
cf. \cite{MT07}, \cite{KL08}.

Note that Ricci flow  does \emph{not} depend continuously on the initial data, in particular, no Lipschitz estimate of the form (\ref{grad-lip}) will hold. Also note that no ``regularizing'' gradient flow is known
which respects lower curvature bounds in the sense of  Alexandrov (Petrunin \cite{Pet07}: ``Please deform an Alexandrov's space'').  Similarly,
no ``regularizing'' gradient flow is known
which respects lower Ricci bounds in the sense of  Lott-Sturm-Villani
 \cite{St06}, \cite{LV09}.

\subsubsection*{}
{\bf VII.}\  
With respect to the parameter $p$, only the value $p=2$ plays a specific role in the analysis of the $L^{p,q}$-distortion distance. It is the only value of $p$ for which $\big(\XX_{pq}, \DD_{p,q}\big)$ becomes a space of nonnegative curvature in the sense of Alexandrov.

With respect to $q$, two values are of interest. The value $q=1$ is  the most natural one from the point of view of transportation theory and image analysis. And, of course, it also leads to the most simple formulas.
  The value $q=2$ is the only value for which 
\begin{itemize}\item[$\wichtig\blacktriangleright$]
\wichtig{geodesic interpolations of spaces with nonnegative (or nonpositive) pre-curvature are again spaces with
nonnegative (or nonpositive, resp.) pre-curvature.}
\end{itemize}
Moreover,
 geodesic interpolations of distances in the case $q=2$ may be regarded as the metric counterpart to
 linear interpolations of metric tensors in Riemannian geometry.
To see this, assume that the optimal coupling  of two Riemannian spaces
 $(M_0,\d_0,\m_0)$ and  $(M_1,\d_1,\m_1)$
is given as $\ol\m=(\Id,\phi)_*\m_0$ with a diffeomorphism $\phi: M_0\to M_1$. Then 
\begin{itemize}\item[$\wichtig\blacktriangleright$]
\wichtig{
the connecting geodesic is
 $(M_0, \d_t,\m_0)$, $t\in (0,1)$, with
\begin{equation*}
 \d_t=\sqrt{(1-t)\d^2_0+t \phi^*\d_1^2}
\end{equation*}
for which the induced 
{length metric} $\d_t^*$    coincides with   the {Riemannian distance} for  the metric tensor
\begin{equation*}
 g_t=(1-t)g_0+t \phi^*g_1.
\end{equation*}}
\end{itemize}

\subsection*{}
\emph{Acknowledgement.}
The author would like to thank
Fabio Cavalletti, Matthias Erbar, Martin Huesmann, Christian Ketterer and in particular Nora Loose for carefully reading  early drafts of this paper and for many valuable comments.
He also gratefully acknowledges stimulating discussions on topics of this paper with  Nicola Gigli, Jan Maas, Shin-ichi Ohta, Takashi Shioya, Asuka Takatsu and Anatoly Vershik in Bonn as well as during conferences in Pisa, Oberwolfach and Sankt Petersburg (May, June 2012).  In particular, he is greatly indebted to Andrea Mondino for
enlightening discussions on criteria for Riemannian manifolds to be balanced (Theorem, 8.13).

\medskip

The author also gratefully acknowledges financial support by the European Union through the ERC-AdG ``RicciBounds''
and by the DFG through the Excellence Cluster ``Hausdorff Center for Mathematics'' as well as through the Collaborative Research Center 1060.
\newpage

\tableofcontents

\newpage

\section{The Metric Space $(\XX_p,\DD_p)$}

\subsection{Metric Measure Spaces and Couplings}

Throughout this paper, a \emph{metric measure space} (briefly: \emph{mm-space})\index{mm-space $(X,\d,\m)$} will always be a triple
$(X,\d,\m)$ where
\begin{itemize}
	\item
		$(X,\d)$ is a complete separable metric space,
	\item
		$\m$ is a Borel probability measure on $X$.
	\end{itemize}
The latter means that $\m$ is a measure on $\borel(X)$ --
 the Borel $\sigma$-field associated with the Polish topology on $X$ induced by the metric $\d$ --
 with normalized total mass $\m(X)=1$.
In the literature, metric measure spaces are also called \emph{metric triples}.

The \emph{support} $\supp(X,\d,\m)$ of such a metric measure space -- or simply the support $\supp(\m)$ of the
measure $\m$ -- is the smallest closed set $X_0\subset X$ such that $\m(X\setminus X_0)=0$.
Occasionally, it will also be denoted by $X^\base$.\index{$X^\base$}
We say that $(X,\d,\m)$ has full support if $\supp(X,\d,\m)=X$. This, however, will not be required in general.
The \emph{diameter} or \emph{$L^\infty$-size} of a metric measure space $(X,\d,\m)$ is defined as the diameter of its support:
\[
	\diam(X,\d,\m)=\sup\Big\{ \d(x,y) \spec x,y\in \supp(X,\d,\m)\Big\}.
\]
For any $p\in[1,\infty)$, the \emph{$L^p$-size}\index{size} of  $(X,\d,\m)$ is defined as
\[
			\size_p(X,\d,\m):=\left( \int_X \int_X \d^p(x,y) d\m(x) d\m(y)\right)^{1/p}.
		\]
Obviously, $\size_p(X,\d,\m)\leq\size_{q}(X,\d,\m)\leq\diam(X,\d,\m)$ for all $1\le p \le q\le \infty$.

Given two mm-spaces $(X_0,\d_0,\m_0)$ and $(X_1,\d_1,\m_1)$ and a map $\psi: X_0\to X_1$, we define
\begin{itemize}
\item
the \emph{pull back}\index{${\psi}^*$}\index{push forward}\index{pull back} of the metric $\d_1$ through $\psi$ as the pseudo metric $\psi^\push\d_1$ on $X_0$ given by
\[(\psi^*\d_1)(x_0,y_0)=\d_1(\psi(x_0),\psi(y_0))\qquad\big(\forall x_0,y_0\in X_0\big);\]
\item
the \emph{push forward}\index{${\psi}_*$} of the probability measure $\m_0$ through $\psi$ -- provided $\psi$ is Borel measurable --  as the probability measure $\psi_{\push}\m_0$ on $(X_1,\borel(X_1))$ given by
\[(\psi_{\push}\m_0)(A_1)=\m_0\big(\psi^{-1}(A_1)\big)=\m_0\Big(\Big\{x_0\in X_0\spec \psi(x_0)\in A_1\Big\}\Big)\qquad \big(\forall A_1\in\borel(X_1)\big).\]
\end{itemize}

\begin{definition}
Given two mm-spaces $(X_0,\d_0,\m_0)$ and $(X_1,\d_1,\m_1)$, any probability measure $\ol{\m}$ on the product space $X_0\times X_1$ (equipped with the product topology and product $\sigma$-field) satisfying
\begin{equation}\label{coupling}
	(\pi_0)_{\push}\ol{\m}=\m_0, \quad (\pi_1)_{\push} \ol{\m}=\m_1
\end{equation}
is called \emph{coupling} of the measures $\m_0$ and $\m_1$. The measures $\m_0$ and $\m_1$ in turn will be called \emph{marginals} of $\ol\m$.
\end{definition}
Here $\pi_0$ and $\pi_1$ denote the projections from $X_0\times X_1$ to $X_0$ and $X_1$, resp.
Condition~\eqref{coupling} can be restated as: \[\ol{\m}(A_0\times X_1)=\m_0(A_0), \quad \ol{\m}(X_0\times A_1)=\m_1(A_1)\] for all $A_0\in \borel(X_0)$, $A_1\in \borel(X_1)$.
The set of all couplings of $\m_0$ and $\m_1$ will be denoted by $\Cpl(\m_0,\m_1)$.\index{c@$\Cpl(.,.)$}\index{coupling}

The set $\Cpl(\m_0,\m_1)$ is non-empty: it always contains the
\emph{product coupling} $\ol{\m}=\m_0\otimes \m_1$ (being uniquely defined by the requirement $\ol{\m}(A_0\times A_1)=\m_0(A_0)\cdot\m_1(A_1)$ for all $A_0\in \borel(X_0)$, $A_1\in \borel(X_1)$).
If one of the measures $\m_0$ and $\m_1$ is a Dirac then the product coupling is indeed the only coupling: $\Cpl(\delta_{x_0},\m_1)=\{\delta_{x_0}\otimes\m_1\}$.

\begin{lemma}\label{cpl-comp}
Given $\m_0$ and $\m_1$, the set of couplings $\Cpl(\m_0,\m_1)$ is a non-empty
	compact subset of $\prob(X_0\times X_1)$, the set of probability measures on $X_0\times X_1$ equipped with the weak topology.
\end{lemma}

\begin{proof} Obviously, $\Cpl(\m_0,\m_1)$ is a closed subset within $\prob(X_0\times X_1)$. (The projection maps are continuous functions.) The relative compactness (`tightness') follows from a simple application of Prohorov's theorem, see \cite{Vi09}, Lemma 4.4.
\end{proof}

For each measurable map $\psi:X_0\to X_1$ with $\psi_{\push} \m_0=\m_1$, a coupling of $\m_0$ and $\m_1$ is given by
\[\ol{\m}=(\Id,\psi)_{\push} \m_0.\]
In the particular case $X_0=X_1$, $\m_0=\m_1$, the choice $\psi=\Id$ leads to the \emph{diagonal coupling}\index{diagonal coupling}
\[d\ol\m(x,y)=d\delta_x(y)\,d\m_0(x).\]
More generally, for each mm-space $(X,\d,\m)$ and measurable maps $\psi_0:X\to X_0$, $\psi_1:X\to X_1$ with $(\psi_0)_{\push} \m=\m_0$,  $(\psi_1)_{\push} \m=\m_1$,
a coupling of $\m_0$ and $\m_1$ is given by
\[\ol{\m}=(\psi_0,\psi_1)_{\push} \m.\]
Indeed, any coupling is of this form -- and without restriction one may choose $(X,\d,\m)$ to be the unit interval $X=[0,1]$ equipped with the standard distance $\d(x,y)=|x-y|$ and the 1-dimensional Lebesgue measure $\m=\Leb^1$ on $[0,1]$, cf. Lemma \ref{standard borel}.

\begin{remark}
The concept of coupling of mm-spaces extends and improves (in an 'optimal' quantitative manner) the concepts of correspondence and $\varepsilon$-isometries between mm-spaces.
\begin{itemize}
\item
Every coupling $\ol{\m}$ of measures $\m_0$ and $\m_1$ induces a \emph{correspondence} between the supports of $(X_0,\d_0,\m_0)$ and $(X_1,\d_1,\m_1)$ by means of
\[ {\mathcal R}=\supp(\ol\m)\ \subset \ X_0\times X_1.\]
But of course, the measure $\ol\m$ itself bears much more information than its support.

\begin{figure}[h!]

	\begin{subfigure}{0.3\textwidth}
		
\psfrag{24}{$\frac{2}{4}$}
\psfrag{14}{$\frac{1}{4}$}
\psfrag{36}{$\frac{3}{6}$}
\psfrag{16}{$\frac{1}{6}$}
\psfrag{424}{$\frac{4}{24}$}
\psfrag{124}{$\frac{1}{24}$}
\psfrag{224}{$\frac{2}{24}$}
\psfrag{924}{$\frac{9}{24}$}
\psfrag{324}{$\frac{3}{24}$}
\psfrag{X0}{$X_0$}
\psfrag{X1}{$X_1$}
\psfrag{m}{$\ol{m}$}
\psfrag{m0}{$m_0$}
\psfrag{m1}{$m_1$}
\includegraphics[scale=0.3]{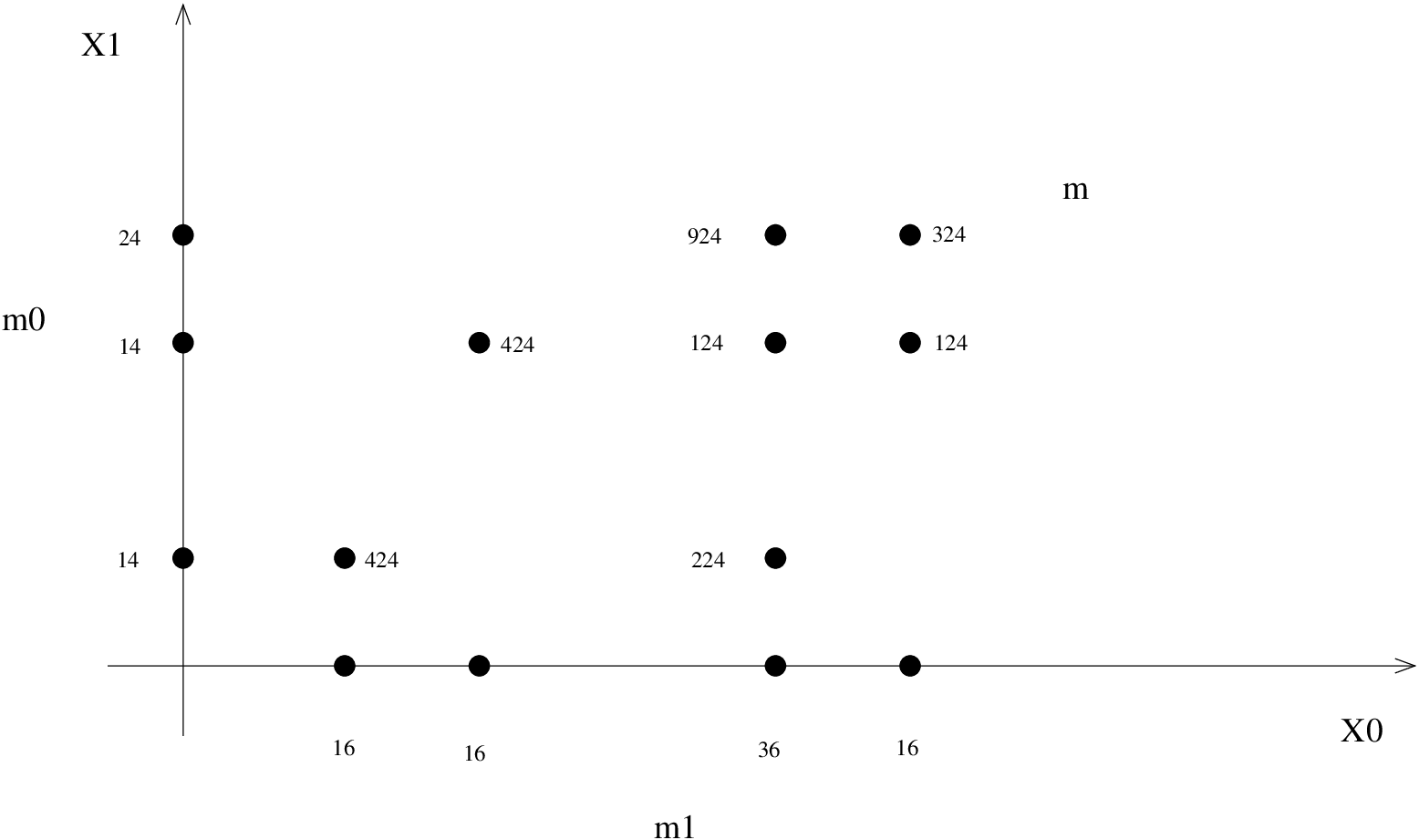}

\caption{\wichtig{Coupling}}

	\end{subfigure}
	\qquad \qquad \qquad \qquad
	\begin{subfigure}{0.3\textwidth}

\psfrag{X0}{$X_0$}
\psfrag{X1}{$X_1$}
\psfrag{suppm}{$\mathcal{R}=\supp(\ol{m})$}
\psfrag{suppm0}{$\supp(m_0)$}
\psfrag{suppm1}{$\supp(m_1)$}
\includegraphics[scale=0.3]{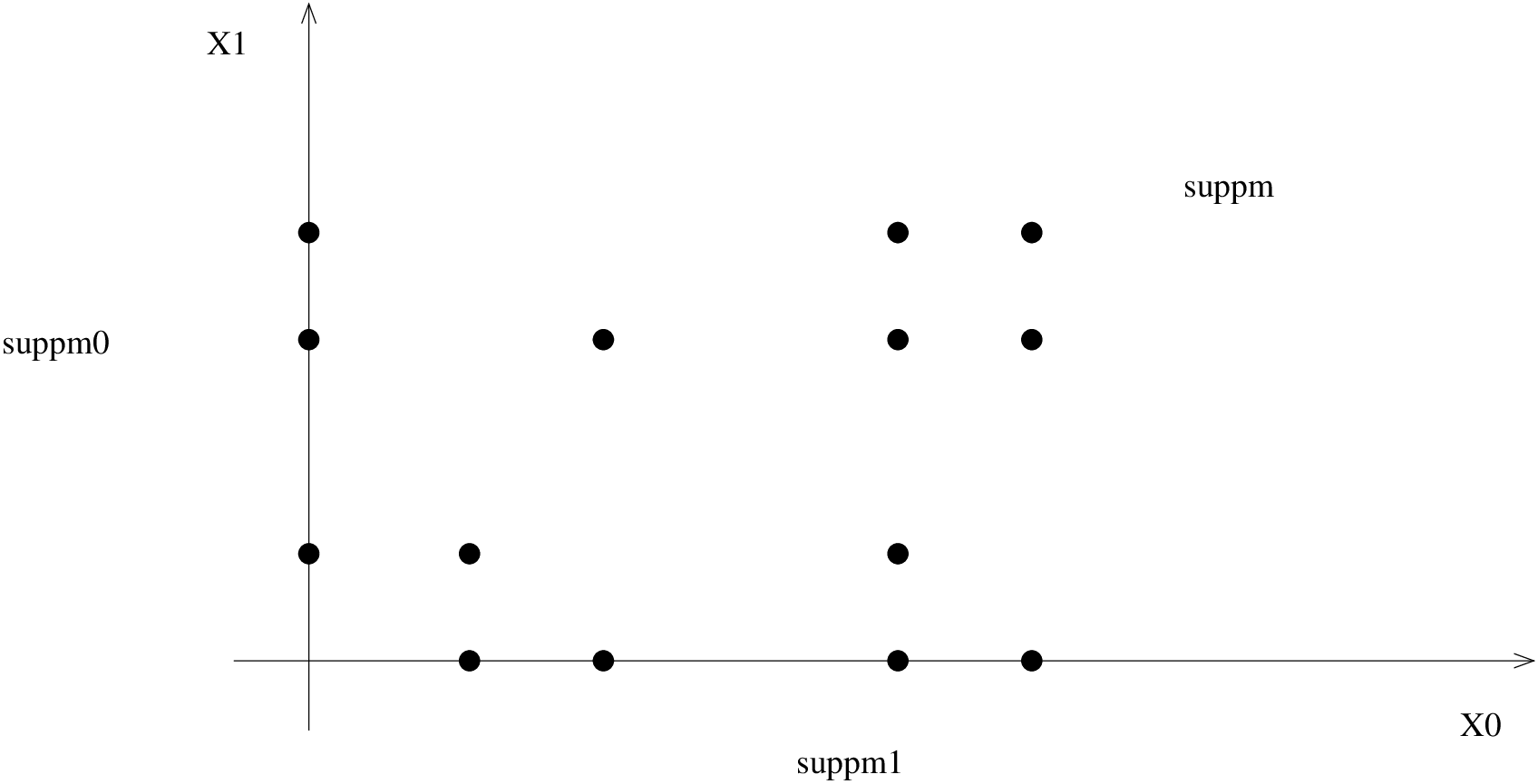}

\caption{\wichtig{Correspondence}}

		\end{subfigure}
		\caption{\wichtig{Coupling vs. Correspondence}}
\end{figure} 
\item
Every coupling $d\ol{\m}(x_0,x_1)$ of measures $d\m_0(x_0)$ and $d\m_1(x_1)$ admits a disintegration $d\ol\m_{x_0}(x_1)$ w.r.t. $d\m_0(x_0)$. That is there exist probability measures $d\ol\m_{x_0}(.)$ on $X_1$ s.t.
\[d\ol{\m}(x_0,x_1)=d\ol\m_{x_0}(x_1)\,d\m_0(x_0)\]
as measures on $X_0\times X_1$. This Markov kernel (`disintegration kernel') $d\ol\m_{x_0}(x_1)$ may be regarded as a replacement of $\varepsilon$-isometries $\psi: X_0\to X_1$. Instead of mapping points $x_0$ in $X_0$ to points $\psi(x_0)$ in $X_1$ --  or to $\varepsilon$-neighborhoods in $X_1$ -- we now map points $x_0$ in $X_0$ to probability measures
$\ol\m_{x_0}(.)$ on $X_1$.
\end{itemize}
\end{remark}

\begin{lemma}[Gluing lemma]\label{glue1}
Let $X_0,X_1,\ldots,X_k$ be Polish spaces and $\m_0,\m_1,\ldots,\m_k$ probability measures, defined on the respective $\sigma$-fields. Then for every choice of couplings
$\mu_i\in \Cpl(\m_{i-1},\m_i)$, $i=1,\ldots,k$, there exists a unique probability measure $\mu\in\prob(X_0\times X_1\times\ldots\times X_k)$ s.t.
\begin{equation}
\label{multiplecoupling}(\pi_{i-1},\pi_i)_{\push} \mu=\mu_i\qquad(\forall i=1,\ldots,k).\end{equation}
$\mu$ is called \emph{gluing}\index{gluing}\index{$\boxtimes$} of the couplings $\mu_1,\ldots,\mu_k$ and denoted by
\[\mu=\mu_1 \boxtimes\ldots\boxtimes\mu_k.\]

In particular,  $\mu$ has marginals $\m_0,\m_1,\ldots,\m_k$. That is,
$(\pi_i)_{\push} \mu=\m_i$ for all $i=0,1,\ldots,k$.
Note, however, that the latter (in contrast to (\ref{multiplecoupling})) does not determine $\mu$ uniquely.
\end{lemma}

\begin{proof}
The proof in the case $k=2$ is well-known, see e.g. \cite{Dud}, proof of Lemma 11.8.3, \cite{Vi03}, Lemma 7.6.
For convenience of the reader, let us briefly recall the construction: disintegration of $d\mu_1(x_0,x_1)$ w.r.t.\ $d\m_1(x_1)$ yields a Markov kernel $dp_{x_1}(x_0)$ such that
\[
	d\mu_1(x_0,x_1) = dp_{x_1}(x_0) d\m_1(x_1).
\]
Similarly, disintegration of $d\mu_2(x_1,x_2)$ w.r.t.\ $d\m_1(x_1)$ leads to a kernel $dq_{x_1}(x_2)$. In terms of these kernels the probability measure $\mu=\mu_1\boxtimes\mu_2$ on $X_0\times X_1 \times X_2$ is defined as
\[
	d\mu(x_0,x_1,x_2)= dp_{x_1}(x_0)dq_{x_1}(x_2)d\m_1(x_1).
\]

The solution for general $k$ is constructed iteratively.
Assume that $\mu^{(i)}:=\mu_1\boxtimes\ldots\boxtimes\mu_i$ is already constructed. By definition/construction it is a coupling of  $\mu^{(i-1)}$ and $\m_i$ whereas $\mu_{i+1}$ is a coupling of $\m_i$ and $\m_{i+1}$. The previous step thus allows to construct the gluing of $\mu^{(i)}$ and $\mu_{i+1}$ which is the desired
$\mu^{(i+1)}=\mu^{(i)}\boxtimes\mu_{i+1}$.
\end{proof}

\begin{lemma}
\label{sec-glue}
Let $X_0$ and $X_k$, $k\in\N$, be Polish spaces and $\m_0$ and $\m_k$, $k\in\N$, probability measures, defined on the respective $\sigma$-fields. Then for every choice of couplings
$\mu_k\in \Cpl(\m_0,\m_k)$, $k\in\N$, there exists a probability measure $\mu\in\prob\big(\prod_{k=0}^\infty X_k\big)$ s.t.
\begin{equation}
\label{countablecoupling}(\pi_{0},\pi_k)_{\push} \mu=\mu_k\qquad(\forall k\in\N).\end{equation}
\end{lemma}

\begin{proof}
 Let $\mu_k\in\Cpl(\m_0,\m_k)$ for $k\in\N$ be given and define for each $n\in\N$ a probability measure $\mu^{(n)}$ on $X=X_0\times X_1\times \ldots  X_n$ by
\[d\mu^{(n)}(x_0,x_1,x_2,\ldots,x_n)=d\mu_{1,x_0}(x_1)\,d\mu_{2,x_0}(x_2)\ldots d\mu_{n,x_0}(x_n)\,d\m_0(x_0)\]
where $d\mu_{k,x_0}(x_k)$ denotes the disintegration of $d\mu_k(x_0,x_k)$ w.r.t. $d\m_0(x_0)$.
The projective limit of these probability measures $\mu^{(n)}$ as $n\to\infty$ is the requested $\mu$.
\end{proof}

\subsection{The $L^p$-Distortion Distance}

\begin{definition}
	For any $p\in[1,\infty)$, the \emph{$L^p$-distortion distance}\index{distortion distance $\DD_p$} between two  metric measure spaces $(X_0,\d_0,\m_0)$ and $(X_1,\d_1,\m_1)$ is
	defined as
	\begin{eqnarray*}
	\lefteqn{\DD_p((X_0,\d_0,\m_0),(X_1,\d_1,\m_1))}\\
&=& \inf \Bigg\{  \bigg( \int_{X_0\times X_1} \int_{X_0\times X_1}
		\left| \d_0(x_0,y_0)-\d_1(x_1,y_1) \right|^p d\ol{\m}
		(x_0,x_1)d\ol{\m}(y_0,y_1)\bigg)^{1/p} \spec \ol{\m}\in \Cpl(\m_0,\m_1)
		\Bigg\}.
	\end{eqnarray*}
Similarly, the \emph{$L^\infty$-distortion distance} is defined as
\begin{eqnarray*}
		\lefteqn{\DD_{\infty}((X_0,\d_0,\m_0),(X_1,\d_1,\m_1))}\\
&=& \inf \Bigg\{\sup\bigg\{
\left| \d_0(x_0,y_0)-\d_1(x_1,y_1) \right| \spec (x_0,x_1), (y_0,y_1)\in \supp(\ol{\m})\bigg\}
\spec  \ol{\m}\in \Cpl(\m_0,\m_1)
		\Bigg\}.
	\end{eqnarray*}
\end{definition}

The $L^{p}$-distortion distance is the particular case of the more general $L^{p,q}$-distortion distance for the choice $q=1$.
Instead of overloading notations and proofs with additional technicalities we try to keep the presentation as simple as possible by first restricting to the most simple case $q=1$. The modifications for general $q\ge1$ will be summarized in Chapter 9.

\begin{lemma}\label{optcoupl1}
			For each $p\in [1,\infty]$ and each pair of metric measure spaces $(X_0,\d_0,\m_0)$ and $(X_1,\d_1,\m_1)$, the infimum in the
			definition of $\DD_p((X_0,\d_0,\m_0),(X_1,\d_1,\m_1))$ will be attained. That is, there exists a measure $\ol\m\in\Cpl(\m_0,\m_1)$ such
			that
\begin{equation}\label{optimal}
				\DD_p((X_0,\d_0,\m_0),(X_1,\d_1,\m_1))=  \bigg( \int_{X_0\times X_1} \int_{X_0\times X_1}
				\left| \d_0(x_0,y_0)-\d_1(x_1,y_1) \right|^p  d\ol{\m}(x_0,x_1) d\ol{\m}(y_0,y_1)
				\bigg)^{1/p}
		\end{equation}
in the case $p<\infty$ and
\begin{equation*}
				\DD_{\infty}((X_0,\d_0,\m_0),(X_1,\d_1,\m_1))=
\sup\bigg\{
\left| \d_0(x_0,y_0)-\d_1(x_1,y_1) \right| \spec (x_0,x_1), (y_0,y_1)\in \supp(\ol{\m})\bigg\}.
\end{equation*}		
		\end{lemma}

\begin{proof}
According to Lemma \ref{cpl-comp}, $\Cpl(\m_0,\m_1)$ is a non-empty
	compact subset of $\prob(X_0\times X_1)$.
	Moreover, for any $p\in[1,\infty)$ the function
	\[
		\dis_p(.): \ \m \mapsto \left(\int_{X_0\times X_1} \int_{X_0\times X_1} \left| \d_0(x_0,y_0)-\d_1(x_1,y_1) \right|^p
		d\m(x_0,x_1) d\m(y_0,y_1)\right)^{1/p}
	\]
	is lower semicontinuous on $\prob(X_0\times X_1)$ due to the continuity of $\d_0$ and $\d_1$. Passing to the limit $p\nearrow\infty$, this also yields the lower semicontinuity for the analogously defined function $\dis_{\infty}(.)$. Thus for any $p\in [1,\infty]$, the function
	$\dis_p(.)$ attains its minimum on $\Cpl(\m_0,\m_1)$.
\end{proof}

\begin{definition}\label{def-opt-cpl}
A coupling $\ol{\m}\in\Cpl(\m_0,\m_1)$ is called \emph{optimal} (for $\DD_p$) if (\ref{optimal}) is satisfied.
The set of optimal couplings of the mm-spaces
$(X_0,\d_0,\m_0)$ and $(X_1,\d_1,\m_1)$ will be denoted by $\Opt(\m_0,\m_1)$\index{o@$\Opt(.,.)$}\index{optimal coupling}.
\end{definition}
Note that -- despite this short hand notation -- the set $\Opt(\m_0,\m_1)$ strongly depends on the choice of the metrics $\d_0,\d_1$ and on the choice of $p$.

\begin{lemma}\label{dildis1}
For each $p\in [1,\infty]$ and each triple of metric measure spaces
$(X_0,\d_0,\m_0)$, $(X_1,\d_1,\m_1)$ and $(X_2,\d_2,\m_2)$,
\[\DD_p((X_0,\d_0,\m_0),(X_2,\d_2,\m_2))\le
\DD_p((X_0,\d_0,\m_0),(X_1,\d_1,\m_1))+\DD_p((X_1,\d_1,\m_1),(X_2,\d_2,\m_2))
.\]
\end{lemma}

\begin{proof}
Choose optimal couplings $\mu\in\Opt(\m_0,\m_1)$ and $\nu\in\Opt(\m_1\,\m_2)$ and glue them together
to obtain a probability measure $r=\mu\boxtimes\nu$ on $X_0\times X_1\times X_2$ with
$(\pi_0,\pi_2)_{\push}r\in \Cpl(\m_0,\m_2)$.
Thus in the case $p<\infty$
\begin{equation*}
\begin{split}
		\lefteqn{\DD_p((X_0,\d_0,\m_0),(X_2,\d_2,\m_2))}\\
		& \leq \bigg( \int \int
		\Big| \d_0(x_0,y_0)-\d_2(x_2,y_2) \Big|^p  dr(x_0,x_1,x_2) dr(y_0,y_1,y_2)
		\bigg)^{1/p} \\
		 & = \bigg( \int \int
		\left| \d_0(x_0,y_0)-\d_1(x_1,y_1)+\d_1(x_1,y_1)-\d_2(x_2,y_2) \right|^p
		dr(x_0,x_1,x_2) dr(y_0,y_1,y_2)	\bigg)^{1/p}\\
		& \leq  \bigg( \int \int
		\left| \d_0(x_0,y_0)-\d_1(x_1,y_1)\right|^p  dr(x_0,x_1,x_2)
		dr(y_0,y_1,y_2) \bigg)^{1/p}  \\
		&\qquad\qquad\qquad \qquad\phantom{ {} \leq {}} +  \bigg( \int \int
		\left| \d_1(x_1,y_1)-\d_2(x_2,y_2) \right|^p dr(x_0,x_1,x_2)
		dr(y_0,y_1,y_2)	\bigg)^{1/p} \\
		& =
\DD_p((X_0,\d_0,\m_0),(X_1,\d_1,\m_1))+\DD_p((X_1,\d_1,\m_1),(X_2,\d_2,\m_2)).
\end{split}
\end{equation*}
This is the claim. Here, the last inequality is a consequence of the triangle inequality for the $L^p$-norm.
Exactly the same arguments also prove the claim in the case $p=\infty$.
\end{proof}

\subsection{Isomorphism Classes of MM-Spaces}

\begin{lemma}\label{iso-null} For each $p\in [1,\infty]$ and each pair of metric measure spaces $(X_0,\d_0,\m_0)$ and $(X_1,\d_1,\m_1)$, the following assertions are equivalent:
\begin{enumerate}
\item $\DD_p((X_0,\d_0,\m_0),(X_1,\d_1,\m_1))=0$.
\item $\exists \ol{\m}\in\Cpl(\m_0,\m_1)$ such that $\d_0(x_0,y_0)=\d_1(x_1,y_1)$ for $\ol{\m}^2$-a.e. $(x_0,x_1,y_0,y_1)\in (X_0\times X_1)^2$.
\item There exist a metric measure space $(X,\d,\m)$ -- complete and separable, as usual -- with full support and Borel maps $\psi_0:X\to X_0$, $\psi_1:X\to X_1$ which push forward the measures and pull back the metrics:
    \begin{itemize}
    \item $(\psi_0)_{\push} \m=\m_0,\quad (\psi_1)_{\push} \m=\m_1$,
    \item $\d=(\psi_0)^\push\d_0=(\psi_1)^\push\d_1$ \ on $X\times X$.
    \end{itemize}
   \item There exists a Borel measurable bijection $\psi: X_0^\base\to X_1^\base$ with Borel measurable inverse $\psi^{-1}$ between the supports $X_0^\base=\supp(X_0,\d_0,\m_0)$ and $X_1^\base=\supp(X_1,\d_1,\m_1)$ such that
\begin{itemize}
    \item $\psi_{\push} \m_0=\m_1$,
    \item $\d_0=\psi^\push\d_1$ \ on $X_0^\base\times X_0^\base$.
    \end{itemize}
\end{enumerate}
\end{lemma}

\begin{proof}
Taking into account the existence of optimal couplings (Lemma \ref{optcoupl1}), the equivalence of (i) and (ii) is obvious.
For the implication $(ii)\Rightarrow(iii)$, one may choose  $\m=\ol{\m}$, restricted to its support
$X$ which is some closed subset of $X_0\times X_1$.
On $X$, a complete separable metric is given by
\[\d((x_0,x_1),(y_0,y_1))=\frac12 \d_0(x_0,y_0)+\frac12 \d_1(x_1,y_1).\]
Finally, one may choose $\psi_0$ and $\psi_1$ to be the projection maps $X\to X_0$ and $X\to X_1$, resp.
They are Borel measurable and push forward $\m$ to its marginals $\m_0$ and $\m_1$.
Moreover, $\d_i(\psi_i(x),\psi_i(y))=\d_i(x_i,y_i)$ for $i=0,1$ and thus, according to assumption (ii), for $\m^2$-a.e. $(x,y)=((x_0,x_1),(y_0,y_1))\in X^2$
\[\d_0(\psi_0(x),\psi_0(y))=\d_0(x_0,y_0)=\d_1(x_1,y_1)=
\d_1(\psi_1(x),\psi_1(y)).\]
However, $\d_0$ and $\d_1$ (more precisely, their pull backs via the projection maps) are continuous functions on $X^2$, and $\m$ has full support. Thus the previous identity holds without exceptional set on $X^2$.
This in turn implies -- according to our choice of $\d$ --  that
\[\d(x,y)=\d_0(\psi_0(x),\psi_0(y))=
\d_1(\psi_1(x),\psi_1(y))\]
 for all $x,y\in X$.

$(iii)\Rightarrow(iv)$: \
The maps $\psi_i: X\to X_i^\base$ for $i=0,1$ are isometric bijections with Borel measurable inverse. Indeed, since the maps $\psi_i$ pull back the metrics, they are injective and isometries. For showing surjectivity, note that any $y\in X_i^\base$ is the limit of a sequence $\{y^k=\psi_i(x^k)\}_{k\in\N}$ in the image of $\psi_i$ since $\psi$ pushes forward the measures. Then $\{x^k\}_{k\in\N}$ is a Cauchy sequence in $X$ and due to the completeness of $X$ it has a limit $x\in X$ whose image $\psi_i(x)$ coincides with $y$.
Now $\psi=\psi_1\circ \psi_0^{-1}: X_0^\base\to X_1^\base$ is the requested bijective Borel map with Borel measurable inverse.

$(iii)$ or $(iv)\Rightarrow(i)$ and $(ii)$: \ Choose $\ol{\m}=(\psi_0,\psi_1)_{\push} \m$ or $\ol{\m}=(\Id,\psi)_{\push} \m_0$.
\end{proof}

\begin{definition}
Two metric measure spaces $(X_0,\d_0,\m_0)$ and $(X_1,\d_1,\m_1)$ will be called \emph{isomorphic} if any (hence every) of the preceding assertions holds true.
This obviously defines an equivalence relation. The corresponding equivalence class will be denoted by $[X_0,\d_0,\m_0]$ and called
\emph{isomorphism class} of $(X_0,\d_0,\m_0)$.
The family of all isomorphism classes of metric measure spaces (with complete separable metric and normalized volume, as usual) will be denoted by $\XX_0$.
\end{definition}

In the sequel, elements of $\XX_0$ will be denoted by $\X$, $\X'$, $\X_0$, $\X_1$ etc.
Each of them is an equivalence class of isomorphic mm-spaces, say
\[ \X=[X,\d,\m], \quad \X'=[X',\d',\m'],\quad \X_0=[X_0,\d_0,\m_0],\quad \X_1=[X_1,\d_1,\m_1]. \]
\index{isomorphism class $[X,\d,\m]$}
 Representatives within these classes will be denoted as before by
$(X,\d,\m)$, $(X',\d',\m')$, $(X_0,\d_0,\m_0)$ or $(X_1,\d_1,\m_1)$, resp.
Note that in each equivalence class there is a space with full support. Indeed, any $(X,\d,\m)$ is
isomorphic to $(\supp(X,\d,\m),\d,\m)$.

All relevant properties of mm-spaces considered in the sequel will be properties of the corresponding isomorphism classes. (This also holds true for the quantities $\diam(.),\, \size_p(.),\, \DD_p(.,.)$ defined so far.) Thus, mostly, there is no need to distinguish between equivalence classes and representatives of these classes and we simply call $\XX_0$
 \emph{the space of metric measure spaces}.
For any $p\in [1,\infty]$, the subspace of mm-spaces with finite $L^p$-size will be denoted by
\[\XX_p=\{\X\in \XX_0 \spec \size_p(\X)<\infty\}.\]
\index{space of mm-spaces}
\index{x@$\XX_p$}

\begin{proposition}\label{d-is-metric}
For each $p\in[1,\infty]$,
$\DD_p$ is a  metric on $\XX_p$.
\end{proposition}

\begin{proof}
Symmetry, finiteness and nonnegativity are obvious. By construction (see  Lemma \ref{iso-null}), $\DD_p$ vanishes only on the diagonal of $\XX_p\times\XX_p$. The triangle inequality was derived in Lemma \ref{dildis1}.
\end{proof}

\begin{remark}
For each $p\in[1,\infty)$, the metric space $\big(\XX_p,\DD_p\big)$ will be \emph{separable} but \emph{not complete}.

The separability will follow from an analogous statement for  $(\XX_p,\D_p)$, see Proposition \ref{d-prop}, combined with the estimate
$\DD_p\le 2\D_p$ from Proposition \ref{lp-d-dd} below. Incompleteness will be proven in Corollary \ref{incomplete}.
\end{remark}


\begin{remark}
	The $L^p$-distortion distance can also be interpreted in terms of classical optimal transportation
	with some additional constraint.
	Given $p\in [1,\infty)$ and metric measure spaces $(X_0,\d_0,\m_0),(X_1,\d_1,\m_1)$, put
	$Y_i:= X_i \times X_i,\,\mu_i=\m_i\otimes\m_i$ for $i=0,1$ and
	\[
		c(y_0,y_1)=\left| a(y_0) - b(y_1) \right|^p
	\]
	with $a(y_0)=\d_0(x_0,x'_0),\, b(y_1)=\d_1(x_1,x'_1)$ for $y_0=(x_0,x'_0)\in Y_0,\,y_1=(x_1,x'_1)\in
	Y_1$. Then
	\[
		\DD_p(\X_0,\X_1)^p= \inf \left\{ \int_{Y_0\times Y_1} c(y_0,y_1) d\mu(y_0,y_1) \spec
		\mu\in \Cpl_\square(\mu_0,\mu_1)\right\},
	\]
	where
	\begin{eqnarray*}
			\Cpl_\square(\mu_0,\mu_1)&=&\Big\{  \mu\in \prob(Y_0\times Y_1) \text{ s.t.\ }
		d\mu(x_0,x'_0,x_1,x'_1)=d\m(x_0,x_1)d\m(x'_0,x'_1)\\
&&\qquad \qquad\qquad\qquad\qquad\qquad\qquad\qquad\text{ for some } \m\in \Cpl(\m_0,\m_1) \Big\}\\
		&\subset& \Cpl(\mu_0,\mu_1).
		\end{eqnarray*}
\end{remark}

An alternative approach to (optimal) couplings and to the $L^p$-distortion distance is based on the fact that every mm-space is a \emph{standard Borel space} or \emph{Lebesgue-Rohklin space} since by definition all (mm-) spaces under consideration are Polish spaces. Thus all of them can be represented as images of the \emph{unit interval} $I=[0,1]$\index{i@$I$} equipped with $\Leb^1$\index{l@$\Leb^1$}, the 1-dimensional Lebesgue measure  restricted to $I$. This leads to a variety of quite impressive representation results.
A drawback of these formulas, however, is that quite often any geometric interpretation gets lost.

\begin{lemma}\label{standard borel}
\begin{enumerate}
\item For every mm-space $(X,\d,\m)$ there exists a Borel map $\psi:I\to X$ such that \[m=\psi_{\push} \Leb^1.\]
Any such map $\psi$ will be called \emph{parametrization} of the mm-space $(X,\d,\m)$.
The set of all parametrizations  will be denoted by $\Par(X,\d,\m)$\index{p@$\Par(.)$} or occasionally briefly by $\Par(\m)$.
\item
Given mm-spaces $(X_0,\d_0,\m_0)$ and $(X_1,\d_1,\m_1)$, a probability measure $\ol\m$ on $X_0\times X_1$ is a coupling of $\m_0$ and $\m_1$ if and only if there exist $\psi_0\in\Par(X_0,\d_0,\m_0)$ and $\psi_1\in \Par(X_1,\d_1,\m_1)$ with
\[\ol\m=(\psi_0,\psi_1)_{\push} \Leb^1.\]
\item
For any $p\in [1,\infty)$ and any $\X_0=[X_0,\d_0,\m_0]$ and $\X_1=[X_1,\d_1,\m_1]$
\begin{eqnarray*}\DD_p(\X_0,\X_1)&=& \inf \Bigg\{  \bigg( \int_0^1 \int_0^1
		\left| \d_0(\psi_0(s),\psi_0(t))-\d_1(\psi_1(s),\psi_1(t)) \right|^p ds\,dt\bigg)^{1/p} \spec \\
&&\qquad\qquad\qquad\qquad\qquad\psi_0\in\Par(X_0,\d_0,\m_0), \ \psi_1\in \Par(X_1,\d_1,\m_1) \,
		\Bigg\}.
\end{eqnarray*}
\end{enumerate}
\end{lemma}

\begin{proof} (i) is well-known, see e.g. \cite{Sri}, Theorem 3.4.23.

(ii) Let parametrizations $\psi_0,\psi_1$ of $\m_0,\m_1$, resp. be given.
If $\ol\m=(\psi_0,\psi_1)_{\push} \Leb^1$ then $(\pi_i)_*\ol\m=(\psi_i)_*\Leb^1=\m_i$ for each $i=0,1$. Thus $\ol\m\in\Cpl(\m_0,\m_1)$.
Conversely, according to part (i) for every $\ol\m\in\Cpl(\m_0,\m_1)$ there exists a Borel map $\psi:I\to X_0\times X_1$ such that $\ol\m=\psi_*\Leb^1$.
Put
$\psi_i=\pi_i\circ \psi$ such that $\psi=(\psi_0,\psi_1)$. Then
$(\psi_i)_*\Leb^1=(\pi_i)_*\ol\m=\m_i$ for each $i=0,1$.

(iii) is an obvious consequence of (ii).
\end{proof}

\begin{remarks}\label{rem:paramet}
\begin{enumerate}
\item\label{rem:paramet1} Given an mm-space $(X,\d,\m)$ \emph{without atoms} (i.e. with $\m(\{x\})=0$ for each $x\in X$), a Borel measurable map  $\psi:I\to X$ with $\m=\psi_{\push} \Leb^1$ can be chosen in such a way that it is bijective with Borel measurable inverse $\psi^{-1}: X\to I$.
    \item For a general mm-space $(X,\d,\m)$, the measure $\m$ can be decomposed into a countable (infinite or finite) weighted sum of atoms and a measure without atoms. That is,
        \[ \m=\sum_{i=1}^\infty\alpha_i\,\delta_{x_i}+\m'\]
        for suitable $x_i\in X$, $\alpha_i\in[0,1]$. Put
        $\ol\alpha_i=\sum_{j=1}^i\alpha_j$ for $i\in\N\cup\{\infty\}$, $I'=\big[\ol\alpha_\infty,1\big)$ and $X'=\supp(\m')$.
        Then there exists a Borel measurable map $\psi:I\to X$ such that $\m=\psi_{\push} \Leb^1$,
        \[\psi: \big[\ol\alpha_{i-1}, \ol\alpha_i\big) \to \{x_i\}\]
        for each $i\in\N$,
        and $\psi|_{I'}: I'\to X'$ is bijective with Borel measurable inverse (see Figure).
\begin{figure}[h!]
\begin{center}
\psfrag{X}{$X$}
\psfrag{I}{$I$}
\psfrag{0}{$0$}
\psfrag{1}{$1$}
\psfrag{alpha1}{\scriptsize\color{blue}$\alpha_1$}
\psfrag{alpha2}{\scriptsize\color{green}$\alpha_2$}
\psfrag{phi}{$\psi$}

\includegraphics[scale=0.5]{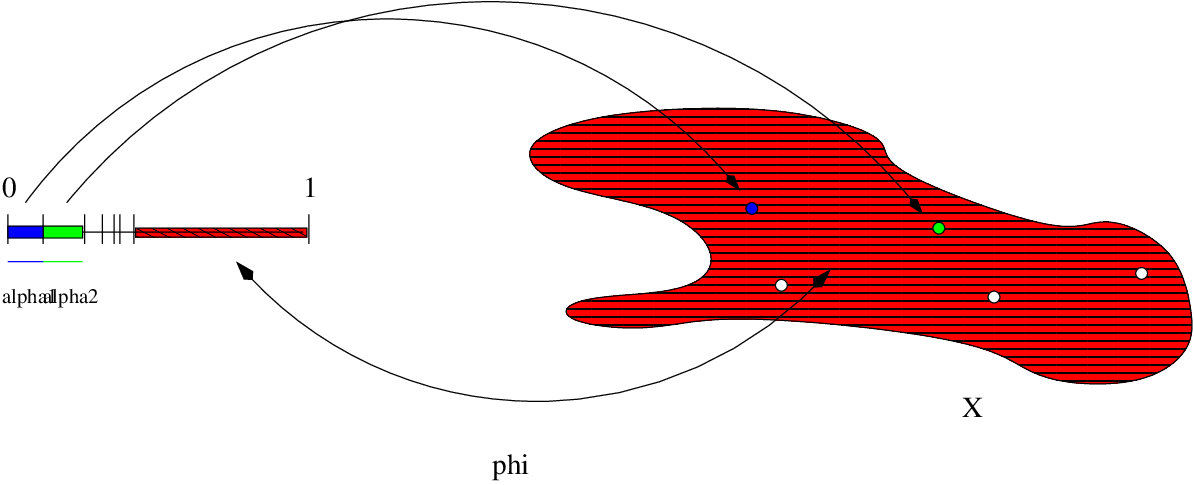}

\caption{\wichtig{Borel isomorphism $\psi$}}
\end{center}
\end{figure} 
\item
Typically, the triple $(I,\psi^\push\d, \Leb^1)$ will not be a mm-space in the sense of the previous section but just a \emph{pseudo} metric measure space in the sense of
chapter \ref{pseudo} below.
For every $\psi\in \Par(X,\d,\m)$, it will be \emph{homomorphic} to the mm-space $(X,\d,\m)$ (see Definition \ref{def-homo} below).
\end{enumerate}
\end{remarks}
For another canonical representation of elements $\X\in\XX$ in terms of matrix distributions, see Proposition
\ref{reconstruct}.

\section{The Topology of $(\XX_p, \DD_p)$}

\subsection{$L^p$-Distortion Distance  vs. $L^0$-Distortion Distance}
In order to characterize the topology on $\XX_p$ induced by $\DD_p$ observe
that it is essentially an $L^p$-distance and recall that $L^p$-convergence for functions is equivalent to
convergence in probability and convergence of the $p$-th moments (or uniform $p$-integrability). Following \cite{Dud}, convergence in probability is the appropriate concept of `$L^0$-convergence'. It is metrized among others by the Ky Fan-metric. Adopting this concept to our setting leads to the following definition of the  $L^0$-distortion distance $\DD_0$:
\[
	\DD_{0}(\X_0,\X_1)= \inf \bigg\{ \epsilon >0\spec
{\ol\m}\otimes{\ol\m} \bigg( \Big\{(x_0,x_1,y_0,y_1)\spec |{\d_0}(x_0,y_0)-\d_1(x_1,y_1)|>\epsilon\Big\}
\bigg)
\le\epsilon, \ \ol\m\in\Cpl(\m_0,\m_1)
\bigg\}.
\]

\begin{proposition}
For each $p\in[1,\infty)$, every point $\X_\infty$ and every sequence $(\X_n)_{n\in\N}$ in $\XX_p$ the following statements are equivalent:
\begin{enumerate}
\item
$\DD_p(\X_n,\X_\infty)\to0$ as $n\to\infty$;
\item
$\DD_0(\X_n,\X_\infty)\to0$ as $n\to\infty$ and
\[\size_p(\X_n)\to \size_p(\X_\infty)\quad\mbox{as }n\to\infty;\]
\item
$\DD_0(\X_n,\X_\infty)\to0$ as $n\to\infty$ and
\begin{equation}
\label{unifinte}\sup_{n\in\N}{\int\int}_{\{\d_n(x,y)>L\}}\d_n(x,y)^pd\m_n(x)\,d\m_n(y)\to0\quad\mbox{as }L\to\infty.
\end{equation}
\end{enumerate}
\end{proposition}

Note that condition (\ref{unifinte}) is void for each sequence $(\X_n)_{n\in\N}$ with uniformly bounded diameter. Such a sequence converges w.r.t. $\DD_p$ (for some, hence all $p\in[1,\infty)$) if and only if it converges w.r.t. $\DD_0$.

\begin{proof}
Given the sequence $(\X_n)_{n\in\N}$ in $\XX_p$, the point $\X_\infty$ as well as optimal couplings $\ol\m_n$ of them, we can model all the distances $\d_n,\d_\infty$ as (suitably coupled) random variables on one probability
space. That is, there exists a probability space $(\Omega,\frak A, \PP)$ and random variables $\xi_n:\Omega\to\R$ for $n\in\N\cup\{\infty\}$ s.t.
\[\big(\xi_n,\xi_\infty\big)_{\push}\PP=\big(\d_n,\d_\infty\big)_{\push}(\ol\m_n\otimes\ol\m_n)\quad(\forall n\in\N),\]
see Lemma \ref{sec-glue}.
Then indeed $\DD_p(\X_n,\X_\infty)$ is the $L^p$-distance of the random variables $\xi_n, \xi_\infty$, and
 $\DD_0(\X_n,\X_\infty)$ is the Ky Fan-distance of them:
 \[\DD_p(\X_n,\X_\infty)=\bigg( \int_\Omega |\xi_n-\xi_\infty|^pd\PP\bigg)^{1/p},\]
 \[\DD_0(\X_n,\X_\infty)=\inf\Big\{\epsilon>0\spec \PP(\{ |\xi_n-\xi_\infty|>\epsilon\})\le\epsilon\Big\}.\]
Moreover, the $L^p$-size of $\X_n$ is just the $p$-th moment of $\xi_n$.
Hence, the claim of the Theorem is an immediate consequence of the  well-known and fundamental result from Lebesgue's integration theory: \emph{The following statements are equivalent:
\begin{itemize}
\item $\xi_n\to\xi_\infty$ in $L^p$;
\item $\xi_n\to\xi_\infty$ in probability and $\int|\xi_n|^pd\PP\to \int|\xi_\infty|^pd\PP$;
\item $\xi_n\to\xi_\infty$ in probability and $(\xi_n)_{n\in\N}$ is uniformly $p$-integrable.
\end{itemize}}
See e.g. \cite{Bau01}, Theorem~21.7.
\end{proof}

\begin{example}
\label{ex-incompl}
For each $n\in\N$, let $\X_n=[X_n,\d_n,\m_n]$ be the complete graph with $2^n$ vertices,
unit distances and uniform distribution, a representative of $\X_n$ is e.g. given by
$X_n=\{1,\ldots,2^n\}$, $\d_n(i,j)=1$ for all $i\not=j$ and
$\m_n=\frac{1}{2^n}\sum_{i=1}^{2^n}\delta_{i}$.
\begin{figure}[h!]
\begin{center}

\psfrag{mbar}{\color{blue}$\ol{\m}$}
\psfrag{m4}{\color{red}$\m_4$}
\psfrag{m2}{$\m_2$}

\includegraphics[scale=0.2]{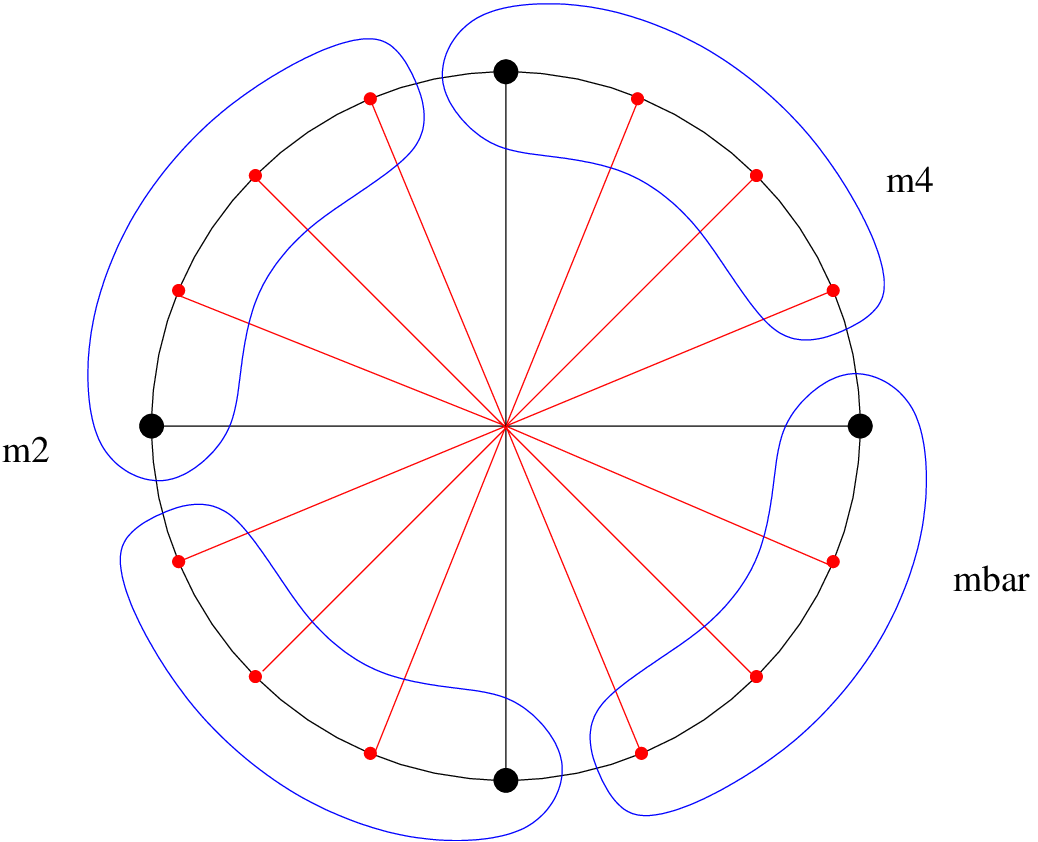}

\end{center}
\end{figure} 
Then $(\X_n)_{n\in\N}$ is a Cauchy sequence w.r.t. $\DD_p$ for each $p\in\{0\}\cup[1,\infty)$.
More precisely, for any $p\in[1,\infty)$,
\[\DD_p(\X_n,\X_k)^p=\DD_0(\X_n,\X_k)\le|2^{-n}-2^{-k}|\quad\text{ for all }k,n\in\N.\]
However, the sequence will not converge in $\XX$, see Lemma \ref{discrete-sphere}.
\end{example}

\begin{proof} Since the distortion function $\dis(i_n,j_n,i_k,j_k)=|\d_n(i_n,j_n)-\d_k(i_k,j_k)|$ can attain only the values $0$ and $1$, for \emph{each} coupling $\ol\m\in\Cpl(\m_n,\m_k)$, independent of $p$ and $\epsilon$,
\begin{eqnarray*}
\int\int\Big|\d_n-\d_k\Big|^pd\ol\m\,d\ol\m&=&\ol\m^2\Big(\dis>\epsilon\Big)=\ol\m^2\Big(\dis\not=0\Big)\\
&=&
\sum_{i_n,i_k}\ol\m(i_n,i_k)\Big[\sum_{j_k\not=i_k}\ol\m(i_n,j_k)+\sum_{j_n\not=i_n}\ol\m(j_n,i_k)\Big|.
\end{eqnarray*}
Assume now that $k>n$. Then  the choice
\[\ol\m=\frac1{2^n}\sum_{i_n=1}^{2^n}\Big(\frac1{2^{k-n}}\sum_{j_k=1}^{2^{k-n}}\delta_{i_n,(i_n-1)2^{k-n}+j_k}\Big)\]
leads to the upper estimate
$\DD_p^p=\DD_0\le\frac1{2^n}\frac1{2^{k-n}}\Big(2^{k-n}-1\Big)$.
\end{proof}
\subsection{$L^p$-Distortion Distance vs. $L^p$-Transportation Distance}

The $L^p$-distortion distance is closely related to the $L^p$-transportation distance $\D_p$ introduced earlier by the author \cite{St06}. 
The definition of the latter requires to introduce some further concepts.

\begin{wrapfigure}{r}{0.4\textwidth}
\begin{center}
\psfrag{X0}{\scriptsize $X_0$}
\psfrag{X1}{\scriptsize $X_1$}
\psfrag{d0}{\scriptsize $d_0$}
\psfrag{d1}{\scriptsize $d_1$}
\psfrag{dbar}{\scriptsize $\ol{d}$}

\includegraphics[scale=0.3]{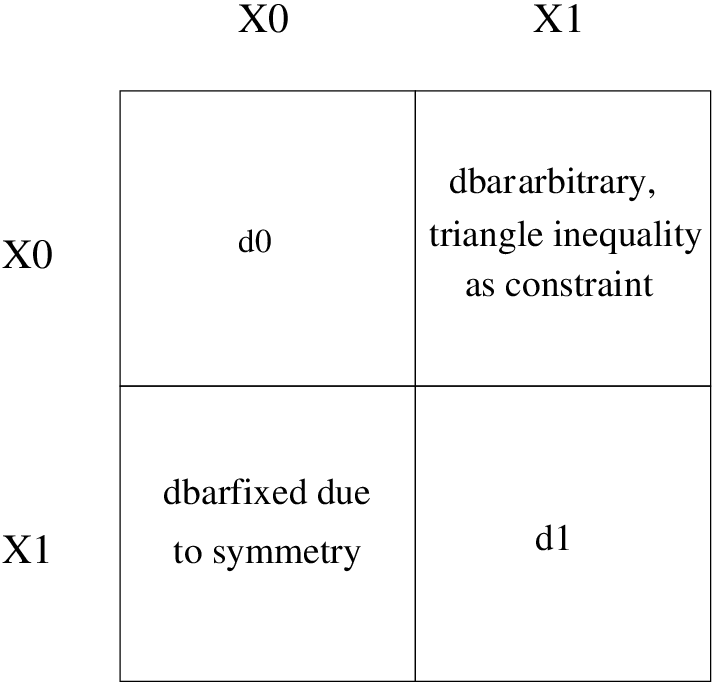}

\end{center}
\end{wrapfigure} 
Given metric spaces $(X_0,\d_0)$ and $(X_1,\d_1)$,
a symmetric $\R_+$-valued function $\ol\d$ on $X\times X$ -- where
$X=	X_0 \sqcup X_1$ denotes the disjoint union of these spaces (with induced topology) --	 will be called \emph{coupling} of the  metrics $\d_0$ and $\d_1$ if
\begin{itemize}
\item it
 satisfies the triangle inequality on $X\times X$
\item it coincides with $\d_0$ on $X_0\times X_0$
\item  it coincides with $\d_1$ on $X_1 \times X_1$.
\end{itemize}
Note that this implies that $\ol{\d}$ is continuous on $X\times X$ since
\[
	|\ol{\d}(x_0,x_1)-\ol{\d}(y_0,y_1)| \leq \d_0(x_0,y_0) + \d_1(x_1,y_1)
\]
but it might  vanish outside the diagonal. Thus, $\ol\d$ is a \emph{pseudo metric} on $X$.

Given metric measure spaces $(X_0,\d_0,\m_0)$ and $(X_1,\d_1,\m_1)$, the set
 $\Cpl(\d_0,\d_1)$ will denote the   set of all \emph{couplings of the metrics restricted to the supports}, that is, couplings of the metric spaces $(X_0^\base,\d_0)$ and $(X_1^\base,\d_1)$ where $X_0^\base$ and $X_1^\base$ denote the {support} of the measures $\m_0$ and $\m_1$, resp.

\bigskip

The \emph{$L^p$-transportation distance}\index{transportation distance $\D_p$} between $\X_0$ and $\X_1$ is defined as
\begin{equation*}
	\D_p(\X_0,\X_1)= \inf \Bigg\{  \bigg( \int_{X_0\times X_1} \ol{\d}^p(x_0,x_1) d\ol{\m}
	(x_0,x_1)\bigg)^{1/p} \spec
\ \ol\m\in\Cpl(\m_0,\m_1), \ \ol{\d}\in\Cpl(\d_0,\d_1)
	\Bigg\}.
\end{equation*}
The usual limiting argument leads to consistent definitions for $p=\infty$:
\begin{align*}
	\D_{\infty}(\X_0,\X_1)= \inf \Bigg\{   \sup \Big\{ \ol{\d}(x_0,x_1)\spec
	(x_0,x_1)\in \supp(\ol{\m})\Big\} \spec
 \ \ol\m\in\Cpl(\m_0,\m_1), \ \ol{\d}\in\Cpl(\d_0,\d_1)
	\Bigg\}.
\end{align*}
One easily verifies that the distances $\D_p(\X_0,\X_1)$ only depend on the isomorphism classes of $\X_0$ and $\X_1$, resp.\ (and not on the choice of the representatives within these equivalence classes).
Obviously, all of them can be estimated in terms of the Gromov-Hausdorff distance between the supports of the measures
\[\D_p(\X_0,\X_1)\le d_{GH}\big(\supp(X_0),\supp(X_1)\big).\]
\begin{remark}
Taking into account that each isometric embedding leads to a coupling of the metrics $\d_0,\d_1$ and vice versa,
each coupling $\ol\d$ defines an isometric embedding into $(X_0^\base\bigsqcup X_1^\base,\ol\d)$, one easily
verifies that
\begin{equation*}
\begin{split}
	\D_p(\X_0,\X_1)= \inf \bigg\{
  &{\hat W}_p\Big(\hat\m_0,\hat\m_1\Big) \spec
\Big(\hat X,\hat\d\Big)\text{ cpl. sep. metric space},\\
& \qquad \imath_0: X_0^\base\to\hat X, \ \imath_1: X_1^\base\to\hat X
\text{ isometric embeddings},\
\hat\m_0={\imath_0}_{\push}\m_0,\ \hat\m_1={\imath_1}_{\push}\m_1
	\bigg\}
\end{split}
\end{equation*}
where ${\hat W}_p(.,.)$ denotes the $L^p$-Wasserstein distance on the space of probability measures on $(\hat X,\hat\d)$. Moreover, in view of Lemma \ref{standard borel} we conclude
\begin{equation*}
\begin{split}
	\D_p(\X_0,\X_1)= \inf \bigg\{
  &\bigg(\int_0^1 {\hat\d}^p\Big(\imath_0\big(\psi_0(t)\big), \imath_1\big(\psi_1(t)\big)\Big)\,dt\bigg)^{1/p}
   \spec \psi_0\in\Par(\m_0),\psi_1\in\Par(\m_1),\\
& \qquad \Big(\hat X,\hat\d\Big)\text{ cpl. sep. metric space},\
\imath_0: X_0^\base\to\hat X, \ \imath_1: X_1^\base\to\hat X
\text{ isometric embeddings}
	\bigg\}.
\end{split}
\end{equation*}
\end{remark}
The infimum in the above definition is always attained.

\begin{proposition}\label{d-prop}
Assume $p\in[1,\infty)$.
\begin{enumerate}
\item
For each pair $\big(\X_0, \X_1\big)$ of  metric measure spaces there exists an `optimal' pair $\big(\ol{\m},\ol{\d}\big)$
of couplings such that
\[
	\D_p(\X_0,\X_1)= \left( \int_{X_0\times X_1} \ol{\d}^p(x_0,x_1) d \ol{\m}(x_0,x_1)\right)^{1/p}.
\]	
\item
	$\D_p$ is a complete separable geodesic metric on $\XX_p$.
\end{enumerate}
\end{proposition}
\begin{proof}
In the case $p=2$, all the assertions are proven in \cite{St06}, Lemma~3.3 and Theorem 3.6. Their proofs, however, apply without any change to general $p\in[1,\infty)$.
\end{proof}
The corresponding $L^0$-transportation distance $\D_0$ is defined -- in the spirit of the Ky Fan metric -- by
\[
	\D_{0}(\X_0,\X_1)= \inf \bigg\{ \epsilon >0\spec
\ol\m \bigg( \Big\{(x_0,x_1)\spec \ol{\d}(x_0,x_1)>\epsilon\Big\}
\bigg)
\le\epsilon, \ \ol\m\in\Cpl(\m_0,\m_1), \ \ol{\d}\in\Cpl(\d_0,\d_1)
\bigg\}.
\]

\begin{remark}
Albeit the $L^p$-transportation distance and the $L^p$-distortion distance are closely related, they measure quite different quantities. Both definitions rely on the choice of an optimal coupling $\ol\m$ which produces
pairs $(x_0,x_1),(y_0,y_1),\ldots$ of matched points.
\begin{itemize}
\item
Each  such pair produces certain transportation cost, say $\ol\d(x_0,x_1)$. The $L^p$ mean of it yields the  $L^p$-transportation distance. It is the $L^p$-Wasserstein distance of the measures in an -- optimally chosen -- ambient metric space. The relevant question here is how far the two spaces (or the two measures) are from each other after they are brought into optimal position (i.e. after choosing the best isometric embedding of the two spaces into some common spaces.)
\item
For the $L^p$-distortion distance the relevant question is how much the distance between any pair of points in one of the two spaces, say $(x_0,y_0)\in X_0^2$, is changed if one passes to the pair of matched points in the other space, say $(x_1,y_1)\in X_1^2$. This is the \emph{distortion} of the distance. This quantity is independent of any embedding. Its $L^p$-mean defines the $L^p$-distortion distance.
\end{itemize}
\end{remark}

\begin{figure}[h!]
\begin{center}
\psfrag{X0}{$X_0$}
\psfrag{X1}{$X_1$}
\psfrag{dbar}{\color{green}$\ol{\d}$}
\psfrag{d0}{\color{red}$\d_0$}
\psfrag{d1}{\color{red}$\d_1$}
\psfrag{x0}{$x_0$}
\psfrag{y0}{$y_0$}
\psfrag{z0}{$z_0$}
\psfrag{x1}{$x_1$}
\psfrag{y1}{$y_1$}
\psfrag{z11}{$z_1$}
\psfrag{z12}{$z_1'$}
\psfrag{z13}{$z_1''=y_1'$}
\includegraphics[scale=0.3]{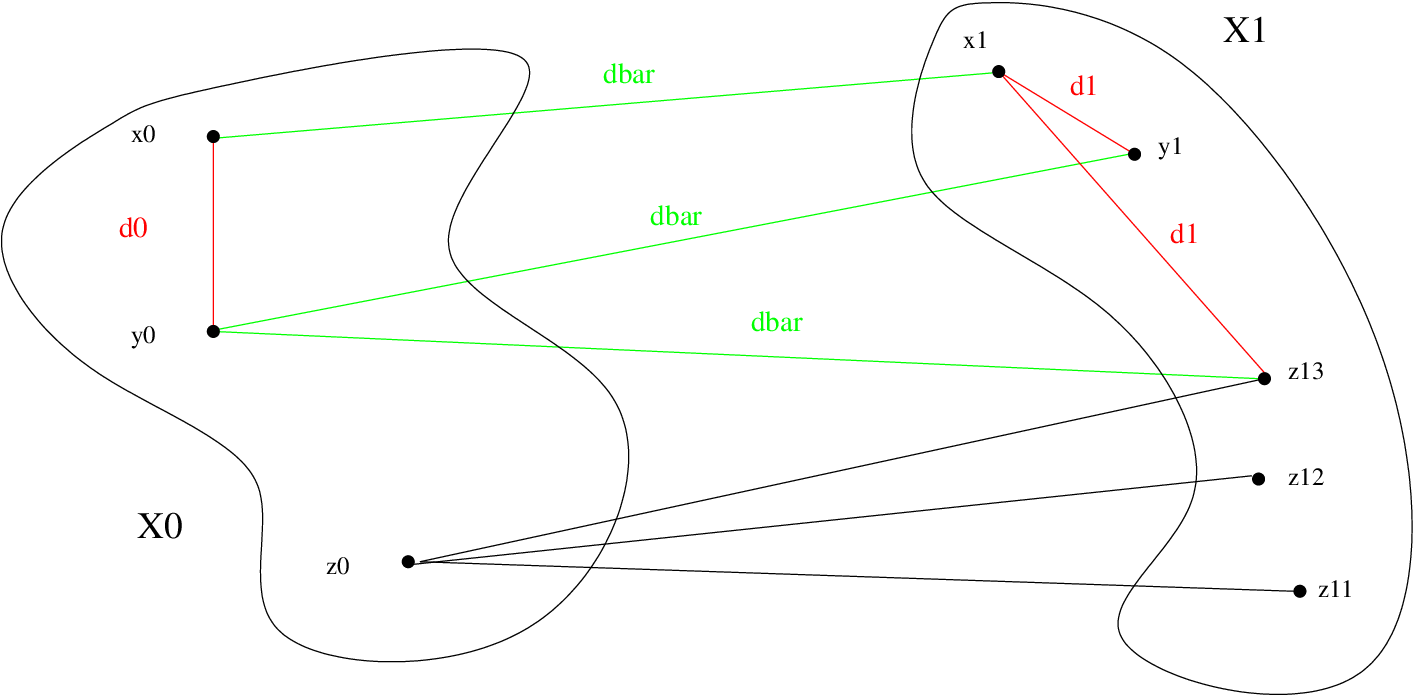}

\caption{\wichtig{
$\D_p$ = $L^p$-mean of $\ol{\d}$,\quad
$\DD_p$ = $L^p$-mean of $|\d_0-\d_1|$
}}

\end{center}
\end{figure} 

Let us summarize some of the elementary estimates for the metrics $\DD_p$ and $\D_p$ for varying $p$'s.

\begin{proposition}\label{lp-d-dd}
	\begin{enumerate}
		\item\label{lptransp1}
			$\forall p\in[1,\infty]: \quad \DD_p\leq 2 \D_p,\quad \DD_0\leq 2 \D_0$\quad and \quad $\DD_{\infty}=2
			\D_{\infty}$.
		\item\label{lptransp2}
			$\forall 1\leq p\leq q\leq\infty$: \qquad $\DD_0^{1+1/p}\le \DD_p\leq \DD_{q}, \qquad \D_0^{1+1/p}\le\D_p\leq
			\D_{q}$.
		\item\label{lptransp3}
			$\forall 1\leq p\leq q <\infty$, restricted to the space $\{\X\in \XX \spec \diam(\X)\leq L\}$ for a given $L\in \R_+$:
			\[
				L^{p-q} \cdot\DD_{q}^q\leq  \DD_p^p \leq \Big(1+L^p\Big)\cdot\DD_0, \qquad
			(L/2)^{p-q}\cdot	\D_{q}^q\leq   \D_p^p\leq\Big(1+(L/2)^p\Big)\cdot\D_0.
			\]
	\end{enumerate}
\end{proposition}

\begin{proof}
(\ref{lptransp1}) \
Let mm-spaces $\X_0$ and $\X_1$ be given. Without restriction, assume that the respective measures have full support and put $X=X_0\times X_1$.
If $\ol\d$ is a coupling of $\d_0$ and $\d_1$ then the function
$\dis: (x_0,x_1,y_0,y_1)\mapsto |\d_0(x_0,y_0)-\d_1(x_1,y_1)|$ defined on $X\times X$ satisfies
$ \dis(x,y)\le \ol\d(x)+\ol\d(y)$ and thus for each $\epsilon>0$
\[\{(x,y)\spec \dis>\epsilon\}\subset
\left(\{x\spec \ol\d(x)>\epsilon/2\}\times X \right)\cup \left(X \times\{y\spec \ol\d(y)>\epsilon/2\}\right).\]
This, in particular, implies for any $\ol{\m}\in \prob(X)$
\[\ol\m^2( \dis>\epsilon)\le 2 \ol\m(\ol\d>\epsilon/2).\]
If we now assume in the case $p=0$ that $\D_0(\X_0,\X_1)<\epsilon/2$ then the right hand side of the previous inequality will be less than $\epsilon$ which in turn proves that $\DD_0(\X_0,\X_1)<\epsilon$. This proves the claim for $p=0$.

For  $p\in[1,\infty)$, choosing the pair $\big(\ol\m,\ol\d\big)$ of couplings optimal for $\D_p$ , the claim follows from
	\begin{equation*}
	\begin{split}
		\DD_p(\X_0,\X_1)\leq {} & \bigg(\int_X \int_X
		\left| \d_0(x_0,y_0)-\d_1(x_1,y_1) \right|^p  d\ol{\m}(x_0,x_1) d\ol{\m}(y_0,y_1)\bigg)^{1/p} \\
		\leq {} & \bigg(\int_X \int_X
		\left| \ol{\d}(x_0,x_1)+\ol{\d}(y_0,y_1) \right|^p  d\ol{\m}(x_0,x_1) d\ol{\m}(y_0,y_1)\bigg)^{1/p}\\
		\leq {} & 2 \bigg(\int_X \ol{\d}(x_0,x_1)^p d\ol{\m}(x_0,x_1)\bigg)^{1/p}=\ 2\D_p(\X_0,\X_1).
	\end{split}
	\end{equation*}
Passing to the limit $p\nearrow\infty$ yields the upper estimate in the case $p=\infty$.

For the lower estimate, assume that $\DD_\infty(\X_0,\X_1)=L$ and that $\ol\m$ is an optimal coupling w.r.t.\ $\DD_\infty$.
Then $\dis(x,y)\leq L$ for $\ol\m$-a.e. $x,y\in X$. Continuity of $\dis$ implies that this holds for all $x,y\in\supp(\ol\m)$.
Therefore, a coupling $\ol\d$ of $\d_0$ and $\d_1$ can be defined by putting
\[\ol\d(x_0,x_1)=\inf\bigg\{ \d_0(x_0,y_0)+L/2+\d_1(y_1,x_1)\spec (y_0,y_1)\in\supp(\ol\m)\bigg\}\]
for arbitrary $x_0\in X_0$ and $x_1\in X_1$.
For this coupling, obviously $\ol\d(x_0,x_1)\le L/2$ for all $(x_0,x_1)\in\supp(\ol\m)$.
Thus $\D_\infty\le L/2$.

(\ref{lptransp2}) \ Simple applications of Jensen's inequality yield for each coupling as above and for all $1\le p\le q\le\infty$
 \begin{equation*}
	\begin{split}
		& \left(\int_X \int_X
		\left| \d_0(x_0,y_0)-\d_1(x_1,y_1) \right|^p  d\ol{\m}(x_0,x_1) d\ol{\m}(y_0,y_1)\right)^{1/p} \\
&\le \left(\int_X \int_X
		\left| \d_0(x_0,y_0)-\d_1(x_1,y_1) \right|^q  d\ol{\m}(x_0,x_1) d\ol{\m}(y_0,y_1)\right)^{1/q} \\
\end{split}
	\end{equation*}
as well as
\begin{equation*}
\left(\int_X \ol{\d}(x_0,x_1)^p d\ol{\m}(x_0,x_1)\right)^{1/p}\le
\left(\int_X \ol{\d}(x_0,x_1)^q d\ol{\m}(x_0,x_1)\right)^{1/q}.
\end{equation*}

For the $L^0$-$L^p$-estimates, recall that Markov's inequality states that
$\epsilon^{p}\cdot \PP(|\xi|>\epsilon)\le \int|\xi|^pd\PP$ for each  random variable $\xi$ and each $\epsilon>0$.
Thus,
$$\epsilon^{p+1}\le\int|\xi|^pd\PP$$
for all $\epsilon>0$ satisfying $\PP(|\xi|>\epsilon)>\epsilon$. Moreover, note that
$\inf\Big\{\epsilon>0\spec \PP(|\xi|>\epsilon)\le\epsilon\Big\}=\sup\Big\{\epsilon>0\spec \PP(|\xi|>\epsilon)>\epsilon\Big\}$, where we define $\sup \emptyset:= 0$.
Applying this to $\xi=\dis(.)$  and to $\xi=\ol\d$, resp., yields the stated $L^0$-$L^p$-estimates.

(\ref{lptransp3})\ To prove the $L^q$-$L^p$-estimate, let $\ol{\m}$ be an optimal coupling for $\DD_p$. Then,
	\begin{equation*}
		\begin{split}
			\DD_{q}(\X_0,\X_1)^q & \leq  \int_X \int_X
										\left| \d_0(x_0,y_0)-\d_1(x_1,y_1) \right|^q
										d\ol{\m}(x_0,x_1) d\ol{\m}(y_0,y_1)\\
									& \leq L^{q-p} \cdot  \int_X \int_X
											\left| \d_0(x_0,y_0)-\d_1(x_1,y_1) \right|^p
											d\ol{\m}(x_0,x_1) d\ol{\m}(y_0,y_1)= L^{q-p}\cdot \DD_p(\X_0,\X_1)^p,
		\end{split}	
	\end{equation*}
	since $| \d_0(x_0,y_0)-\d_1(x_1,y_1)|\leq L$ for all $x_0,y_0,x_1,y_1$ under consideration.
	Moreover, it also follows immediately that $\DD_{\infty}(\X_0,\X_1)\leq L$ and thus (according to (i)) that
	\[
		 \D_{\infty}(\X_0,\X_1)\leq \frac{L}{2}.
	\]
	This finally proves
	\begin{equation*}
	\begin{split}
		\D_{q}(\X_0,\X_1)^q & \leq
		 \int_X \ol{\d}^q(x_0,x_1) d\ol{\m}(x_0,x_1) \leq \Big(\frac{L}{2}\Big)^{q-p} \cdot  \int_X \ol{\d}^p(x_0,x_1)
		 d\ol{\m}(x_0,x_1)= \Big(\frac{L}{2} \Big)^{q-p} \cdot \D_p(\X_0,\X_1)^p,
	\end{split}
	\end{equation*}
	where $\ol{\m}$ is now an optimal coupling w.r.t.\ $\D_p$.
For the $L^p$-$L^0$-estimate, recall the obvious estimate
\[ \int\xi^pd\PP=\int_{\{\xi>\epsilon\}} \xi^pd\PP+\int_{\{\xi\leq\epsilon\}}\xi^pd\PP
\le \epsilon L^p+\epsilon^p\le \epsilon (L^p+1)\]
provided $0\le\xi\le L$ and $\PP(\xi>\epsilon)\le\epsilon\le1$.
Applying this to $\xi=\dis(.)$  and to $\xi=\ol\d$, resp., -- in the latter case with $L/2$ in the place of $L$ --  yields the asserted $L^p$-$L^0$-estimates.
\end{proof}

\subsection{$L^0$-Distortion Distance vs. $L^0$-Transportation Distance and Gromov's Box Distance}

Our next goal is to analyze the topologies induced by $\DD_0$ and $\D_0$, resp. For this purpose, define the
\emph{modulus of mass distribution} as a function on $\XX_0\times\R_+$ by
\[\vartheta(\X,r)=\inf\bigg\{\epsilon>0\spec \m\Big(\Big\{x\in X: \m(B_\epsilon(x))\le r\Big\}\Big)\le\epsilon\bigg\} \]
and
put $\Theta(\X,r)=24\vartheta(\X,r^{1/4})+12r^{1/4}$.

\begin{lemma}[\cite{gpw}, Prop. 10.1, Lemma 10.3]
\begin{enumerate}
\item For each $\X_0\in\XX_0$,
\[\lim_{r\to0}\Theta(\X_0,r)=0.\]
\item
For all $\X_0,\X_1\in\XX_0$,
\[\D_0\Big(\X_0,\X_1\Big)\le \Theta\bigg(\X_0,\ \DD_0\Big(\X_0,\X_1\Big)\bigg).\]
\end{enumerate}
\end{lemma}

Recall the corresponding lower bound
$\D_0\Big(\X_0,\X_1\Big)\ge \frac12 \DD_0\Big(\X_0,\X_1\Big)$ from Proposition \ref{lp-d-dd}.

\begin{corollary}
For every sequence $(\X_n)_{n\in\N}$ in $\XX_0$ and every $\X_0\in\XX_0$,
\[\D_0(\X_n,\X_0)\to0\ \mbox{ as }n\to\infty\qquad\Longleftrightarrow\qquad\DD_0(\X_n,\X_0)\to0\ \mbox{ as }n\to\infty.\]
In other words, $\D_0$ and $\DD_0$ induce the same topology on $\XX_0$, called \emph{Gromov-weak topology}.
\end{corollary}

Note that the metric $\D_0$ is complete (\cite{Gro}) whereas $\DD_0$ is non-complete (Example~\ref{ex-incompl}).
Thus, in particular, the two metrics are neither Lipschitz nor H\"older equivalent.

These metrics are closely related to Gromov's \emph{box metric} $\underline\square_\lambda$ defined by
\begin{eqnarray*}
\underline\square_\lambda(\X_0,\X_1)&=&
\inf\bigg\{\epsilon>0\spec \exists \psi_0\in\Par(\m_0), \, \psi_1\in\Par(\m_1)\spec\\
 &&\qquad\qquad\forall s,t\in[0,1-\lambda\epsilon)\spec
\Big|\d_0(\psi_0(s),\psi_0(t))-\d_1(\psi_1(s),\psi_1(t))\Big|\le\epsilon\bigg\}
\end{eqnarray*}
for any $\lambda>0$.
Obviously, $\DD_0$ admits a quite similar representation in terms of parametrizations:
\begin{eqnarray*}
\DD_0(\X_0,\X_1)&=&
\inf\bigg\{\epsilon>0\spec \exists \psi_0\in\Par(\m_0), \, \psi_1\in\Par(\m_1)\spec\\
&&\qquad\qquad\Leb^2\Big(\Big\{(s,t\in[0,1]^2\spec
\Big| \d_0(\psi_0(s),\psi_0(t))-\d_1(\psi_1(s),\psi_1(t)) \Big|
\le\epsilon\Big\}\Big)\ge1-\epsilon\bigg\},
\end{eqnarray*}
the main difference between both formulas being that the `exceptional set' in the first case is the complement of a square (of side length close to 1) within the unit square
whereas in the second case it is any subset of the unit square of small $\Leb^2$-measure.
\begin{lemma}[\cite{Loe}]
\quad $\D_0=\underline\square_{1/2}$.
 \end{lemma}
Together with the trivial estimate $\frac12\underline\square_1\le\underline\square_{1/2}\le\underline\square_1$
this implies
\[\D_0\le \underline\square_{1}\le 2\D_0.\]

\begin{corollary}
For every sequence $(\X_n)_{n\in\N}$ in $\XX_0$ with uniformly bounded diameters, for every $\X_\infty\in \XX_0$ and for all $\lambda>0$ and $p\in[1,\infty)$, the following are equivalent:
\begin{enumerate}
\item
$\X_n\to\X_\infty$ w.r.t.
$\underline\square_\lambda$;
\item
$\X_n\to\X_\infty$ w.r.t.
$\DD_0$;
\item
$\X_n\to\X_\infty$ w.r.t.
$\D_0$;
\item
$\X_n\to\X_\infty$ w.r.t.
$\DD_p$;
\item
$\X_n\to\X_\infty$ w.r.t.
$\D_p$.
\end{enumerate}
If $\X_n=[X_n,\d_n,\m_n]$ with compact spaces $X_n, n\in\N\cup\{\infty\}$, each of these properties will follow from
\begin{enumerate}
\item[(vi)] $(X_n,\d_n,\m_n)\to (X_\infty,\d_\infty,\m_\infty)$ in the  \emph{measured Gromov Hausdorff} sense (`mGH').
\end{enumerate}
Conversely, any of the properties (i)-(v) will imply (vi) provided the spaces $(X_n,\d_n,\m_n)$ have full support and satisfy uniform bounds for doubling constants and diameters.
    \end{corollary}
\begin{proof} For the relation between $\D_p$- and  mGH-convergence we refer to \cite{St06}, Lemma 3.18. The rest is obvious by the previous discussions.
\end{proof}

\begin{remarks}\begin{itemize}
\item
The history of  mm-spaces essentially starts with
Gromov's monograph \cite{Gro}, more precisely, the famous Chapter $3\frac12$ therein.
He promoted very much the idea of focussing on properties which are invariant under isomorphisms.
He also introduced several distances on $\XX_0$, among others,  the box distance
$\underline\square_\lambda$. (Even before that, the topology of mGH-convergence on the space of mm-spaces was introduced by Fukaya \cite{Fuk87}. The concept of mGH-convergence, however, is not compatible with the equivalence relation of isomorphism classes.)
\item
The $L^p$-transportation distance $\D_p$ was introduced and discussed in detail (mainly restricted to the case $p=2$) by the author in
\cite{St06}.
\item
Both the $L^0$-transportation distance and the $L^0$-distortion distance $\DD_0$ were introduced by Greven, Pfaffelhuber and Winter \cite{gpw}. They called them \emph{Gromov-Prohorov metric} and \emph{Eurandom metric}, resp. Indeed, they derived an equivalent formulation for $\DD_0$ in the spirit of the usual definition of the Prohorov distance. 
They also introduced the $L^1$-distortion distance $\DD_1$ (at least for truncated $\d$'s) and gave Example \ref{ex-incompl}
(with non-optimal constants).
The Gromov-Prohorov metric and its relation to  the  so-called Gromov-Hausdorff-Prohorov metric were discussed in
\cite{Vi09}.

\item
The space $\XX_0$ serves as an important model in image analysis and shape matching.
In a series of papers, Memoli
  introduced and analyzed various  distances
 (partly for finite, partly for compact mm-spaces)
 with emphasis on computational aspects and in view of applications to shape matching and object recognition.
In \cite{Me}, he presented an exhaustive survey on the distances $\DD_p$ and $\D_p$ (which he denoted by $2\mathcal D_p$ and $\mathcal S_p$, resp.), their mutual relations and applications in image analysis. Among others, he deduced a slightly restricted version of Proposition \ref{d-is-metric} (i.e. restricted to compact mm-spaces) as well as several estimates of Proposition \ref{lp-d-dd}
    (partly with non-optimal constants).
\item
In recent years, the concept of mm-spaces and related topological/metric issues on the space $\XX_0$ found surprising new applications in the study of random graphs
and their limits, e.g. the continuum random tree or the Brownian map,  see e.g. \cite{gpw}, \cite{adh}, \cite{Legall} and   \cite{Mie07}.
\end{itemize}
\end{remarks}

In none of the previous works, any geometric properties of the space $\XX_0$ itself have been derived. (The only exception might be \cite{St06} where geodesics had been characterized.) From our point of view, the emphasis of this paper is not on the `metric results' from the previous chapters but on the `geometric results' (concerning geodesics, curvature, quasi-Riemannian tangent structure etc.) of the subsequent chapters.

\section{Geodesics in $(\XX_p,\DD_p)$}

Recall that (as usual in metric geometry) a curve $(\X_t)_{t\in J}$
-- where $J$ denotes some interval in $\R$ -- is called \emph{geodesic} if $\forall S,s,t,T\in J$ with $S<s<t<T$:
\[
	\DD_p(\X_s,\X_t)=\frac{t-s}{T-S}\DD_p(\X_S,\X_T).
\]
Thus, by definition, geodesics are always distance minimizing and have constant speed.
\begin{theorem}\label{thmoptcoupl} For each $p\in[1,\infty]$, $\big(\XX_p, \DD_p\big)$ is a geodesic space. More specifically, the following assertions hold:
	\begin{enumerate}
				\item\label{optcoupl2}
For  each pair of mm-spaces $\X_0, \X_1\in \XX_p$ and each optimal coupling
  $\ol{\m}$ of them (cf. Definition \ref{def-opt-cpl}), the family of metric
			measure spaces
			\[
				{\X}_t=[{X_0\times X_1},{\d}_t,\ol{\m}],\qquad t\in (0,1),
			\]
			with
			\[
				{\d}_t\left((x_0,x_1),(y_0,y_1)\right):= (1-t)\d_0(x_0,y_0)+t \d_1(x_1,y_1)
			\]
			defines a geodesic $({\X}_t)_{0\leq t \leq 1}$ in $\XX_p$ connecting $\X_0$ and $\X_1$.
		\item\label{optcoupl3}
If $p\in (1,\infty)$, then each geodesic $({\X}_t)_{0\leq t\leq 1}$ in $\XX_p$ is of the form as stated in
			(\ref{optcoupl2}).
			That is, for each geodesic $({\X}_t)_{0\leq t\leq 1}$ there exists an optimal coupling $\ol{\m}$ of the measures $\m_0, \m_1$, defined on the product space of $(X_0,\d_0,\m_0)$ and $(X_1,\d_1,\m_1)$, representatives of the endpoints, such that for each $t\in (0,1)$ a representative of the
			 isomorphism class
			$\X_t$ is given by $(X_0\times X_1,\d_t,\m)$ with $\d_t:=(1-t)\d_0+t\d_1$.
	\end{enumerate}
\end{theorem}
 Note that in the case $p\in(1,\infty)$ a conclusion from (ii) is that geodesics  $({\X}_t)_{0\leq t\leq 1}$ in $\XX_p$ do not branch at times $t\not= 0,1$. And they do not collapse to atoms at interior points. More precisely,

\begin{corollary}
If $({\X}_t)_{t\in [0,1]}$ and $({\X'}_t)_{t\in [0,1]}$ are two non-identical geodesics in $\XX_p$ (for  $1<p<\infty$) with identical initial and terminal points (i.e. $\X_0=\X'_0, \X_1=\X'_1$ and $\X_t\not=\X'_t$ for some $t\in(0,1)$)
then none of these geodesics can be extended to a geodesic beyond $t=0$ or $t=1$.
\end{corollary}

\begin{corollary} If the initial point $\X_0$ of a geodesic $({\X}_t)_{t\in [0,1]}$ in $\XX_p$ (for  $1<p<\infty$) has no atoms then each inner point ${\X}_t$, $t\in (0,1)$, of the geodesic has no atoms.
\end{corollary}

\begin{proof}[Proof of the theorem]
		(\ref{optcoupl2}) 	
	In order to prove that $({\X}_t)_{0\leq t\leq 1}$ is a geodesic in $\XX_p$, it suffices to verify
	that
	\[
		\DD_p({\X}_s,{\X}_t)\leq |s-t| \DD_p({\X}_0,{\X}_1)
	\]
	for all $s,t\in [0,1]$. We will restrict the discussion to the case $p<\infty$.
	For a given pair $s,t\in (0,1)$, note that the `diagonal coupling'
	\[
		d\dol{\m}(x,y):=d\delta_x(y)d\ol{\m}(x)
	\]
	is one of the possible couplings of the measures of ${\X}_s$ and ${\X}_t$ (both being
	$\ol{\m}$).
	Thus,  with $X:=X_0\times X_1$
	\begin{equation*}
	\begin{split}
		\DD_p({\X}_s,{\X}_t)^p & \leq \int_{{X}\times{X}} \int_{{X}\times{X}}
				\left| {\d}_s(x,y)-{\d}_t(x',y') \right|^p  d\dol{\m}(x,x') d\dol{\m}(y,y')\\
			& = \int_{{X}} \int_{{X}} \left| {\d}_s(x,y)-{\d}_t(x,y) \right|^p  d\ol{\m}(x)
				d\ol{\m}(y)\\
			& = |s-t|^p \int_{{X}} \int_{{X}} \left| \d_0(x_0,y_0)-\d_1(x_1,y_1) \right|^p  d\ol{\m}(x_0,x_1) 	 d\ol{\m}(y_0,y_1)\\
			& = |s-t|^p \DD_p(\X_0,\X_1)^p.
	\end{split}
	\end{equation*}
In the case $s=0$ and $t\in (0,1)$, a slight modification of the argument is requested.
 Now we choose
 \[
		d\dol{\m}(x_0,y):=d\delta_{y_0}(x_0)d\ol{\m}(y)
	\]
(where $y=(y_0,y_1)$)	as one of the possible couplings of the measures $\m_0$ of ${\X}_0$ and $\ol\m$ of ${\X}_t$. Then the argument works as before.
Similarly, for the case $s\in (0,1)$ and $t=1$.

	(\ref{optcoupl3}) 
Let a geodesic  $({\X}_t)_{0\leq t\leq 1}$ in $\XX_p$ be given. Fix a  number $k\in\N$ and let
$\mu_i$ (for $i=1,\ldots, 2^k$) be optimal couplings of the measures $\m_{(i-1)2^{-k}}$ and $\m_{i2^{-k}}$.
Glue together all these couplings
to obtain a probability measure
\[\mu=\mu_1\boxtimes\mu_2\boxtimes\ldots\boxtimes \mu_{2^k}\]
on $X_0\times X_{2^{-k}}\times\ldots \times X_{i2^{-k}}\times\ldots\times X_1$.
Put $\ol\m=(\pi_0,\pi_1)_{\push} \mu$ as well as $\ol\m_{t}=(\pi_0,\pi_t,\pi_1)_{\push} \mu$ for all $t\in(0,1)$ of the form $t=i2^{-k}$ (for $i=1,\ldots, 2^k-1$).
Thus $\ol\m$ is a coupling of $\m_0$ and $\m_1$ (a priori not optimal).

Let us now first restrict to the case $p\ge 2$. Then for each $t=i2^{-k}$ (for some $i=1,\ldots, 2^k-1$),
\begin{eqnarray*}
\lefteqn{\DD_p(\X_0,\X_1)^p}\\
&\stackrel{(*)}\le&\int
\int
\Big| \d_0(x_0,y_0)-\d_1(x_1,y_1) \Big|^p  d\ol{\m}(x_0,x_1) 	 d\ol{\m}(y_0,y_1)\\
&=&
\int
\int
\Big| \big[\d_0(x_0,y_0)-\d_t(x_t,y_t)\big]+\big[ \d_t(x_t,y_t)-\d_1(x_1,y_1)\big]\Big|^p  d\ol{\m}_t(x_0,x_t,x_1) 	 d\ol{\m}_t(y_0,y_t,y_1)\\
&\stackrel{(**)}\le&
\int
\int
\Big[ \frac1{t^{p-1}}\big|\d_0(x_0,y_0)-\d_t(x_t,y_t)\big|^p+\frac1{(1-t)^{p-1}}\big| \d_t(x_t,y_t)-\d_1(x_1,y_1)\big|^p\Big]
d\ol{\m}_t(x_0,x_t,x_1) 	 d\ol{\m}_t(y_0,y_t,y_1)\\
&&-\frac1{C[t(1-t)]^{p-1}}
\int
\int
\Big|(1-t)\big[\d_0(x_0,y_0)-\d_t(x_t,y_t)\big]-t\big[\d_t(x_t,y_t)-\d_1(x_1,y_1)\big]\Big|^p\\
&&\qquad\qquad\qquad\qquad\qquad\qquad\qquad\qquad\qquad\qquad\qquad\qquad\qquad\qquad
  d\ol{\m}_t(x_0,x_t,x_1) 	 d\ol{\m}_t(y_0,y_t,y_1)\\
  &=& {\bf (I)}\ -\ {\bf (II)}.
  \end{eqnarray*}
The last inequality $(\ast\ast)$ is based on the estimate (ii) of  Lemma \ref{unif p convex} below, applied pointwise to the integrand taking $a=\frac{\d_0-\d_t}{t}$ and $b=\frac{\d_t-\d_1}{1-t}$. In the case $p=2$, it is even an equality with $C=1$.

Let us have a closer look on the first integral $\bf (I)$. Using estimate (i) of the Lemma below, it can be
bounded from above as follows
\begin{eqnarray*}
{\bf (I)}
&=&
2^{k(p-1)}\int
\int
\left[ \frac1{i^{p-1}}\big|\d_0(x_0,y_0)-\d_{i2^{-k}}(x_{i2^{-k}},y_{i2^{-k}})\big|^p+\frac1{(2^k-i)^{p-1}}\big| \d_{i2^{-k}}(x_{i2^{-k}},y_{i2^{-k}})-\d_1(x_1,y_1)\big|^p\right]\\
&&\qquad\qquad\qquad\qquad\qquad\qquad\qquad\qquad\qquad\qquad
d\mu(x_0,\ldots,x_{i2^{-k}},\ldots, x_1) 	 d\mu(y_0,\ldots,y_{i2^{-k}},\ldots,y_1)\\
&\le&
2^{k(p-1)}\sum_{j=1}^{2^k}\int
\int
\Big|\d_{(j-1)2^{-k}}(x_{(j-1)2^{-k}},y_{(j-1)2^{-k}})-\d_{j2^{-k}}(x_{j2^{-k}},y_{j2^{-k}})\Big|^p\\
&&\qquad\qquad\qquad\qquad\qquad\qquad
d\mu(x_0,\ldots,x_{(j-1)2^{-k}},x_{j2^{-k}},\ldots, x_1) 	 d\mu(y_0,\ldots,y_{(j-1)2^{-k}},y_{j2^{-k}},\ldots,y_1)\\
&=&2^{k(p-1)}\sum_{j=1}^{2^k}\DD_p(\X_{(j-1)2^{-k}}, \X_{j2^{-k}})^p\\
&=&\DD_p(\X_0,\X_1)^p.
\end{eqnarray*}
This allows two conclusions: i) The coupling $\ol\m$ of $\m_0$ and $\m_1$ is optimal since the very first inequality $(*)$ must be an equality.
ii) The second integral $\bf (II)$ in the above derivation must vanish. That is,
\[\int_{X_0\times  X_t\times X_1}
\int_{X_0\times X_t\times X_1}
\Big|(1-t)\d_0(x_0,y_0)+t\d_1(x_1,y_1)-\d_t(x_t,y_t)\Big|^p
  d\ol{\m}_t(x_0,x_t,x_1) 	 d\ol{\m}_t(y_0,y_t,y_1)=0.\]
Since
$\ol\m_t$ is a coupling of $\ol\m$ and $\m_t$, this implies
that the mm-spaces $(X_t,\d_t,\m_t)$ and
$(X_0\times X_1,(1-t)\d_0+t\d_1,\ol\m)$  are isomorphic.
This holds true for any $t\in(0,1)$ of the form
$t=i2^{-k}$  for some $i=1,\ldots, 2^k-1$.

\medskip

Now let us consider the case $p\le 2$ which requires a slightly modified argumentation.
Here we consider the $L^p$-distortion distance to the power 2. It yields
\begin{eqnarray*}
\lefteqn{\DD_p(\X_0,\X_1)^2}\\
&\le&
\bigg(
\int
\int
\Big| \big[\d_0(x_0,y_0)-\d_t(x_t,y_t)\big]+\big[ \d_t(x_t,y_t)-\d_1(x_1,y_1)\big]\Big|^p  d\ol{\m}_t(x_0,x_t,x_1) 	 d\ol{\m}_t(y_0,y_t,y_1)\bigg)^{2/p}\\
&\stackrel{(***)}\le&
\frac1{t}\bigg(\int
\int
\Big[ \big|\d_0(x_0,y_0)-\d_t(x_t,y_t)\big|^p
\Big]
d\ol{\m}_t(x_0,x_t,x_1) 	 d\ol{\m}_t(y_0,y_t,y_1)
\bigg)^{2/p}\\
&&\qquad
+\frac1{1-t}
\bigg(\int
\int
\Big[
\big| \d_t(x_t,y_t)-\d_1(x_1,y_1)\big|^p\Big]
d\ol{\m}_t(x_0,x_t,x_1) 	 d\ol{\m}_t(y_0,y_t,y_1)\bigg)^{2/p}\\
&&-\frac{p-1}{t(1-t)}
\bigg(\int
\int
\Big|(1-t)\big[\d_0(x_0,y_0)-\d_t(x_t,y_t)\big]-t\big[\d_t(x_t,y_t)-\d_1(x_1,y_1)\big]\Big|^p
  d\ol{\m}_t(x_0,x_t,x_1) 	 d\ol{\m}_t(y_0,y_t,y_1)\bigg)^{2/p}\\
  &=& {\bf (I')}\ -\ {\bf (II')}.
  \end{eqnarray*}
Now the last inequality $(\ast\ast\ast)$ is based on the estimate (iii) of  Lemma \ref{unif p convex} below, applied to the $L^p$-norms (w.r.t. the measure $\ol\m_t^2$) of the involved functions.

The quantity ${\bf (I')}$ can be estimated similarly as before, using the triangle inequality for the $L^p$-norm and estimate (i) of Lemma~\ref{unif p convex} with $p=2$:

\begin{eqnarray*}
{\bf (I')}
&=&
\frac{2^{k}}{i}\bigg(\int
\int
 \big|\d_0(x_0,y_0)-\d_{i2^{-k}}(x_{i2^{-k}},y_{i2^{-k}})\big|^p\\
&&\qquad\qquad\qquad\qquad\qquad\qquad\qquad\qquad
 d\mu(x_0,\ldots,x_{i2^{-k}},\ldots, x_1) 	 d\mu(y_0,\ldots,y_{i2^{-k}},\ldots,y_1)\bigg)^{2/p}\\
&&  +
  \frac{2^{k}}{2^k-i}\bigg(\int
\int
  \big| \d_{i2^{-k}}(x_{i2^{-k}},y_{i2^{-k}})-\d_1(x_1,y_1)\big|^p\\
&&\qquad\qquad\qquad\qquad\qquad\qquad\qquad\qquad
d\mu(x_0,\ldots,x_{i2^{-k}},\ldots, x_1) 	 d\mu(y_0,\ldots,y_{i2^{-k}},\ldots,y_1)\bigg)^{2/p}\\
&\le&
2^{k}\sum_{j=1}^{2^k}
\bigg(
\int
\int
\Big|\d_{(j-1)2^{-k}}(x_{(j-1)2^{-k}},y_{(j-1)2^{-k}})-\d_{j2^{-k}}(x_{j2^{-k}},y_{j2^{-k}})\Big|^p\\
&&\qquad\qquad\qquad\qquad\qquad
d\mu(x_0,\ldots,x_{(j-1)2^{-k}},x_{j2^{-k}},\ldots, x_1) 	 d\mu(y_0,\ldots,y_{(j-1)2^{-k}},y_{j2^{-k}},\ldots,y_1)\bigg)^{2/p}\\
&=&2^{k}\sum_{j=1}^{2^k}\DD_p(\X_{(j-1)2^{-k}}, \X_{j2^{-k}})^2\\
&=&\DD_p(\X_0,\X_1)^2.
\end{eqnarray*}
This allows the very same conclusions as before: i) the coupling is optimal and ii) the mm-spaces $(X_t,\d_t,\m_t)$ and
$(X_0\times X_1,(1-t)\d_0+t\d_1,\ol\m)$  are isomorphic.

\medskip

To indicate the dependence on $k$, let us now denote the optimal coupling $\ol\m$ (obtained via the above construction) by $\ol\m^{(k)}$.
According to Lemma \ref{cpl-comp}, the family $(\ol\m^{(k)})_{k\in\N}$ has an accumulation point $\ol\m^{(\infty)}$ in $\Cpl(\m_0,\m_1)$.
With this $\ol\m^{(\infty)}$ in the place of the previous $\ol\m^{(k)}$ it follows that for all dyadic numbers $t\in (0,1)$, the  mm-spaces $(X_t,\d_t,\m_t)$ and
$(X_0\times X_1,(1-t)\d_0+t\d_1,\ol\m^{(\infty)})$  are isomorphic. Continuity of both as elements in $\XX$ in $t$ finally allows to conclude this identification for all $t\in (0,1)$.
\end{proof}

In the previous proof we used the following basic estimates between real numbers, partly known as Clarkson's inequalities.
\begin{lemma}\label{unif p convex}
\begin{enumerate}
\item
$\forall p\in (1,\infty)$, $\forall t_0< t_1\ldots<t_n$, $\forall a_1,\ldots,a_n\in\R_+$
\[\frac1{(t_n-t_0)^{p-1}}\Big(\sum_{i=1}^n a_i\Big)^p\le\sum_{i=1}^n\frac1{(t_i-t_{i-1})^{p-1}}a_i^p.\]
\item
$\forall p\in [2,\infty), \forall t\in (0,1)\spec\exists C=C(p,t)>0\spec\forall a,b\in\R$
\[|ta+(1-t)b|^p\le t |a|^p+(1-t)|b|^p-\frac{t(1-t)}C|a-b|^p.\]
\item
For all $p\in(1,2]$, all $t\in (0,1)$, all probability spaces $(\Omega,\frak A,\mathbb P)$ and all $f,g\in L^p(\Omega,\mathbb P)$,
\[ \|tf+(1-t)g\|_p^2\le t\|f\|_p^2+(1-t)\|g\|_p^2-(p-1)t(1-t)\|f-g\|_p^2.\]
\end{enumerate}
\end{lemma}

\begin{proof}
(i) Consequence of Jensen's inequality applied to numbers $\frac{a_i}{t_i-t_{i-1}}$ and weights $\lambda_i=\frac{t_i-t_{i-1}}{t_n-t_0}$ with $\sum_i\lambda_i=1$.

For (ii) and (iii), see e.g. Prop.~3 of \cite{bcl}. (ii) is the quantitative version of the \emph{uniform convexity} of $r\mapsto r^p$ for $p\ge 2$.
(iii) is the \emph{2-convexity of the $L^p$-norm} for $p\le 2$. Actually, both inequalities are stated only for $t=\frac12$. However, a simple iteration argument allows to deduce them for arbitrary dyadic $t$ (with the optimal constant in case of (iii) and with some constant $C(p,t)>0$ in case of (ii)).
\end{proof}

\begin{remark} Given a mm-space $\X_0$ we say that another mm-space $\X_1$ is a \emph{regular target} for $\X_0$ if there exists a measurable map $\phi:X_0\rightarrow X_1$ such that
\[\ol\m=(\Id,\phi)_{\push}\m_0\] is a coupling of $\m_0$ and $\m_1$ which is optimal for $\DD_p$. In other words,
 $\X_1$ is a {regular target} for $\X_0$ if there exists a measurable map $\phi$ with $\phi_{\push} \m_0=\m_1$ such that
\begin{equation}
\label{monge}
	\DD_p(\X_0,\X_1)^p=\int_{X_0} \int_{X_0} \left| \d_0(x,y)-\d_1(\phi(x),\phi(y))\right|^p
						d\m_0(x) d\m_0(y).
\end{equation}
A geodesic $\geod{0}{1}$ emanating from $\X_0$ is called \emph{regular} (for $\X_0$) if it connects $\X_0$ with some regular target $\X_1$. Such a geodesic can be represented on the state space of $\X_0$ as
\[\X_t=\big[ X_0, (1-t)\d_0+t\,\phi^*\d_1,\m_0\big]\]
where $\phi^*\d_1$ denotes the pull back of $\d_1$ from $X_1$ to $X_0$ through $\phi$, that is,
$\phi^*\d_1(x_0,y_0)=\d_1(\phi(x_0),\phi(y_0))$.

This is in analogy to the `classical' theory of optimal transportation where in `nice situations'
the (unique) solution to the Kantorovich problem coincides with the solution to the Monge problem.
Note, however, that there is a significant difference to the `classical' theory of optimal transportation on Euclidean or Riemannian spaces.
\begin{itemize}
\item
`Nice' points $\mu_0$ of the Wasserstein space $\mathcal P_p(X)$ on a Riemannian manifold $X$ have the property that \emph{each} target $\mu_1\in \mathcal P_p(X)$ is regular for $\mu_0$. For instance, all probability measures $\mu_0$ which are absolutely continuous with respect to the volume measure on $X$ are `nice'.
\item
In contrast to that,
even for `nice' points in $\XX_p$ like smooth compact Riemannian manifolds, e.g. $n$-dimensional spheres $\mathbb{S}^n$, we expect that there are plenty of non-regular targets, e.g. products $\mathbb{S}^n\times \mathbb{S}^k$.
\end{itemize}
\end{remark}

\begin{challenge} \begin{enumerate}
\item Prove the existence (and uniqueness) of such a transport map $\phi$ between `nice' spaces (e.g. smooth compact Riemannian manifolds of the same dimension) -- i.e. $\XX_p$-version  of Brenier \cite{Bre91} and McCann \cite{McC01};
    \item Derive regularity and smoothness results for this map -- i.e. $\XX_p$-version of Ma, Trudinger, Wang \cite{mtw}.
    \end{enumerate}
\end{challenge}

For  further discussions and results for geodesic interpolations of Riemannian manifolds, we refer to the last chapter, in particular, to Example \ref{riemannian}.

\bigskip

\begin{definition}
A metric measure space $(X,\d,\m)$ is  called \emph{geodesic mm-space} if for all $x,y\in\supp(\m)$ there exists a curve $\gamma: [0,1] \to  \supp(\m)$ with
$\gamma_0=x, \gamma_1=y$ and $\length(\gamma)=\d(x,y)$.

$(X,\d,\m)$ is called \emph{length mm-space} if for all $x,y\in\supp(\m)$
\[\d(x,y)=\inf\Big\{ \length(\gamma)\spec \gamma \mbox{ in }\supp(\m), \gamma_0=x, \gamma_1=y\Big\}.\]

\end{definition}

Obviously, a mm-space $(X,\d,\m)$ is a geodesic (or length) mm-space if and only if the \emph{metric space}
$(\supp(\m),\d)$ is a geodesic (or length, resp.) space in the usual sense of metric geometry, see e.g. \cite{bbi}. 
It is easy to see that being a geodesic (or length) mm-space is a property of the isomorphism class  $[X,\d,\m]$.
The space of all isomorphism classes of geodesic mm-spaces will be denoted by $\XX^{geo}$ and the space of all length mm-spaces by $\XX^{length}$.
\index{x@$\XX^{geo}, \XX^{length}$}

\begin{remark}
Given two mm-spaces $[X_0,\d_0,\m_0]$, $[X_1,\d_1,\m_1]$ then every connecting geodesic can be represented
as $[X_0\times X_1,\d_t,\overline\m]$ for some optimal coupling $\overline\m$.
If the  metric spaces $(X_0,\d_0)$ and $(X_1,\d_1)$ are geodesic then for each $t\in(0,1)$, the $t$-intermediate metric space $(X_0\times X_1,\d_t)$ is geodesic.
(Analogously, for length spaces.)

Note, however, that in general the midpoints in $(X_0\times X_1,\d_t)$ will not necessarily be in $\supp(\ol\m)$.
\end{remark}

\begin{proof}
It is well-known that $(X,\d)$ is a geodesic (or length, resp.) space if and only if  for each pair $(x,y)\in X^2$ there exists a \emph{midpoint} $M(x,y)$ (or a sequence of $1/n$-midpoints $M_n(x,y)$, resp.) characterized by
\[ \d(x,M(x,y))=\d(y,M(x,y))=\frac12\d(x,y)\]
(or $\d(x,M_n(x,y))\le (\frac12+\frac1n)\d(x,y)$ and $\d(y,M_n(x,y))\le (\frac12+\frac1n)\d(x,y)$, resp.).

Now let two geodesic mm-spaces $[X_0,\d_0,\m_0]$ and $[X_1,\d_1,\m_1]$ be given as well as an optimal coupling
$\ol\m\in\Cpl(\m_0,\m_1)$.  Assume without restriction that the chosen representatives have full support.
(This does not imply that $\ol\m$ has full support in $X_0\times X_1$.)
Let
\[M_0: X_0^2\to X_0,\quad M_1: X_1^2\to X_1\]
be the  midpoint maps and define
\[M:\quad\begin{array}{ccc}
(X_0\times X_1)^2&\to& X_0\times X_1\\
\big((x_0,x_1), (y_0,y_1)\big)&\mapsto &\Big(M_0(x_0,y_0), M_1(x_1,y_1)\Big).
\end{array}\]
Then for each $t\in (0,1)$,
$M$ is a midpoint map for $(X_0\times X_1, \d_t)$ with $\d_t=(1-t)\d_0+t\d_1$.
Indeed,
\begin{eqnarray*}
\d_t(x,M(x,y))&=&
(1-t)\d_0(x_0,M_0(x_0,y_0))+t \d_1(x_1,M_1(x_1,y_1))\\
&=&
(1-t)\frac12\d_0(x_0,y_0)+t\frac12 \d_1(x_1,y_1)\\
&=&\frac12\d_t(x,y)
\end{eqnarray*}
and also $\d_t(y,M(x,y))=\frac12\d_t(x,y)$.

Essentially the same argumentation applies to $1/n$-midpoint maps in the case of length spaces.
\end{proof}

\begin{remarks} \begin{enumerate}
\item
Since the set of all possible midpoints
is closed the measurable selection theorem provides a \emph{Borel measurable} map $M: X^2\to X$ such that for each $x,y\in X^2$ the point $M(x,y)$ is a midpoint of $x$ and $y$, provided of course $X$ is a geodesic space.
Similarly, for each $n\in\N$ it provides a Borel measurable $1/n$-midpoint map on a given length space.

\item
Neither $\XX^{geo}$ nor $\XX^{length}$ is \emph{closed}. An easy counterexample is provided by the sequence of geodesic mm-spaces
\[ \Big[ I, |.|, \frac1n \Leb^1+\frac12(1-\frac1n)\delta_0+\frac12(1-\frac1n)\delta_1\Big]\]
which $\DD_p$-converges to
\[ \Big[ I, |. |, \frac12\delta_0+\frac12\delta_1\Big].\]
\end{enumerate}
\end{remarks}

\section{Cone Structure and Curvature Bounds for $(\XX, \DD)$}\label{seccurvbds}

\subsection{Cone Structure}

From now on, for the rest of the paper we will restrict ourselves to the case $p=2$. We simply write $\XX$ instead of $\XX_2$, $\DD$ instead of $\DD_2$, and $\size(.)$ instead of $\size_2(.)$.

We begin with a reformulation of the $L^2$-distortion distance which is analogous to the reformulations of the classical transport problem for the cost functions $|x-y|^2$ in terms of the transport problem for the cost function $-2xy$. Indeed, such a result only holds for $p=2$.

\begin{proposition}\label{product}
	$\forall \X_0,\X_1\in \XX$:
	\begin{equation*}
	\begin{split}
		\DD(\X_0,\X_1)^2 = & \size(\X_0)^2 + \size(\X_1)^2 \\
							& - 2\sup \bigg\{
							\int_{X_0\times X_1} \int_{X_0\times X_1}
						\d_0(x_0,y_0) \d_1(x_1,y_1) d\ol{\m}(x_0,x_1) d\ol{\m}(y_0,y_1)\spec
							\ol{\m}\in\Cpl(\m_0,\m_1)\bigg\}.
	\end{split}
	\end{equation*}
\end{proposition}

\begin{proof}
	Decompose the integrand $|\d_0(x_0,y_0)-\d_1(x_1,y_1)|^2$ in the integrals used in the definition
	of $\DD^2$ into two squares of distances and a midterm. Then observe that each of the integrals
	of a distance square only depends on one of the marginals of $\ol{\m}$, e.g.
	\[
		\int_{X_0\times X_1} \int_{X_0\times X_1} \d_0(x_0,y_0)^2 d\ol{\m}(x_0,x_1) d\ol{\m}(y_0,y_1)
		= \int_{X_0}\int_{X_0}  \d_0(x_0,y_0)^2 d\m(x_0) d\m(y_0)
		= \size(\X_0)^2.
	\]
\end{proof}

The space $\XX_0$ has a distinguished element: the isomorphism class of metric measure spaces $(X,\d,\m)$ whose support consist of one point, say $x\in X$ (and thus $\m=\delta_x$). This isomorphism class will be called \emph{$1$-point space} and denoted by $\bfdelta$.\index{d@$\bfdelta$} Note that for each $\X\in \XX_0$,
\[
\size(\X)=\DD(\bfdelta,\X)
\]
and thus
\[
	\XX^1:= \{ \X\in \XX \spec \size(\X)=1\}
\]
is the unit sphere in $(\XX,\DD)$ around $\bfdelta$.\index{x@$\XX^1$}
Given any $\X_1=[X_1,\d_1,\m_1]\in \XX$, the unique unit speed geodesic through $\X_1$ and emanating from $\bfdelta$ is given by
\[
	\X_t=[X_1,t\d_1,\m_1].
\]
It is called \emph{ray} through $\X_1$.
Each element $\X\neq\bfdelta$ in $\XX$ can uniquely be characterized as a pair $(r,\X_1)\in (0,\infty)\times \XX^1$. The number $r$ is the size of $\X$, the element $\X_1\in \XX^1$ is the
`standardization' of $\X=[X, \d,\m]$:
\[
	\X_1:=[X, \frac{\d}{\size(\X)},\m].
\]

A remarkable, quite surprising fact is that the $L^2$-distortion distance between two spaces $\X=(r,\X_1)$ and $\X'=(r',\X_1')$ is completely determined by the sizes $r=\size(\X),\,r'=\size(\X')$ and the distance $\DD(\X_1,\X_1')$ of the standardized spaces.

\begin{lemma}
	Let $\X_1,\X_1'\in \XX^1$ and let $(\X_s)_{s\geq 0}, (\X'_t)_{t\geq 0}$ be the corresponding rays.
	Then the quantity
	\[
		\frac{1}{2st} \left[ \DD^2(\X_s,\X_t')-s^2-t^2\right]
	\]
	is independent of $s,t\in (0,\infty)$.
\end{lemma}

\begin{proof}
	Let the rays be given as $\X_s=(X,s\d,\m)$ and $\X_t'=(X',t\d',\m')$. Then for each
	$\ol{\m}\in\Cpl(\m,\m')$ and all $s,t\in (0,\infty)$:
	\begin{align*}
			& \frac{1}{2st}\left[ \int_{X\times X'} \int_{X\times X'}
								\left| s\d(x,y)-t\d'(x',y') \right|^2 d\ol{\m}(x,x') d\ol{\m}(y,y')
								-s^2-t^2
							\right] \\
			= &\frac{1}{2st}\left[ s^2 \int_{X} \int_{X}
								\d(x,y)^2 d\m(x) d\m(y)-s^2 \right.\\
								&\hphantom{\frac{1}{2st}\left[ \right.}
								+ t^2 \int_{X'} \int_{X'}
								\d'(x',y')^2 d\m'(x') d\m'(y')-t^2\\
								&\hphantom{\frac{1}{2st}\left[ \right.}
								\left. -2st\int_{X\times X'} \int_{X\times X'}
								\d(x,y)\d'(x',y')
								 d\ol{\m}(x,x') d\ol{\m}(y,y')
						\right] \\
			= &- \int_{X\times X'} \int_{X\times X'}\d(x,y)\d'(x',y')
								 d\ol{\m}(x,x') d\ol{\m}(y,y'),
	\end{align*}
	which obviously is independent of $s$ and $t$.
	The last equality is due to the fact that $\size(\X_1)=1$ as well as $\size(\X'_1)=1$.
\end{proof}

For $\X,\X'\in \XX^1$ put
\[
	\DD^{(1)}(\X,\X'):= 2\arcsin(\frac{1}{2} \DD(\X,\X')).
\]
Of course, this is equivalent to saying that
\[\DD(\X,\X')^2=2-2 \cos \DD^{(1)} (\X,\X').\]
\begin{figure}[h!]

	\begin{subfigure}{0.4\textwidth}
		
		\psfrag{delta}{$\bfdelta$}
		\psfrag{XX1}{$\XX^1$}
		\psfrag{d}{\color{green}$\DD$}
		\psfrag{d1}{\color{red}$\DD^{(1)}$}

		\includegraphics[scale=0.4]{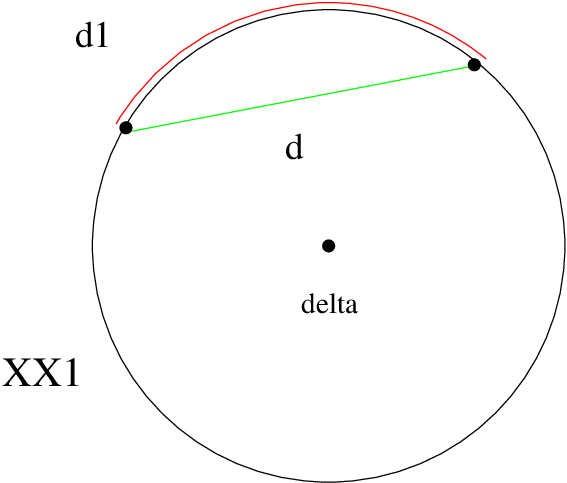}
	\caption*{\wichtig{$\sin(\frac{\DD^{(1)}}{2})=\frac{\DD}{2}$}}
	\end{subfigure}
	\qquad
	\begin{subfigure}{0.4\textwidth}

		\psfrag{X1}{$\X_1$}
		\psfrag{Xs}{$\X_s$}
		\psfrag{Xt}{$\X_t$}
		\psfrag{X12}{$\X_1'$}
		\psfrag{delta}{$\bfdelta$}
		\psfrag{angle}{\color{red}$\measuredangle=\DD^{(1)}(\X_1,\X_1')$}
		\psfrag{dXsXt}{\color{green}$\DD(\X_s,\X_t')$}
				\includegraphics[scale=0.35]{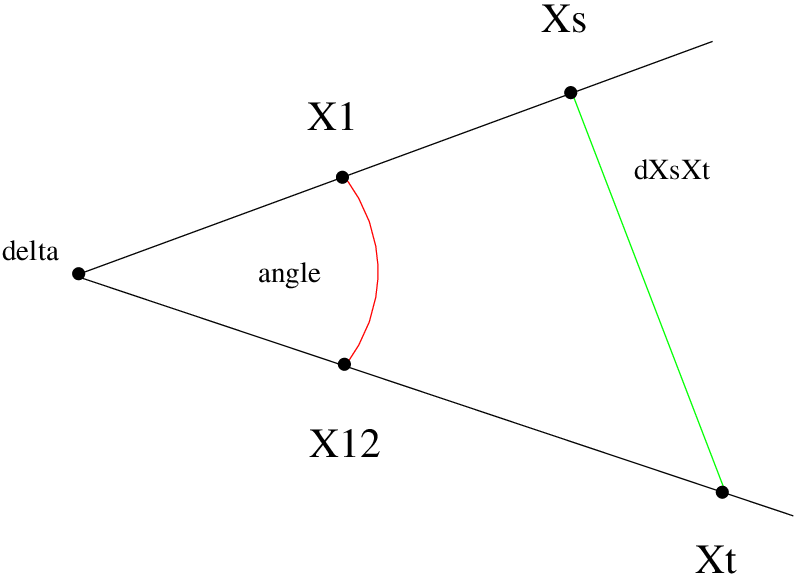}
	
\caption*{\wichtig{Law of cosines:\quad $\DD^2=s^2+t^2-2st\measuredangle$}}
		\end{subfigure}
\caption{\wichtig{Cone structure}}
\end{figure} 

Thus we have proved the following:
\begin{theorem}\label{cone}
			The space $\XX$ is the cone over $\XX^1$. For each $\X_1,\X_1'\in\XX^1$ and for all $s,t
			\in (0,\infty)$:
			\[
				\DD(\X_s,\X_t')^2=s^2+t^2-2st \cos \DD^{(1)} (\X_1,\X_1'),
			\]
			where $\X_s$ denotes the point with size $s$ on the ray through $\X_1$ and, similarly,
			$\X_t'$ the point with size $t$ on the ray through $\X_1'$.
	
\end{theorem}

\subsection{Curvature Bounds}

\begin{theorem}\label{thmXX}
	$(\XX,\DD)$ is a geodesic space of nonnegative curvature in the sense of Alexandrov:
both the triangle comparison and the quadruple comparison property are satisfied. That is,
\begin{enumerate}
\item for each geodesic $(\X_t)_{0\leq t\leq1}$ in $\XX$ and each point
	$\X'$ in $\XX$,
	\begin{align}
		\DD^2(\X_t,\X')\geq (1-t) \DD^2(\X_0,\X') + t \DD^2(\X_1, \X') -t (1-t)\DD^2(\X_0,\X_1);
	\end{align}
\item
for each quadruple of points $\X_0, \X_1,\X_2,\X_3$ in $\XX$,
\[\sum_{i=1,2,3}\DD^2(\X_0,\X_i)\ge \frac13\sum_{1\leq i<j\leq 3}\DD^2(\X_i,\X_j).\]
\end{enumerate}
\end{theorem}
Note that for \emph{complete} length spaces, properties (i) and (ii) are known to be equivalent \cite{lp}. However, due to lack of completeness this does not apply directly.
\begin{figure}[h!]

	\begin{subfigure}{0.4\textwidth}
\begin{center}
\psfrag{X0}{\color{green}$\X_0$}
\psfrag{X1}{\color{green}$\X_1$}
\psfrag{Xt}{\color{green}$\X_t$}
\psfrag{Xprime}{$\X'$}


\caption*{\wichtig{Triangle comparison}}
\end{center}
	\end{subfigure}
	\qquad
	\begin{subfigure}{0.3\textwidth}
\begin{center}
\psfrag{X2}{\color{green}$\X_2$}
\psfrag{X1}{\color{green}$\X_1$}
\psfrag{X3}{\color{green}$\X_3$}
\psfrag{X0}{\color{red}$\X_0$}
\psfrag{a1}{$a_1$}
\psfrag{a2}{$a_2$}
\psfrag{a3}{$a_3$}
\psfrag{b1}{$b_1$}
\psfrag{b2}{$b_2$}
\psfrag{b3}{$b_3$}
\includegraphics[scale=0.3]{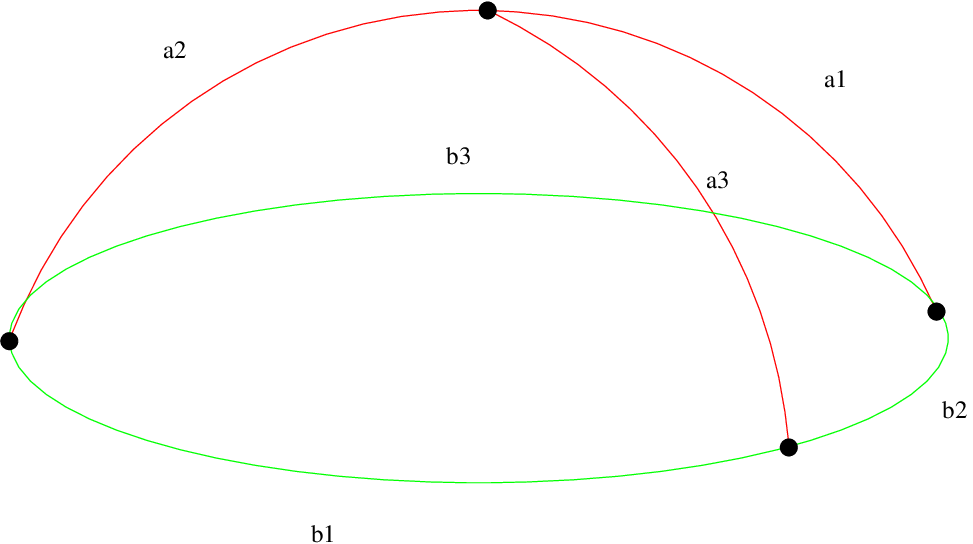}

\caption*{\wichtig{\mbox{Quadruple comparison: $\sum a_i^2 \geq \frac{1}{3} \sum b_i^2$}}}
\end{center}
		\end{subfigure}
\caption{\wichtig{Nonnegative curvature}}
\end{figure} 
\begin{proof}
(i)
	According
to Theorem \ref{thmoptcoupl}, we may assume that the geodesic is given as $\X_t=[X,\d_t,\ol\m]$
	with $X=X_0\times X_1$, $\d_t=(1-t)\d_0+t\d_1$, and some $\ol\m\in \Opt(\m_0,\m_1)$.
	
	Let $\X'=[X',\d',\m']$ and for fixed $t\in[0,1]$, let $\hat{\m}\in\Cpl(\ol\m,\m')$ be a
	coupling
which minimizes
\[\int \int \left|\d_t( x, y)-\d'(x',y')\right|^2 d\hat{\m}( x,x') d\hat{\m}( y,y').\]
In other words, $\hat\m$ is a probability measure on $\hat{X}=X\times X'$ which couples $\ol\m$ and $\m'$ in an optimal way  w.r.t. $\DD$.
 Then
	\begin{align*}
		& \DD^2(\X_t,\X') + t(1-t) \DD^2(\X_0,\X_1)\\
		=& \int_{\hat{X}} \int_{\hat{X}} \left|\d_t(x,y)-\d'(x',y')\right|^2 d\hat{\m}(x,x')d\hat{\m}
		(y,y') \
		 +\ t (1-t) \int_{X} \left|\d_0(x,y)-\d_1(x,y)\right|^2 d\ol\m(x)d\ol\m(y)\\
		= &\int_{\hat{X}} \int_{\hat{X}}
			\left[ \left| (1-t)\d_0(x,y)+t \d_1(x,y) -\d'(x',y')\right|^2
					+  t (1-t)  \left|\d_0(x,y)-\d_1(x,y)\right|^2
			\right]
			d\hat{\m}(x,x')d\hat{\m}(y,y')\\
		= & \int_{\hat{X}} \int_{\hat{X}}
			\left[  (1-t) \left|\d_0(x_0,y_0) -\d'(x',y')\right|^2
					+  t   \left|\d_1(x_1,y_1)-\d'(x',y')\right|^2
			\right]
			d\hat{\m}(x_0,x_1,x')d\hat{\m}(y_0,y_1,y')\\
		\geq {} & (1-t) \DD^2(\X_0,\X') + t\DD^2(\X_1,\X'),
	\end{align*}
	where the last inequality follows from the fact that $(\pi_0,\pi_2)_*\hat{\m}$ is a coupling of $\m_0$ and $\m'$ - but not necessarily an optimal one for $\DD$. Similarly, for $(\pi_1,\pi_2)_*\hat{\m}$ and $\m_1$, $\m'$.

\medskip

(ii) Given points $\X_0,\ldots, \X_3\in\XX$, choose $\ol\m_i\in\Opt(\m_0,\m_i)$ and define (according to Lemma \ref{sec-glue}) a measure $\mu$ on $X=X_0\times X_1\times X_2\times X_3$ by
\[d\mu(x_0,x_1,x_2,x_3)=d\ol\m_{1,x_0}(x_1)\,d\ol\m_{2,x_0}(x_2)\,d\ol\m_{3,x_0}(x_3)\,d\m_0(x_0)\]
where $d\ol\m_{i,x_0}(x_i)$ denotes the disintegration of $d\ol\m_i(x_0,x_i)$ w.r.t. $d\m_0(x_0)$.
Then
\begin{eqnarray*}
\sum_{i=1}^3\DD^2(\X_0,\X_i)&=&
\int_X\int_X\sum_{i=1}^3\big|\d_0(x_0,y_0)-\d_i(x_i,y_i)\big|^2d\mu(x)\,d\mu(y)\\
&\ge&
\int_X\int_X\frac13\sum_{1\leq i<j\leq 3}\big|\d_i(x_i,y_i)-\d_j(x_j,y_j)\big|^2d\mu(x)\,d\mu(y)\\
&\ge& \frac13\sum_{1\leq i<j\leq 3}\DD^2(\X_i,\X_j).
\end{eqnarray*}
The last inequality here comes from the fact that for all $i,j\in\{1,2,3\}$
\[(\pi_i,\pi_j)_*\mu\in\Cpl(\m_i,\m_j)\]
but is not necessarily optimal.
The first inequality follows from the quadruple inequality in the metric space $(\R^1,|.|)$ applied to the 4 points
$\xi_i=\d_i(x_i,y_i)$, $i=0,1,2,3$, for each fixed pair $(x,y)\in X^2$.
\end{proof}

\begin{corollary}\label{XX-compl}
The metric completion	$(\ol\XX,\DD)$ of $(\XX,\DD)$ is a complete length space of nonnegative curvature in the sense of Alexandrov.

Obviously, also $\ol\XX$ is a cone over its unit sphere $\ol\XX^1$ (which is the completion of $\XX^1$).
\end{corollary}

\begin{proof}
The quadruple inequality immediately carries over to the completion. According to \cite{lp}, for complete length spaces this characterizes nonnegative curvature in the sense of Alexandrov.
\end{proof}

\begin{corollary}\label{olisalex}	\begin{enumerate}
\item		 $(\ol\XX^1,\DD^{(1)})$ is a complete length space  with  curvature $\geq 1$
			in the sense of Alexandrov.
\item
 $(\XX^1,\DD^{(1)})$ is a geodesic space  with  curvature $\geq 1$: both the triangle and the quadruple comparison property are satisfied.
 \end{enumerate}
					\end{corollary}

\begin{proof}
(i) It is a well-known fact from  geometry of Alexandrov spaces, see e.g. \cite{bbi}, Thm. 10.2.3., that
 cone structure together with nonnegative curvature implies that the unit sphere has curvature $\ge1$.
This result immediately applies to the completion $\ol\XX$ and its unit sphere $\ol\XX^1$.

(ii)
The fact that $(\XX^1,\DD^{(1)})$ is a geodesic space follows from Theorem \ref{cone} (`cone structure') together with the fact that $(\XX,\DD)$ itself is a geodesic space.
The triangle and the quadruple inequality now both follow from (i) by applying it to points in $\XX^1$.
\end{proof}

\subsection{Space of Directions, Tangent Cone, and Gradients on $\ol\XX$}

According to the previous Corollary \ref{olisalex}, $(\ol\XX,\DD)$ is a complete length space of nonnegative curvature. Indeed, we will see in Theorem \ref{olisgeo} that $(\ol\XX,\DD)$ is even a geodesic space (not just a length space).
As consequences of general results on Alexandrov spaces this implies  a variety of  existence and structural results on tangent cones, exponential maps and gradients.
We present some of the basic concepts and results, following mainly \cite{Pl}. We formulate these definitions and assertions for the particular space $(\ol\XX,\DD)$. Actually, however, they will be true for arbitrary complete geodesic spaces of lower bounded curvature. The crucial point is that no (local) compactness is required.


The \emph{space of geodesic directions} at $\X_0$ -- denoted by $\mathring{T}_{\X_0}^1\ol\XX$ --
consists of equivalence classes of unit speed geodesics emanating from $\X_0$ where two such geodesics $\geod{0}{\tau}$ and $(\X'_t)_{0\leq t\leq \tau'}$ are regarded as equivalent if one of them is an extension of the other one, say e.g.\ $\tau'\geq \tau$ and
\[
	\X_t=\X'_t \quad \text{for } t\leq \tau.
\]	
The space of geodesic directions is a metric space with a metric $\measuredangle$ \index{$\measuredangle$}
given by
\[
	\measuredangle(\X_\bullet,\X'_\bullet)=\lim_{s,t\searrow 0} \arccos
			\left[ \frac{1}{2st}\left(s^2+t^2-\DD^2(\X_s,\X'_t)\right) \right].
\]
The limit always exists. Indeed, as a consequence of the curvature bound, the quantity $\arccos[.]$ in the above formula is non-increasing in $s$ and in $t$.
The \emph{space of directions} at $\X_0$\index{directions} -- denoted by ${T}_{\X_0}^1\ol\XX$\index{t@$\mathring{T}_{\X}^1$, ${T}_{\X}^1$, ${T}_{\X}$} -- is the completion of the space of geodesic directions at $\X_0$ w.r.t.\ the metric $\measuredangle$. The \emph{tangent cone}\index{tangent cone} $T_{\X_0}\ol\XX$ at $\X_0$ is the cone over the space of directions at $\X_0$.

\begin{definition} \begin{enumerate}
\item
Given a number $\lambda\in\R$, a function $\U:\ol\XX\rightarrow\R$ will be called \emph{$\lambda$-Lipschitz continuous} if
\[
	|\U(\X_0)-\U(\X_1)|\leq \lambda \cdot \DD(\X_0,\X_1)
\]
for all $\X_0,\X_1\in\ol\XX$. In this case, we briefly write $\Lip(\U)\leq \lambda$. The function $\U$ is called Lipschitz continuous if it is $\lambda'$-Lipschitz continuous for some $\lambda'$.
\item
Given a number $\kappa\in\R$, the function $\U:\ol\XX\rightarrow\R$ is called \emph{$\kappa$-convex}  if for all geodesics $\geod{0}{1}$ in $\ol\XX$ and for all $t\in[0,1]$,
\[
	\U(\X_t)\leq (1-t)\U(\X_0)+ t\U(\X_1)- \frac{\kappa}{2}t(1-t)\DD^2(\X_0,\X_1).
\]
(Note that if $\U$ is continuous, then the latter is equivalent to
$\frac{d^2}{dt^2}\U(\X_t)\geq\kappa\cdot \DD^2(\X_0,\X_1)$
in distributional sense on the interval $(0,1)$ for each given geodesic.)
The function $\U$ is called \emph{semiconvex} if it is $\kappa'$-convex for some $\kappa'$.
\item
The function $\U$ is called \emph{$\kappa$-concave} (or \emph{semiconcave}) if $-\U$ is $(-\kappa)$-convex (or semiconvex, resp.), that is, if
 $	\U(\X_t)\geq (1-t)\U(\X_0)+ t\U(\X_1)- \frac{\kappa}{2}t(1-t)\DD^2(\X_0,\X_1) $
for all geodesics $\geod{0}{1}$ in $\ol\XX$ and all $t\in[0,1]$.
\end{enumerate}
Note that  functions which we call $\kappa$-concave are called by some other authors $(-\kappa)$-concave.
The sign convention is not consistent in the literature.
\end{definition}

\begin{example}
The function $\X\mapsto -\DD^2(\X,\X_0)$ is $-2$-convex for each $\X_0$.
The same is true for the function
\[\X\mapsto \max\Big\{-\DD^2(\X,\X_i)\spec i=1,\ldots,k\Big\}\]
for any given set of points $\X_1,\ldots,\X_k\in\ol\XX$.
\end{example}

For every Lipschitz continuous, semiconcave function $\U:\ol\XX\to\R$ the \emph{`ascending slope'}  of $\U$ at $\X\in\ol\XX$ is
\[
	\aslope{\X}:=\limsup_{\X'\rightarrow \X} \frac{\left[ \U(\X')-\U(\X)\right]^+}
	{\DD(\X',\X).}
\]
A point $\X\in\ol\XX$ is called \emph{critical} for $\U$ if $\aslope{\X}=0$. The set $\ol\XX_\U$ of critical points for $\U$ is a closed subset of $\ol\XX$. Each local maximizer (as well as each local minimizer) is critical for $\U$.

For each geodesic direction $\Phi\in T_{\X_0}\ol\XX$, say $\Phi=(\X_t)_{0\leq t\leq \tau}$, the \emph{directional derivative} of $\U$ in direction $\Phi$\index{d@$D_\Phi$}
\[
	D_\Phi\U=\lim_{t\searrow 0} \frac{1}{t}[\U(\X_t)-\U(\X_0)]
\]
exists and depends continuously on $\Phi\in T_{\X_0}\ol\XX$ (and thus extends to all of $T_{\X_0}\ol\XX$).

\begin{lemma}
	For every Lipschitz continuous, semiconcave function $\U$ on $\ol\XX$ and each point $\X\in \ol\XX$:
		\begin{enumerate}
		\item
			$\aslope{\X}=\sup \left\{ D_\Phi\U \spec \Phi \in T_\X \ol\XX,\, \|\Phi \|_{T_{\X}\ol\XX}=1
			\right\}$
		\item
			If $\aslope{\X}\neq 0$ then there exists a unique unit vector $\Phi\in T_\X \ol\XX$ such that
			\begin{equation}\label{uniquephi}
				\aslope{\X}=D_\Phi\U.
			\end{equation}
	\end{enumerate}
\end{lemma}

	The \emph{gradient} of $\U$ at $\X\in \ol\XX$, denoted by $\nabla\U(\X)$ or more precisely by $\nabla^{\ol\XX}\U(\X)$,\index{g@$\nabla$} is now defined as an element in $T_\X \ol\XX$ as follows:
	\begin{itemize}
		\item
			if $\X$ is critical for $\U$, put $\nabla\U(\X)=0$,
		\item
			otherwise, put $\nabla\U(\X):=t\Phi$ where $\Phi\in T_\X \ol\XX$ is the unique unit
			tangent vector satisfying \eqref{uniquephi} and $t:=\aslope{\X}$.
	\end{itemize}
Note that by construction,
\[
	\left\Vert \nabla\U(\X) \right\Vert_{T_{\X}\ol\XX} = \aslope{\X}.
\]

\subsection{Gradient Flows on $\ol\XX$}

\begin{definition}
	A curve $\X_\bullet:[0,L)\rightarrow \ol\XX$ (with $L\in(0,\infty]$) is called \emph{ascending gradient curve
	of $\U$} or solution of the (\emph{`upward gradient flow'}) differential equation
	\[
		\dot{\X}_t= \nabla\U(\X_t)
	\]
if for all $t\in [0,L)$:
	\begin{align}
		\lim_{s\searrow 0}  \frac{1}{s} \DD(\X_{t+s},\X_t)& =\aslope{\X_t}\label{grad flow equ1}
		\intertext{and}
		\lim_{s\searrow 0} \frac{1}{s} [\U(\X_{t+s})-\U(\X_t)]&=\aslope{\X_t}^2.\label{grad flow equ2}
	\end{align}
\end{definition}

\begin{theorem}Let $\U:\ol\XX\to\R$ be  Lipschitz continuous and $\kappa$-concave.
\begin{enumerate}
\item
	Then for each $\X_0\in\ol\XX$ there exists a unique ascending gradient curve
	$(\X_t)_{0\leq t<\infty}$ of $\U$.
\item
For all $\X_0,\X_0'\in\ol\XX$ and every $t>0$
\[\DD\Big(\X_t,\X_t'\Big)\le e^{\kappa\, t}\cdot \DD\Big(\X_0,\X_0'\Big).\]
\end{enumerate}
\end{theorem}
The uniqueness in particular implies that $\X_t=\X_\tau$ for all $t\ge\tau$ where $\tau=\inf\{s\ge0\spec \X_s\in\ol\XX_\U\}$.

\begin{proof}
If $\X_0\in\ol\XX_\U$, then one possible solution to the gradient flow equation (as defined above) is always given by
\[\X_t=\X_0\quad (\forall t\ge0).\]
For  $\X_0\not\in\ol\XX_\U$, Plaut \cite{Pl} as well as Lytchak \cite{Ly}, based on (unpublished) previous work of Perelman and Petrunin \cite{pp}, proved the existence of gradient flow curves. (The concept of gradient-like curves used in \cite{Pl} leads to re-parametrizations of gradient flow curves -- at least as long as they do not hit  the closed set  $\ol\XX_\U$.) The crucial point is that this existence result does not require any compactness of the underlying space $\ol\XX$.
The uniqueness result and exponential Lipschitz bound is taken from
 \cite{Ly}.
\end{proof}

\begin{remark}
In analysis (PDEs, mathematical physics), instead of the upward gradient flow mostly the \emph{downward gradient flow} for a given function $\U$ on $\ol\XX$ is considered.
\[
		\dot{\X}_t= \nabla(-\U)(\X_t).
	\]
It is just the upward gradient flow for $-\U$.
(Note that in metric geometry we have to distinguish between $\nabla(-\U)$ and $-\nabla\U$.)

This requires the function $\U$ now to be semiconvex. The relevant quantity then is the \emph{descending slope}
\[ |D^-\U(\X)|= |D^+(-\U)(\X)|=\limsup_{\X'\rightarrow \X} \frac{\left[ \U(\X)-\U(\X')\right]^+}
	{\DD(\X',\X).}\]
\end{remark}

\section{The Space $\YY$ of Gauged Measure Spaces}
\subsection{Gauged Measure Spaces}
In order to analyze and characterize elements $\X$ in the completion $\ol\XX$ of the space of mm-spaces, and to obtain a more explicit representation of tangent spaces $T_\X$ and exponential maps $\Expf{}$, we embed the space of mm-spaces into a bigger space $\YY$ which in the sense of Alexandrov geometry is more regular. (In particular, it has less boundary.)

\begin{definition}\label{def-homo}
A \emph{gauged measure space} is a triple $(X,\f,\m)$ consisting of a Polish space $X$, a Borel probability measure $\m$ on $X$, and a function $\f\in L^2_s(X^2,\m^2)$. The latter denotes the space of symmetric functions $f$ on $X\times X$ which are square integrable w.r.t. the product measure $\m \otimes \m$.
Any such function $\f$ is called \emph{gauge}.\index{gauge function}
\end{definition}
This extends the concept of \emph{metric measure spaces} in two respects: i) the function $\f$ replacing the distance $\d$ is no longer requested to satisfy the triangle inequality; ii) even if it did so, it is no longer requested to induce the (Polish) topology on $X$. Metric (or gauge) and topology are decoupled to the greatest possible extent. The only remaining constraint is that $\f$ should be measurable w.r.t. the $\sigma$-field induced by the topology. To abandon the triangle inequality will make the space of all gauged measure spaces `more linear'.

The size of a gauged measure is simply defined as the $L^2$-norm of its gauge function, i.e.
\[\size\big( X,\f,\m\big)=\Big(\int_{X}\int_X \f^2(x,y) d\m(x)\,d\m(y)\Big)^{1/2}.\]
\begin{definition}
\begin{enumerate}
\item
The $L^2$-distortion distance between two gauged measure spaces $(X_0,\f_0,\m_0)$ and $(X_1,\f_1,\m_1)$ is defined by
\begin{eqnarray*}
	\lefteqn{\DD\Big((X_0,\f_0,\m_0),(X_1,\f_1,\m_1)\Big)}\\
&=& \inf \Bigg\{  \bigg( \int_{X_0\times X_1} \int_{X_0\times X_1}
		\left| \f_0(x_0,y_0)-\f_1(x_1,y_1) \right|^2 d\ol{\m}
		(x_0,x_1)d\ol{\m}(y_0,y_1)\bigg)^{1/2} \spec \ol{\m}\in \Cpl(\m_0,\m_1)
		\Bigg\}.
	\end{eqnarray*}
\item
Every minimizer $\ol\m$ of the above RHS will be called \emph{optimal coupling} of the given gauged measure spaces. In other words, a coupling $\ol{\m}\in \Cpl(\m_0,\m_1)$ is optimal if
\begin{eqnarray*}
	\DD\Big((X_0,\f_0,\m_0),(X_1,\f_1,\m_1)\Big)=  \bigg( \int \int
		\left| \f_0-\f_1 \right|^2 d\ol{\m}\,
		d\ol{\m}\bigg)^{1/2}.
	\end{eqnarray*}
\item
Two gauged measure spaces $(X_0,\f_0,\m_0)$ and $(X_1,\f_1,\m_1)$ are called \emph{homomorphic} if \[\DD\Big((X_0,\f_0,\m_0),(X_1,\f_1,\m_1)\Big)=0.\] Obviously, this defines an equivalence relation.
\end{enumerate}
\end{definition}

\begin{lemma}\label{obdA}
\begin{enumerate}
\item
Every gauged measure space $(X,\f,\m)$ is homomorphic to the space $(I,\f',\Leb^1)$ for a suitable
$\f'\in L^2_s(I^2,\Leb^2)$.
Indeed, one may choose $\f'=\psi^*\f$ for any $\psi\in\Par(\m)$.
\item
An optimal coupling of $(X,\f,\m)$ and $(I,\psi^*\f,\Leb^1)$ is given by
$(\psi, \Id)_*\Leb^1$.
\end{enumerate}
\end{lemma}

\begin{proof}
Define $\ol\m=(\psi, \Id)_*\Leb^1$ for $\psi\in\Par(\m)$. Then obviously $\ol\m$ is a coupling of $\m=\psi_*\Leb^1$ and $\Leb^1$.
Moreover,
\begin{eqnarray*}
\int_{X\times I} \int_{X\times I}
		\left| \f-\f' \right|^2 d\ol{\m}\,
		d\ol{\m}
=
\int_{I} \int_{ I}
		\left| \psi^*\f-\f' \right|^2 d\Leb^1\,
		d\Leb^1
=0
	\end{eqnarray*}
according to our choice $\f'=\psi^*\f$.
\end{proof}

\begin{proposition}\label{exist opt}
For every pair of gauged measure spaces
$(X_0,\f_0,\m_0)$ and $(X_1,\f_1,\m_1)$ there exists an optimal coupling, i.e. a measure $\ol\m\in\Cpl(\m_0,\m_1)$ which realizes the $L^2$-distortion distance.
\end{proposition}

\begin{proof}
(i) Let us first prove the claim in the particular case $X_0=X_1=I$ and $\m_0=\m_1=\Leb^1$.
As before in the proof of Lemma \ref{optcoupl1} the claim will follow from compactness of the set $\Cpl(\m_0,\m_1)$ and lower semicontinuity of the functional $\m\mapsto \big(\int \int
		\big| \f_0-\f_1 \big|^2 d{\m}\,
		d{\m}\big)^{1/2}$ on $\Cpl(\m_0,\m_1)$.
The former remains true in this more degenerate setting, i.e. Lemma \ref{cpl-comp} applies without any change.
The latter requires more care
and will be the content of the next lemma.

(ii) The case of general $X_0,X_1$ and $\m_0,\m_1$ can be reduced to the previous case as follows.
Choose $\psi_i\in\Par(\m_i)$ for $i=0,1$ and put $\f_i'={\psi_i}^*\f_i$. Apply the previous part (i) to deduce the existence of a coupling $\tilde\m\in\Cpl(\Leb^1,\Leb^1)$ which minimizes
\[
\bigg(\int_{I^2} \int_{I^2}
		\big| \f_0'(x_0,y_0)-\f_1'(x_1,y_1) \big|^2 d\tilde\m(x_0,x_1)\,
		d\tilde\m(y_0,y_1)\bigg)^{1/2}.
\]
Put $\ol\m=(\psi_0,\psi_1)_*\tilde\m$. This defines a coupling of $\m_0$ and $\m_1$ and satisfies
\begin{eqnarray*}
\lefteqn{
\bigg(\int_{X_0\times X_1} \int_{X_0\times X_1}
		\big| \f_0-\f_1 \big|^2 d\ol\m\,
		d\ol\m\bigg)^{1/2}}\\
&=&
\bigg(\int_{I^2} \int_{I^2}
		\big| \psi_0^*\f_0-\psi_1^*\f_1 \big|^2 d\tilde\m\,
		d\tilde\m\bigg)^{1/2}\\
&=&\DD\Big((I,\f_0',\Leb^1),(I,\f_1',\Leb^1)\Big)\\
&\le&
\DD\Big((I,\f_0',\Leb^1),(X_0,\f_0,\m_0)\Big)+
\DD\Big((X_0,\f_0,\m_0),(X_1,\f_1,\m_1)\Big)
+\DD\Big((X_1,\f_1,\m_1),(I,\f_1',\Leb^1)\Big)\\
&=&\DD\Big((X_0,\f_0,\m_0),(X_1,\f_1,\m_1)\Big)
\end{eqnarray*}
according to the previous lemma.
This proves the optimality of $\ol\m$.
\end{proof}

\begin{lemma}\label{approx-lsc}
Given two functions $\f_0,\f_1\in L^2_s(I^2,\Leb^2)$, the
functional
\[\m\mapsto \Xi(\m):=\bigg(\int_{I^2} \int_{I^2}
		\big| \f_0(x_0,y_0)-\f_1(x_1,y_1) \big|^2 d{\m}(x_0,x_1)\,
		d{\m}(y_0,y_1)\bigg)^{1/2}\]
is continuous on $\Cpl(\Leb^1,\Leb^1)$, the latter being regarded as a subset of $\mathcal P(I^2)$ equipped with the topology of weak convergence.
\end{lemma}

\begin{proof}
Every $\f\in L^2_s(I^2,\Leb^2)$ can be approximated in $L^2$-norm by continuous symmetric functions on $I^2$. (Just apply the heat kernel or any mollifier to $\f$, see e.g. the construction in the proof of Theorem~\ref{hatXXeqolXX}.) Thus there exist $\f_{i,n}\in L^2_s(I^2,\Leb^2)\cap \mathcal C(I^2)$ for $i=0,1$ and $n\in\N$ such that
\[
\bigg(\int_I\int_I \big|\f_i(s,t)-\f_{i,n}(s,t)\big|^2ds\,dt\bigg)^{1/2}\le \frac1{n}.\]
For each $n\in\N$ the functional
\[\m\mapsto \Xi_n(\m):=\bigg(\int_{I^2} \int_{I^2}
		\big| \f_{0,n}(x_0,y_0)-\f_{1,n}(x_1,y_1) \big|^2 d{\m}(x_0,x_1)\,
		d{\m}(y_0,y_1)\bigg)^{1/2}\]
is  continuous on $\Cpl(\Leb^1,\Leb^1)$ due to the fact that the integrand $|\f_{0,n}-\f_{1,n}|^2$ is continuous and bounded on $I^2\times I^2$. Moreover, by a simple application of the triangle inequality in $L^2(I^2\times I^2)$,
\begin{eqnarray*}
\Big|\Xi(\m)-\Xi_n(\m)\Big|&\le&
\bigg(\int_{I^2} \int_{I^2}
		\big| \f_{0}(x_0,y_0)-\f_{0,n}(x_0,y_0) \big|^2 d{\m}(x_0,x_1)\,
		d{\m}(y_0,y_1)\bigg)^{1/2}\\
&&+
\bigg(\int_{I^2} \int_{I^2}
		\big| \f_{1,n}(x_1,y_1)-\f_{1}(x_1,y_1) \big|^2 d{\m}(x_0,x_1)\,
		d{\m}(y_0,y_1)\bigg)^{1/2}\\
&=&\bigg(\int_I\int_I \big|\f_0(s,t)-\f_{0,n}(s,t)\big|^2ds\,dt\bigg)^{1/2}
+
\bigg(\int_I\int_I \big|\f_{1,n}(s,t)-\f_{1}(s,t)\big|^2ds\,dt\bigg)^{1/2}\\
&\le& \frac2n
\end{eqnarray*}
for each $n\in\N$. This proves the continuity of
$\m\mapsto \Xi(\m)$ on $\Cpl(\Leb^1,\Leb^1)$.
\end{proof}

\begin{proposition}
 For any pair of gauged measure spaces $(X_0,\f_0,\m_0)$ and $(X_1,\f_1,\m_1)$,
\begin{eqnarray*}
\DD\Big((X_0,\f_0,\m_0),(X_1,\f_1,\m_1)\Big)=0 &\Longleftrightarrow &
\exists 
(X,\f,\m), \ \exists \psi_i:X\to X_i \text{ measurable s.t. }\\
 &&(\psi_i)_*\m=\m_i,\  (\psi_i)^*\f_i=\f \quad(\forall i=0,1).
\end{eqnarray*}
In particular,
\begin{eqnarray*}
\DD\Big((X_0,\f_0,\m_0),(X_1,\f_1,\m_1)\Big)=0 \quad\Longleftarrow \quad
\exists \psi:X_0\to X_1 \text{ measurable s.t. } \psi_*\m_0=\m_1,\  \psi^*\f_1=\f_0.
\end{eqnarray*}
\end{proposition}
Here and in the sequel, identities like $(\psi_i)^*\f_i=\f$ or $\psi^*\f_1=\f_0$ have to be understood as equalities $\m^2$-a.e. on $X^2$ or
$\m_0^2$-a.e. on $X_0^2$, resp.
\begin{proof}
Assume the existence of the space $(X,\f,\m)$ and the maps $\psi_0,\psi_1$  with given properties. Put $\ol\m=(\psi_0,\psi_1)_*\m$. Obviously, this is an element of $\Cpl(\m_0,\m_1)$ satisfying
\[ \int_{X_0\times X_1}\int_{X_0\times X_1} \Big|\f_0-\f_1\Big|^2 d\ol\m\,d\ol\m
=\int_{X}\int_{X}\Big|\psi_0^*\f_0-\psi_1^*\f_1\Big|^2 d\m\,d\m=0.\]
Now, conversely, assume that $\DD(.,.)=0$. Then according to Proposition \ref{exist opt} there exist
$\ol\m\in\Cpl(\m_0,\m_1)$ with $\int\int|\f_0-\f_1|^2d\ol\m\, d\ol\m=0$.
 Then
\[\f_0(x_0,y_0)=\f_1(x_1,y_1)\quad \text{for }\ol\m^2\text{-a.e. }\big((x_0,x_1),(y_0,y_1)\Big)\in X^2\]
for
$X:=X_0\times X_1$. Thus $(X,\f,\ol\m)$ with $\f:=\frac12\f_0+\frac12\f_1$ will do the job together with $\psi_i=\pi_i: X\to X_i$ being the projections ($i=0,1$).
\end{proof}

\begin{remarks}
\begin{enumerate}
\item
If $\m_0$ has atoms and $\m_1$ has no atoms then there exists no map $\psi:X_0 \to X_1$ with $\psi_*\m_0=\m_1$.
\item
For each gauged measure space $(X_0,\f_0,\m_0)$ there exist gauged measure spaces $(X_1,\f_1,\m_1)$ without atoms and with
$\DD\Big((X_0,\f_0,\m_0),(X_1,\f_1,\m_1)\Big)=0$. This follows from Lemma~\ref{obdA}.

    \end{enumerate}
\end{remarks}

Equivalence classes of homomorphic gauged measure spaces will  be denoted by
\[\X_0= \auf X_0,\f_0,\m_0\zu, \quad \X_1= \auf X_1,\f_1,\m_1\zu,\quad \X'=\auf X',\f',\m'\zu\quad \text{etc.}\]
and their respective representatives as before by
$(X_0,\f_0,\m_0)$, $(X_1,\f_1,\m_1)$, $(X',\f',\m')$ etc.
The space of equivalence classes of homomorphic gauged measure spaces will be denoted by $\YY$.
\index{homomorphism class $\auf X,\f,\m\zu$}
\index{y@$\YY$}
\index{space of gauged measure spaces}

\begin{theorem}\label{YY-geo}
			$(\YY,\DD)$ is a complete geodesic space of nonnegative curvature in the sense of Alexandrov.
More specifically, the following assertions hold:
\begin{enumerate}
\item For each pair of gauged measure spaces $(X_0,\f_0,\m_0)$ and $(X_1,\f_1,\m_1)$, there exists an optimal coupling $\ol\m\in\Cpl(\m_0,\m_1)$.
\item For each choice of optimal coupling $\ol\m\in\Cpl(\m_0,\m_1)$, a geodesic in $\YY$ connecting
$\auf X_0,\f_0,\m_0\zu$ and $\auf X_1,\f_1,\m_1\zu$ is given by
\begin{equation}\label{Y geod}
\X_t=\auf X_0\times X_1,(1-t)\f_0+t\f_1,\ol\m\zu, \qquad t\in(0,1).
\end{equation}
\item Every geodesic $(\X_t)_{t\in[0,1]}$ in $\YY$ is of this form. That is, given representatives $(X_0,\f_0,\m_0)$ and $(X_1,\f_1,\m_1)$ of the endpoints of the geodesic, there exists an optimal coupling $\ol\m\in\Cpl(\m_0,\m_1)$ defined on $X_0\times X_1$ such that (\ref{Y geod}) holds.
\item $(\YY,\DD)$ satisfies the triangle comparison and the quadruple comparison properties.
\item $(\YY,\DD)$ is a cone over its unit sphere
\[\YY^1=\{ \X\in\YY\spec \size(\X)=1\}.\]
\item $\YY^1$ with the induced distance $\DD^{(1)}$ is a complete geodesic space with curvature $\ge 1$ in the sense of Alexandrov.
\end{enumerate}
\end{theorem}

\begin{proof}
\begin{itemize}
\item
Obviously, $(\YY,\DD)$ is a \emph{metric} space. (Same proof as for Lemma \ref{dildis1}.)
\item The existence of optimal couplings was already stated as Proposition \ref{exist opt}.
The assertions on existence and uniqueness of \emph{geodesics} thus follow exactly as in
Theorem \ref{thmoptcoupl}. None of the arguments used in the proof required that $\d$ is continuous or satisfies the triangle inequality.
\item
The proof of the \emph{cone} property from Theorem \ref{cone} applies without any change.
\item All assertions on \emph{curvature} bounds for $\YY$ and $\YY^1$ follow with exactly the same arguments as for
$\XX$ and $\XX^1$, see Theorem \ref{thmXX} and Corollary \ref{olisalex}.
\item It remains to prove the \emph{completeness} of $(\YY,\DD)$:

Let a sequence of gauged measure spaces $(X_n,\f_n,\m_n)$, $n\in\N$, be given with
\[\DD\Big((X_n,\f_n,\m_n),(X_k,\f_k,\m_k)\Big)\to0\quad\text{as }k,n\to\infty.\]
Passing to a subsequence if necessary, we may assume that
   \[\DD\Big((X_n,\f_n,\m_n),(X_{n+1},\f_{n+1},\m_{n+1})\Big)\le 2^{-n}\]
for all $n\in\N$ which (according to Proposition \ref{exist opt}) implies the existence of a coupling
$\mu_n\in\Cpl(\m_n,\m_{n+1})$ satisfying
\begin{equation}\label{f,n,n+1}
\left( \int_{X_n\times X_{n+1}}\int_{X_n\times X_{n+1}}\Big|\f_n-\f_{n+1}\Big|^2d\mu_n\,d\mu_n\right)^{1/2}
\le 2^{-n}.
\end{equation}
Gluing together all these measures for $n=1,\ldots,N-1$ yields a measure
\[\hat\mu_N=\mu_1\boxtimes\ldots\boxtimes\mu_{N-1}\quad\text{on}\quad \hat{X}_N=\prod_{n=1}^{N} X_n.\]
For $N\to\infty$, the projective limit
\[\hat\mu=\lim_{\longleftarrow}\hat\mu_N\]
of these measures is a probability measure on $\hat X=\prod_{n=1}^{\infty} X_n$ with the property
\[(\pi_n,\pi_{n+1})_*\hat\mu=\mu_n\]
for each $n\in\N$.
Define functions $\hat\f_n\in L^2_s({\hat X}^2,\hat\mu^2)$ by
\[\hat\f_n(x,y)=\f_n(x_n,y_n)\]
for $x=(x_i)_{i\in\N}, y=(y_i)_{i\in\N}\in \hat X$.
Then
\[ \|\hat\f_n-\hat\f_{n+1}\|_{L^2_s({\hat X}^2,\hat\mu^2)}=\DD\Big((X_n,\f_n,\m_n),(X_{n+1},\f_{n+1},\m_{n+1})\Big)\le 2^{-n}\]
for all $n\in\N$. Therefore, $(\hat\f_n)_n$ is a Cauchy sequence in the Hilbert space
$L^2_s({\hat X}^2,\hat\mu^2)$ and thus there exists
$\hat\f\in L^2_s({\hat X}^2,\hat\mu^2)$
with
\[\|\hat\f_n-\hat\f\|_{L^2_s({\hat X}^2,\hat\mu^2)}\to0.\]
The triple $(\hat X,\hat\f,\hat\mu)$ is the gauged measure space we are looking for. Indeed,
  \[\DD\Big((X_n,\f_n,\m_n),(\hat X,\hat\f,\hat\mu)\Big)\le
  \|\hat\f_n-\hat\f\|_{L^2_s({\hat X}^2,\hat\mu^2)}\to0.\]
This proves the claim.\end{itemize}
\noindent
\end{proof}

\subsection{Equivalence Classes in $L^2_s(I^2,\Leb^2)$}

The space $\YY$ admits a remarkable and very instructive representation in terms of parametrizations.
For this purpose, let us consider the \emph{semigroup}
$\Inv(I,\Leb^1)$ \index{i@$\Inv(.,.)$}
of all  Borel measurable maps $\phi: I\to I$ which leave $\Leb^1$ invariant, i.e.
which satisfy $\phi_*\Leb^1=\Leb^1$.
This semigroup, call it $G$ for the moment, acts on the linear space $H=L^2_s(I^2,\Leb^2)$ via pull back
\begin{eqnarray*}
G\times H&\to& H\\
(\phi,f)&\mapsto& \phi^*f
\end{eqnarray*}
with $\big(\phi^*f\big)(s,t)=f\big(\phi(s),\phi(t)\big)$.
\begin{lemma}
$G$ acts isometrically on $H$.
\end{lemma}
\begin{proof}
\begin{eqnarray*}
\|\phi^*f\|_H^2=\int_0^1\int_0^1 \Big| f\big(\phi(s),\phi(t)\big)\Big|^2ds\,dt\stackrel{(\ast)}=
\int_0^1\int_0^1 \Big| f\big(s,t\big)\Big|^2ds\,dt=\|f\|_H^2
\end{eqnarray*}
where $(\ast)$ holds due to the $\Leb^1$-invariance of $\phi$.
\end{proof}

The semigroup $G$ induces an equivalence relation $\simeq$ in $H$:
\[f\simeq g\Longleftrightarrow \exists \phi,\psi\in G\spec \phi^*f=\psi^*g.\]
The set of equivalence classes for this relation $\simeq$  will be called \emph{quotient space} and denoted by
\[\LL=H/G=L^2_s(I^2,\Leb^2)/\Inv.\]
\index{l@$\LL$}
It is a pseudo metric space with pseudo metric
$d_\LL=d_{H/G}=d_{L^2/\Inv}$ given by
\begin{eqnarray*}
d_{H/G}(\auf f\zu,\auf g\zu)&=&\inf\Big\{\|f'-g'\|_H\spec f'\in\auf f\zu, g'\in\auf g\zu \Big\}\\
&=&\inf\Big\{\|\phi^*f-\psi^*g\|_H\spec \phi,\psi\in G\Big\}.
\end{eqnarray*}
Here  $\auf f\zu$ and $\auf g\zu$ denote the equivalence classes of $f,g\in H$.

\begin{theorem}\label{YY-iso}
	\begin{enumerate}
\item $(\LL, d_\LL)$ is a  metric space.
				\item\label{YY2}
			The metric spaces
\[(\LL, d_\LL)\qquad\text{and}\qquad (\YY,\DD)\]
are isometric. An isometry is given by
\[\Theta: \quad \begin{array}{ccc}
L^2_s(I^2,\Leb^2)/\Inv&\to& \YY\\
{{\auf}f{\zu}} &\mapsto& \auf I,f,\Leb^1\zu .
\end{array}
\]
The inverse map $\Theta^{-1}$ assigns to each representative $(X,\f,\m)$ of a gauged measure space $\auf X,\f,\m\zu \in\YY$ the function
$\f'=\psi^*\f\in L^2_s(I^2,\m^2)$ where $\psi$ is any element in $\Par(\m)$.
\item\label{YY1}
			$L^2_s(I^2,\Leb^2)/\Inv$ is a complete geodesic space of nonnegative curvature in the sense of Alexandrov.
	\end{enumerate}
\end{theorem}

\begin{proof}
(i), (ii) Let ${\auf}f{\zu}, {\auf}g{\zu}\in L^2_s(I^2,\Leb^2)/\Inv$ with representatives $f,g$ in $L^2_s(I^2,\Leb^2)$.
Then
\begin{eqnarray*}
d_{L^2/\Inv}({\auf}f{\zu},{\auf}g{\zu})&=&
\inf\Big\{\|\phi^*f-\psi^*g\|_{L^2}\spec \phi,\psi\in \Inv\Big\}\\
&\ge& \DD\bigg((I,f,\Leb^1),( I,g,\Leb^1) \bigg)=\DD\bigg(\auf I,f,\Leb^1\zu,\auf I,g,\Leb^1\zu \bigg)
\end{eqnarray*}
since each pair $(\phi,\psi)\in\Inv\times\Inv$ defines a coupling of $\Leb^1$ with itself via
$(\phi,\psi)_*\Leb^1$.

Conversely, given any coupling $\ol\m$ of $\Leb^1$ with itself, there exists $\phi\in\Par(\ol\m)$, i.e.
$\phi=(\phi_0,\phi_1): I\to I^2$ such that $\phi_*\Leb^1=\ol\m$. Thus
\begin{eqnarray*}
\int_{I^2}\int_{I^2}\big|f(x_0,y_0)-g(x_1,y_1)\big|^2d\ol\m(x_0,x_1)\,d\ol\m(y_0,y_1)
&=&
\int_I \int_I\big|f(\phi_0(s),\phi_0(t))-g(\phi_1(s),\phi_1(t))\big|^2ds\,dt\\
&=&
\|\phi_0^*f-\phi_1^*g\|^2_{L^2}
\end{eqnarray*}
with $\phi_0,\phi_1\in\Inv(I,\Leb^1)$.
Hence, $\DD\big((I,f,\Leb^1),(I,g,\Leb^1)\big)\ge d_{L^2/\Inv}({\auf}f{\zu},{\auf}g{\zu}).$

Nondegeneracy of $d_{L^2/\Inv}$ follows from Proposition \ref{exist opt}.
Indeed,
\[d_{L^2/\Inv}({\auf}f{\zu},{\auf}g{\zu})=0\] implies
$\DD\big((I,f,\Leb^1),(I,g,\Leb^1)\big)=0$
which in turn implies the existence of an optimal coupling $\ol\m$ with
$\int\int |f-g|^2\,d\ol\m\,d\ol\m=0$. Any such coupling $\ol\m$ can be represented as $(\phi,\psi)_*\Leb^1$
for suitable $\phi,\psi\in\Inv(I,\Leb^1)$. Thus
\[\phi^*f=\psi^*g.\]

It remains to prove that $\Theta$ is surjective.
This simply follows from the fact that for each gauged measure space $(X,\f,\m)$
there exists a parametrization $\psi\in\Par(\m)$ of its  measure and that the function
$\f'=\psi^*\f$ defined in terms of this parametrization lies in $L^2_s(I^2,\m^2)$.
Moreover, the gauged measure space $(I,\f',\Leb^1)$ will be homomorphic to the originally given $(X,\f,\m)$:
\begin{eqnarray*}
\DD\bigg((I,\f',\Leb^1),(X,\f,\m)\bigg)=0,
\end{eqnarray*}
see Lemma~\ref{obdA}.

(iii) All assertions follow immediately from (ii) together with the analogous statements of Theorem \ref{YY-geo}.
\end{proof}

\begin{remark}
If $\Inv(I,\Leb^1)$ was a group (instead just a semigroup) then assertion (iii) of the previous Theorem (together with all the assertions from Theorem \ref{YY-geo}) would be an immediate consequence of standard results in Alexandrov geometry. Indeed, if $H$ is a complete length space of nonnegative curvature and if $G$ is a group which acts isometrically on $H$ then the quotient  space
$H/G$ again is a length space of nonnegative curvature,
\cite{bbi}, Prop. 10.2.4.
\end{remark}

\subsection{Pseudo Metric Measure Spaces}\label{pseudo}

\begin{definition}\label{m2triangle} Given a gauged measure space $(X,\d,\m)$, we say that  the gauge $\d$ satisfies the \emph{triangle inequality $\m^2$-almost everywhere} if
there exists a Borel set $N\subset X^2$ with $\m^2(N)=0$ such that
\[\d(x_1,x_2)+\d(x_2,x_3)\ge\d(x_1,x_3)\]
for every $(x_1,x_2,x_3)\in X^3$ with $(x_i,x_j)\not\in N$ for all $\{i,j\}\subset\{1,2,3\}$.

Any such function $\d\in L^2_s(X^2,\m^2)$ will be called \emph{pseudo metric on $X$}.\index{pseudo metric} In particular, a pseudo metric is not required to be continuous but merely measurable on $X\times X$. And of course it may vanish also outside of the diagonal.
\end{definition}

\begin{remarks}\label{m2 vs m3}
\begin{enumerate}
\item
Any pseudo metric $\d$ is nonnegative $\m^2$-a.e. on $X^2$. Indeed, combining the estimates
$\d(x_1,x_3)\le\d(x_1,x_2)+\d(x_2,x_3)$ and $\d(x_2,x_3)\le\d(x_2,x_1)+\d(x_1,x_3)$ -- both valid for every $(x_1,x_2,x_3)\in X^3$ with $(x_i,x_j)\not\in N$ for all $\{i,j\}\subset\{1,2,3\}$
 -- yields
$\d(x_1,x_3)\le2\d(x_1,x_2)+\d(x_1,x_3)$ which proves the claim.

\item
The triangle inequality $\m^2$-almost everywhere (as defined above) obviously implies that the gauge function $\d$ satisfies the \emph{triangle inequality $\m^3$-almost everywhere} in the sense that
\[\d(x_1,x_2)+\d(x_2,x_3)\ge\d(x_1,x_3)\]
for $\m^3$-a.e. triple $(x_1,x_2,x_3)\in X^3$.
For the converse, see Corollary \ref{m2 = m3} below
where it is shown that the latter implies that the given gauged measure space is \emph{homomorphic} to a pseudo metric measure space (i.e. a gauged measure space which satisfies the $\m^2$-a.e.-triangle inequality).

See also recent work of Petrov, Vershik and Zatitskiy \cite{ZP11}, \cite{VPZ12} where it is shown that the validity of the $\m^3$-a.e.-triangle inequality for a \emph{separable} gauged measure space $(X,\d,\m)$ implies that
there exists a \emph{correction} of $\d$ which satisfies the triangle inequality \emph{everywhere} and coincides  $\m^2$-a.e. on $X^2$ with $\d$ (and thus in particular  $\d$ is a pseudo metric in our sense).
\end{enumerate}
\end{remarks}

\begin{lemma}\label{pseudo lemma}
\begin{enumerate}
\item
Let $(X,\d,\m)$ be a gauged measure space and $\psi\in\Par(\m)$ a parametrization. Then
\[\d \text{ is a pseudo metric on } X\qquad\Longleftrightarrow \quad
\psi^*\d \text{ is a pseudo metric on } I.\]
\item
Let $(X_k,\d_k,\m_k)$, $k\in\N$, be a sequence of gauged measure spaces with \[\DD\Big((X_k,\d_k,\m_k),(X_\infty,\d_\infty,\m_\infty)\Big)\longrightarrow0 \ \text{ as }k\to\infty\]
for some gauged measure space $(X_\infty,\d_\infty,\m_\infty)$.
If for each $k\in\N$, $\d_k$ is a pseudo metric on $X_k$ then $\d_\infty$ is a pseudo metric on $X_\infty$.
\end{enumerate}

\end{lemma}
\begin{proof}
(i) Assume that $\d$
satisfies the triangle inequality $\m^2$-a.e. with `exceptional set' $N\subset X^2$.
Put $\d'=\psi^*\d$ and $N'=(\psi,\psi)^{-1}(N)\subset I^2$.
Then
$\Leb^2(N')=\m^2(N)=0$ and $\d'$ satisfies the triangle inequality for every $(t_1,t_2,t_3)\in I^3$ with $(t_i,t_j)\not\in N'$ for all $\{i,j\}\subset\{1,2,3\}$.

Conversely, assume that
$\d'$ satisfies the triangle inequality $\Leb^2$-a.e. with
`exceptional set' $N'\subset I^2$. Put $M'=I^2\setminus N'$ and
\[M=(\psi,\psi)(M'),\quad N=X^2\setminus M=(\psi,\psi)(N').\]
Then $\Leb^2(M')=1$ and thus $m^2(M)=1$. Moreover,
$\d$
satisfies the triangle inequality
for every $(x_1,x_2,x_3)\in X^3$ with $(x_i,x_j)\in M$ for all $\{i,j\}\subset\{1,2,3\}$.

(ii)
Following the argumentation in the proof of Theorem \ref{YY-geo}
(completeness assertion), we may assume without restriction that $X_k=X_\infty$, $\m_k=\m_\infty$ for all $k\in\N$ and, moreover,
\[\|\d_k-\d_\infty\|_{L^2_s(X_\infty^2,\m_\infty^2)}\to0\]
as $k\to\infty$. Passing to a subsequence, the latter implies
\[ \d_k\to\d_\infty\quad\m_\infty^2\text{-a.e. on }X_\infty^2.\]
Thus the $\m_\infty^2$-a.e. triangle inequality carries over from $\d_k$ to $\d_\infty$.
\end{proof}

Applied to two gauged measure spaces $(X_0,\d_0,\m_0)$ and $(X_1,\d_1,\m_1)$ which are homomorphic,  i.e.
$\DD\big((X_0,\d_0,\m_0),(X_1,\d_1,\m_1)\big)=0$, the previous Lemma in particular implies that $\d_0$ satisfies the triangle inequality $\m_0^2$-almost everywhere if and only if
$\d_1$ satisfies the triangle inequality $\m_1^2$-almost everywhere.
Thus the `almost everywhere triangle inequality' is a property of homomorphism classes.

\begin{definition} A (homomorphism class of)  gauged measure space(s) $\X=\auf X,\d,\m\zu $ is called \emph{pseudo metric measure space}
if the gauge $\d$ satisfies the triangle inequality $\m^2$-almost everywhere.

The space of homomorphism classes of pseudo metric measure spaces is denoted by
$\hat\XX$.
\end{definition}\index{x@$\hat\XX$}
\index{space of pseudo metric measure spaces}

\begin{corollary}\label{closedness of tr-in}
The space $\hat\XX$ of pseudo metric measure spaces is a closed, convex subset of $\YY$. It contains the space $\XX$ of metric measure spaces and its closure $\ol\XX$.
\end{corollary}

\begin{proof}
Closedness of $\hat\XX$ follows from part (ii) of the previous Lemma. Since it obviously contains $\XX$ it therefore also contains $\ol\XX$.

To see the convexity, let a geodesic $(\X_t)_{0\le t\le1}$ in $\YY$ be given.
It is always of the form
\[ \X_t=\auf X_0\times X_1, (1-t)\d_0+t\d_1,\ol\m\zu.\]
Thus if the endpoints lie in $\hat\XX$, the gauges $\d_0$ and $\d_1$ satisfy the triangle inequality on $X_0\times X_1$ with suitable exceptional sets $N_0, N_1$ of vanishing $\ol\m^2$-measure. But then also the convex combinations of $\d_0$ and $\d_1$ satisfy the triangle inequality with exceptional set $N_0\cup N_1$.
\end{proof}

\begin{lemma}\label{discrete-sphere}
\begin{enumerate}
\item
Let $(X,\m)$ and $(X',\m')$ be arbitrary standard Borel spaces without atoms (i.e. $X$ is a Polish space and $\m$ a probability measure on $\borel(X)$ with $\m(\{x\})=0$ for all $x\in X$; similarly $X'$ and $\m'$). Equip $X$ as well as $X'$ with the \emph{discrete metric}
\[ \d(x,y)=\d'(x,y)=\left\{ \begin{array}{cl}
0,\ &x=y\\
1,\ &\text{else}
\end{array}\right.
\]
Then $(X,\d,\m)$  and $(X',\d',\m')$ are homomorphic. The equivalence class $\auf  X,\d,\m\zu $ will be called  \emph{the discrete continuum}.\index{discrete continuum}
\item  The pseudo metric measure space
$\X=\auf X,\d,\m\zu $ from (i) is the limit of the sequence of metric measure spaces $\X_n=[X_n,\d_n,\m_n]$, $n\in\N$, considered in Example \ref{ex-incompl}. More precisely,
\[\DD(\X_n,\X)\le2^{-n/2}\quad\text{ for all }n\in\N.\]

\item For each $n\in\N$, the geodesic $(\X_{n,t})_{0\le t\le 1}$ connecting $\X_n=\X_{n,0}$ and $\X=\X_{n,1}$ \emph{instantaneously} leaves the set $\XX$. That is, for each $t>0$,
    \[\X_{n,t}\not\in \XX.\]
\end{enumerate}
\end{lemma}

\begin{proof}
(i) For \emph{every} coupling $\ol\m\in\Cpl(\m,\m')$
\begin{eqnarray*}
\int\int\Big|\d-\d'\Big|^2d\ol\m\,d\ol\m&=&
\int_{X\times X'}\Big[ \ol\m(\{(x,x')\spec x=y, x'\not=y'\})+\ol\m(\{(x,x')\spec x\not=y, x'=y'\})\Big]
\,d\ol\m(y,y')\\
&\le&
\int_{X\times X'}\Big[\m(\{y\})+\m'(\{y'\})\Big]\,d\ol\m(y,y')=0.
\end{eqnarray*}

(ii) Decompose $X$ into $2^n$ disjoint subsets of equal volume
\[X=\bigcup_{i=1}^{2^n}X_i,\quad \m(X_i)=2^{-n}.\] Indeed, by Remark~\ref{rem:paramet} (\ref{rem:paramet1}), we can find a Borel measurable bijection $\psi:I\to X$ with $\m=\psi_*\Leb^1$ and Borel measurable inverse. Now perform the decomposition on $I$.\\
Define a coupling $\ol\m$ of $\m_n$ and $\m$ by
\[d\ol\m(j,x)=\sum_{i=1}^{2^n}1_{X_i}(x)d\m(x)\, d\delta_i(j).\]
Then
\begin{eqnarray*}
\DD^2(\X_n,\X)&\le&
\int\int\Big|\d_n-\d\Big|^2 d\ol\m\,d\ol\m \\
&=&\sum_{i,j=1}^{2^n}\int\int\Big| \d_n(i,j)-\d(x,y)\Big|^2\,1_{X_j}(y)\,1_{X_i}(x)\,d\m(x)\,d\m(y)\\
&=&\sum_{i=1}^{2^n}\m(X_i)^2=2^{-n}.
\end{eqnarray*}
This yields the asserted upper estimate.
%

(iii) The geodesic $(\X_{n,t})_{0\le t\le1}$ connecting $\X_n=\X_{n,0}$ and $\X=\X_{n,1}$ is given by
\[\X_{n,t}=\auf X_n\times X, \d_t,\ol\m \zu\]
with $\d_t=(1-t)\d_n+t\d$. For each $t>0$, the pseudo metric $\d_t$ is \emph{not} a metric which generates the Polish topology of $X_n\times X$.
\end{proof}

\begin{corollary}\label{incomplete} $\XX$ is not closed. Even more, it is not open in $\ol\XX$.
\end{corollary}

To obtain at least a vague geometric interpretation
of the convergence $\X_n\to\X$ in the previous Lemma \ref{discrete-sphere}(ii), think of $\X_n$ being the tree consisting of $2^n$ edges $e_i=(0,v_i)$ of length $1/2$, glued together at the origin.
The vertices $v_i$ may be regarded as points on the circle with radius $1/2$, connected to each other only via the origin. The limit space $\X$ then may be regarded as the circle with radius $1/2$ equipped with the uniform distribution (= Haar measure) and the discrete metric (which amounts to say that each pair of points is connected only via the origin).

\begin{theorem}\label{hatXXeqolXX}\quad
$\hat\XX=\ol\XX$.
\end{theorem}

\begin{proof}
Given any pseudo metric measure space $(X,\d,\m)$, we have to find metric measure spaces $(X_n,\d_n,\m_n)$ with
\[\DD\Big((X_n,\d_n,\m_n),(X,\d,\m)\Big)\to 0.\]
We will modify the given pseudo metric step by step to transform  it into a complete separable metric.

(i) According to Lemma \ref{obdA} and Lemma \ref{pseudo lemma}(i),  we may assume without restriction that $X=I$, $\m=\Leb^1$. We then also will choose $X_n=I$, $\m_n=\Leb^1$ for all $n$.
Let $\d$ be the given pseudo metric on $I$. That is, $\d$ is a symmetric $L^2$-function on $I\times I$ which satisfies the triangle inequality $\Leb^2$-a.e. in the sense of Definition \ref{m2triangle}.

(ii) Without restriction $\d$ is bounded, say bounded by $L$.
Indeed, $\d$ is square integrable on $I^2$ and thus
can be approximated in $L^2$-norm by $\d_k=\min\{\d,k\}$ for $k\in\N$. Obviously, $\d_k$ is again a pseudo metric and now in addition bounded. The convergence $\d_k\to\d$ in $L^2$ implies  $(I,\d_k,\Leb^1)\to (I,\d,\Leb^1)$ in $\DD$-distance.

(iii) We extend $\d$ to a pseudo metric $\d'$ on $\R$ by
\[ \d'(x,y)=\left\{\begin{array}{cl}
\d(x,y),\quad& \text{if } x,y\in I\\
L/2,\quad& \text{if } x\in I, y\not\in I \text{ or }y\in I, x\not\in I\\
0,\quad& \text{if } x,y\not\in I.
\end{array}\right.
\]

(iv) Let $\eta_n$ for $n\in\N$ be a smooth mollifier kernel on $\R$, i.e. $\eta_n\ge0$ on $\R$, $\eta_n=0$ outside of $[-\frac1n,+\frac1n]$ and  $\int\eta_n(t)\,dt=1$, say
$\eta_n(t)=n\cdot\eta(nt)$ with
\[\eta(t)=\left\{\begin{array}{ll}
C\cdot \exp\Big(\frac1{t^2-1}\Big),\quad&t\in(-1,1)\\
0,&\text{else}.
\end{array}\right.
\]
Put
\begin{equation}
\d'_n(x,y)=\int_\R\int_\R \d'(x+s,y+t)\,\eta_n(s)\,\eta_n(t)\,ds\,dt.
\end{equation}
For each $n\in\N$, this defines a pseudo metric on $\R$. The triangle inequality holds for each triple of points $x,y,z\in\R$. Indeed,
\begin{eqnarray*}
\lefteqn{\d'_n(x,y)+\d'_n(y,z)-\d'_n(x,z)}\\
&=&
\int_\R\int_\R\int_\R \Big[\d'(x+s,y+t)
+\d'(y+t,z+u)-\d'(x+s,z+u)\Big]
\,\eta_n(s)\,\eta_n(t)\,\eta_n(u)\,ds\,dt\,du
\end{eqnarray*}
which is nonnegative since the integrand $[\ldots]$ is nonnegative for $\Leb^3$-a.e. triple $(s,t,u)$.

Hence, $\d'_n$ is continuous and satisfies the triangle inequality. Moreover,
\[ \|\d'_n-\d\|_{L^2(I^2)}\to0\]
as $n\to \infty$.

(v) Finally, we put
\begin{equation}
\d_n(x,y)=\d'_n(x,y)+\frac1n |x-y|
\end{equation}
for $x,y\in I$. Then $\d_n$ is a complete separable metric which induces the standard Euclidean topology on $I$. In particular, $(I,\d_n,\Leb^1)$ is a metric measure space.
Moreover, $\|\d_n-\d\|_{L^2(I^2)}\le\|\d'_n-\d\|_{L^2(I^2)}+\frac1n\to0$
as $n\to\infty$. This  proves the claim.
\end{proof}

The proof of the previous theorem in particular leads to the following

\begin{corollary}\label{m2 = m3}
For a gauged measure space $(X,\d,\m)$, the following assertions are equivalent:
\begin{enumerate}
\item $\d$ satisfies the triangle inequality $\m^2$-a.e.
\item $\d$ satisfies the triangle inequality $\m^3$-a.e.
\end{enumerate}
\end{corollary}
Note that in contrast to \cite{ZP11},  our result does not require the pseudo metric to be separable.

\begin{proof} Let us briefly sketch the arguments for "$(ii)\Rightarrow (i)$". (The converse implication is obvious.)
Given $(X,\d,\m)$, we choose a parametrization $\psi\in\Par(\m)$ to transfer everything from $X$ to $I$. In particular,
the pull back $\d'=\psi^*\d$ will satisfy the triangle inequality $\Leb^3$-a.e. on $I$. We approximate $\d'$ by convolution with the mollifier kernels $\eta_n$ (as in the previous proof)
 and obtain pseudo metrics $\d_n'$ on $I$ which satisfy the triangle inequality everywhere. For $n\to\infty$ we obtain, at least along subsequences, that $\d_n'\to \d'$   $\Leb^2$-a.e. on $I^2$.
Thus $\d'$ satisfies the triangle inequality $\Leb^2$-a.e. Back to the space $X$, this amounts to say that the original $\d$ satisfies the triangle inequality $\m^2$-a.e.
\end{proof}

\begin{corollary}\label{olisgeo} $\ol\XX$ is a complete geodesic space of nonnegative curvature in the sense of Alexandrov.

It is a convex (`totally geodesic') subset of $\YY$   and it contains $\XX$ as a convex subset.
\end{corollary}


\subsection{The $n$-Point Spaces}

For each $n\in\N$, let $\sf M^{(n)}$  be the linear space of real-valued symmetric $(n\times n)$-matrices vanishing on the diagonal. Equipped with the re-normalized $l_2$-norm
\[
	\|f\|_{{\sf M}^{(n)}}:=\bigg(\frac{2}{n^2} \sum_{1\le i<j\le n} f_{ij}^2\bigg)^{1/2}\qquad \text{for } f=(f_{ij})_{1\leq
	i<j\leq n}\in {\sf M}^{(n)}
\]
it is a Hilbert space (and as such of course a very particular example of an Alexandrov space of nonnegative curvature). It is isometric to $\R^{\frac{n(n-1)}{2}}$ equipped with a constant multiple of the Euclidean metric.

The permutation group $S_n$ acts isometrically on ${\sf M}^{(n)}$ via
\[
	(\sigma,f) \mapsto \sigma^*f \qquad \text{with } (\sigma^*f)_{ij}:=f_{\sigma_i \sigma_j}.
\]	
It defines an equivalence relation $\sim$ in ${\sf M}^{(n)}$ by
\[
	f\sim f' \quad \Longleftrightarrow \quad \exists \sigma\in S_n:\, f_{ij}=f'_{\sigma_i \sigma_j}
	\;(\forall i,j\in \{1,\ldots,n\}).
\]

\begin{theorem}\begin{enumerate}
\item
	The quotient space $\Mn:={\sf M}^{(n)}/\sim$ equipped with the metric
	\[
		d_{\Mn}(f,f')=\inf \{ \|f-\sigma^*f'\|_{{\sf M}^{(n)}} \spec \sigma\in S_n \}
	\]
	is a complete geodesic space of nonnegative curvature. Its Hausdorff dimension is
	$\frac{n(n-1)}{2}$.
\item $(\Mn,d_{\Mn})$ is isometric to a cone in $\R^{ \frac{n(n-1)}{2}}$ (with the induced inner metric in the cone).
	This cone can be regarded as fundamental domain for the group action of $S_n$.
\item
$\Mn$ is a Riemannian orbifold. The tangent space at $f\in\Mn$ is given by
\[ {\sf T}_f\Mn=\R^{ \frac{n(n-1)}{2}}/{\sf Sym}(f)\]
where \[{\sf Sym}(f)=\Big\{\sigma\in S_n\spec \sigma^*f=f\Big\}\]
is the \emph{symmetry group} (or \emph{stabilizer subgroup} or \emph{isotropy group}) of $f$.
\end{enumerate}
\end{theorem}

\begin{proof}
(i)	According to general results on geometry of Alexandrov spaces, lower curvature bounds are preserved
	under passing to quotient spaces w.r.t.\ any isometric group action, cf. \cite{bbi},
	Proposition 10.2.4.
	The remaining claims in (i) and (ii) are straightforward.
For (iii), we refer to \cite{Thurston}, chapter 13.
\end{proof}

\begin{figure}[h!]
	\psfrag{r1}{$r_1$}
	\psfrag{r2}{$r_2$}
	\psfrag{r3}{$r_3$}
	\psfrag{t=0}{$t=0$}
	\psfrag{t=1}{$t=1$}
	\psfrag{t=2}{$t=2$}
	\centering
	\caption{\wichtig{Triangles $r(t)=\exp_r(tg)$ for $r=(3,4,5)\in \M^{(3)}$, $g=(0,\frac12,-\frac12)\in {\sf T}_r\M^{(3)}$ and $t=0,1,2$.
\\
Note that for the equilateral triangle $r(1)\in \M^{(3)}$:
$\exp_{r(1)}(tg)=\exp_{r(1)}(-tg)\qquad(\forall t\in\R)$.}}
	
\end{figure} 

\begin{figure}[h!]
	\psfrag{r1}{$r_1$}
	\psfrag{r2}{$r_2$}
	\psfrag{r3}{$r_3$}

	\centering
		\includegraphics[scale=0.3]{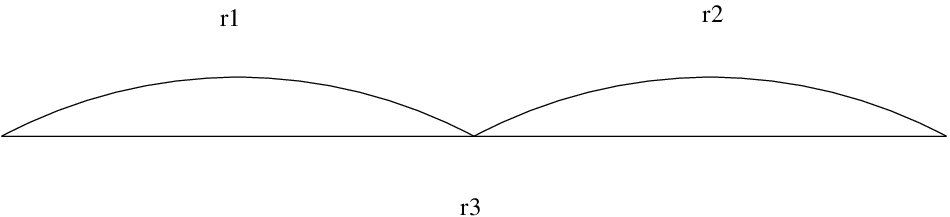}
	\caption{\wichtig{For $r=(1,1,2)$ and $g=(0,0,1)$:\quad
$g\in \sf T_r\M^{(3)}\quad\text{but}\quad
g\not\in \sf T_r\M^{(3)}_\leq$.}}

\end{figure} 

\index{m@$\sf M^{(n)}$, ${\sf M}^{(n)}_\leq$}
\index{m@$\Mn$, $\Mnl$}
Now let us consider the subset ${\sf M}^{(n)}_\leq$ in ${\sf M}^{(n)}$ consisting of those symmetric $(n\times n)$-matrices $(f_{ij})_{1\leq i<j \leq n}$ which `satisfy the triangle inequality' in the following sense:
\begin{equation}\label{trianglemat}
	f_{ij}+f_{jk}\geq f_{ik} \qquad (\forall i,j,k\in\{1,\ldots,n\}).
\end{equation}
Note that this constraint is compatible with the equivalence relation $\sim$ induced by the action of the permutation group $S_n$:\\
\[\forall f,f'\in {\sf M}^{(n)}\text{ with }f\sim  f':\qquad
	f\in {\sf M}^{(n)}_\leq \quad \Longleftrightarrow \quad f'\in {\sf M}^{(n)}_\leq.
\]
Hence, the space $\Mnl:={\sf M}^n_\leq/\sim$ coincides with the subset of $\Mn$ of equivalence classes of $f$ which satisfy \eqref{trianglemat}.

\begin{example} The simplest non-trivial case is $n=3$. Here
\[{\sf M}^{(3)}=\left\{\left(\begin{array}{ccc}
0&r_1&r_2\\
r_1&0&r_3\\
r_2&r_3&0
\end{array}\right)\spec r=(r_1,r_2,r_3)\in \R^3\right\}\]
and
\[{\sf M}^{(3)}_\leq=\Big\{r\in\R^3\spec r_1\le r_2+ r_3, \ r_2\le r_3+ r_1,\ r_3\le r_1+ r_2\Big\}.\]
A fundamental domain of the quotient space $\sf M^{(3)}/S_3$ is for instance given by
\[\tilde{\sf M}^{(3)}=\Big\{r\in\R^3\spec r_1\le r_2\le r_3\Big\}.\]
\begin{figure}[h!]
\begin{subfigure}{0.1\textwidth}
\psfrag{x1}{$r_2$}
\psfrag{x2}{$r_3$}
\psfrag{x3}{$r_1$}
\includegraphics[scale=0.25]{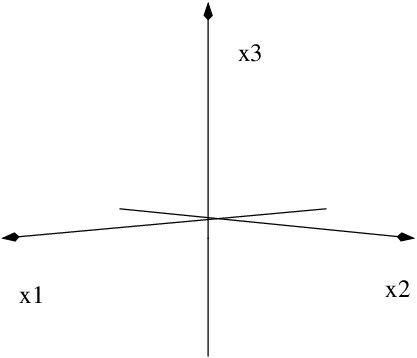}
\end{subfigure}
\begin{subfigure}{0.45\textwidth}
\centering
		\includegraphics[scale=0.35]{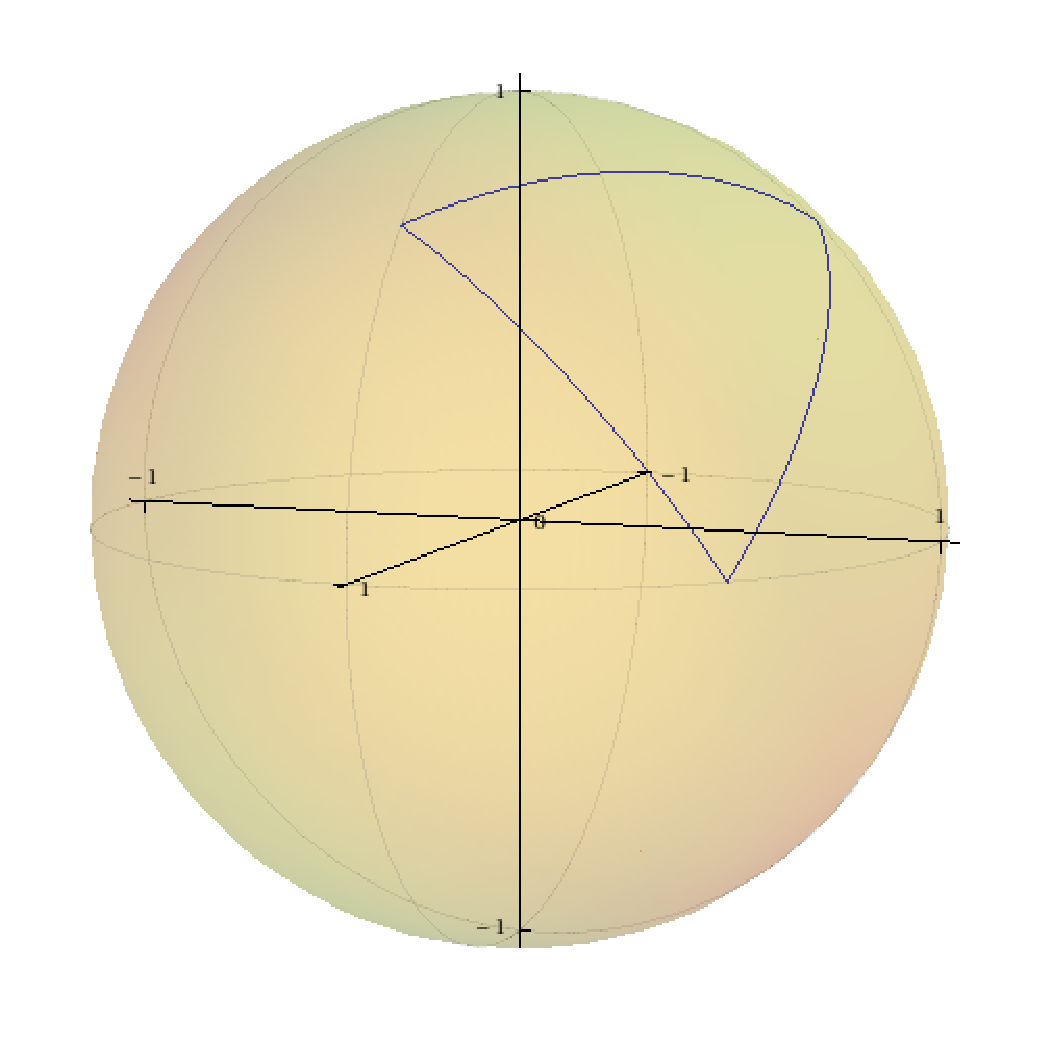}
		\caption{\wichtig{The domain bounded by the blue lines is\\
		${\sf M}^{(3)}_\leq\cap {\mathbb S}^2$.}}
\end{subfigure}
\begin{subfigure}{0.4\textwidth}
\centering
		\caption{\wichtig{The red colored area is\\
		$\tilde{\sf M}^{(3)}\cap {\mathbb S}^2$.}}
\end{subfigure}	
\end{figure} 
\begin{figure}[h!]
\begin{center}
\end{center}
\caption{\wichtig{The green vectors illustrate elements in ${\sf T}_f{\sf M}^{(3)}_\leq$ which are mutually identified,
whereas the black vector is in ${\sf T}_f{\sf M}^{(3)}$, but not in in ${\sf T}_f{\sf M}^{(3)}_\leq$.}}
	
\end{figure} 
\end{example}

%
%
\newpage
\begin{corollary}
\begin{enumerate}
\item
	$\Mnl$ is a closed convex subset of $\Mn$. It is itself an Alexandrov space of nonnegative curvature
	with dimension $\frac{n(n-1)}{2}$.
\item
For $f\in \Mnl$ the tangent space ${\sf T}_f\Mnl$ consists of those $g\in {\sf T}_f\Mn$ for which
$\exp_f(tg)=f+tg$ stays within $\Mnl$
at least for some $t>0$.
\end{enumerate}
\end{corollary}

Now let us consider the injection
\begin{equation*}\Phi: \quad\begin{array}{ccc}
		\Mn &\rightarrow &\YY,\\
		 f=(f_{ij})_{1\leq i<j \leq n} &\mapsto& \X=\auf\{1,\ldots,n\},f, \frac{1}{n}\sum_{i=1}^n \delta_{i}\zu.
	\end{array}
\end{equation*}
Elements in the image $\YYn:=\Phi\Big(\Mn\Big)$ are called \emph{$n$-point spaces}. They are characterized as gauged measure spaces for which the mass is uniformly distributed on $n$ (not necessarily distinct) points. For convenience, we also require that the gauge functions vanish on the diagonal. The image
\[\XXn:=\Phi\Big(\Mnl\Big)\]
of $\Mnl$ consist of those mm-spaces with mass uniformly distributed on $n$ points.

\index{x@$\XXn$, $\YYn$}
\begin{proposition}
For each $n\in\N$, $\Phi$ is a 1-Lipschitz map:
\[\DD\big(\Phi(f),\Phi(g)\big)\le d_{\Mn}(f,g)\qquad(\forall f,g\in \Mn).\]
Moreover,
\[\size\big(\Phi(f)\big)=\|f\|_{{\sf M}^{(n)}}.\]
	\end{proposition}

\begin{proof}
Obviously,
$\size^2\big(\Phi(f)\big)=\frac1{n^2}\sum_{i,j=1}^nf_{ij}^2=\|f\|^2_{{\sf M}^{(n)}}$. Moreover,
(cf. Proposition \ref{product})
\begin{eqnarray*}
{-\DD^2\big(\Phi(f),\Phi(g) \big)+\size^2\big(\Phi(f)\big)+\size^2\big(\Phi(g)\big)}
&=&\sup_{p\in {\sf P}^{(n)}}\frac2{n^2}\sum_{i,j=1}^n\sum_{k,l=1}^n
f_{ij}\cdot g_{kl}\cdot p_{ik}\cdot p_{jl}
\end{eqnarray*}
where ${\sf P}^{(n)}$ denotes the set of doubly stochastic $(n\times n)$-matrices, i.e. set of all $p=(p_{ij})_{1\le i,j\le n}\in\R_+^{n\times n}$ satisfying
$\sum_{i=1}^np_{il}=\sum_{j=1}^np_{kj}=1$ for all $k,l=1,\ldots,n$.
Particular examples of such doubly stochastic matrices are given for each $\sigma\in S_n$ by
\[p_{ij}=\delta_{i\sigma_j}.\]
The claim thus follows from the fact that
\begin{eqnarray*}
{-d_{\Mn}^2\big(f,g \big)+\|f\|^2_{{\sf M}^{(n)}}+\|g\|^2_{{\sf M}^{(n)}}}
&=&\sup_{\sigma\in S_n}\frac2{n^2}\sum_{i,j=1}^n
f_{ij}\cdot g_{\sigma_i \sigma_j}.
\end{eqnarray*}
\end{proof}

\begin{remark} {\it The injection
\[\Phi:\ \Mnl \rightarrow \ol\XX\]
is an embedding.}
Indeed, assume that
\[\DD\Big(\Phi(d^k),\Phi(d^\infty)\Big)\to0\qquad\text{as }k\to\infty\]
for some $d^\infty\in\Mnl$ and some sequence $(d^{k})_{k\in\N}$ in $\Mnl$.
Assume for simplicity that $d^\infty$ and all the $d^k$ are metrics on $\{1,\ldots,n\}$. (All  $d^\infty,d^k\in\Mnl$
can be approximated by metrics.) The $d^k$ are uniformly bounded. Thus according to Corollary 2.10
 \[ \D_2 \Big(\Phi(d^k),\Phi(d^\infty)\Big)\to0\qquad\text{as }k\to\infty.\]
 According to the union lemma (\cite{Gro}, \cite{St06}) this implies that there exists a metric space $(X,\d)$
 and isometric embeddings $\eta^k:\big(\{1,\ldots,n\},d^k\big)\to (X,\d)$ for all $k\in\N\cup\{\infty\}$ such that
 \[\d_2\Big((\eta^k)_*\big(\frac1n\sum_{i=1}^n\delta_i\big), (\eta^\infty)_*\big(\frac1n\sum_{j=1}^n\delta_j\big)\Big)\to0\qquad\text{as }k\to\infty\]
 where $\d_2$ now denotes the $L^2$-Wasserstein distance for probability measures on $(X,\d)$, i.e.
 \[\d_2\Big((\eta^k)_*\big(\frac1n\sum_{i=1}^n\delta_i\big), (\eta^\infty)_*\big(\frac1n\sum_{j=1}^n\delta_j\big)\Big)=
 \inf\Big\{\frac1n\sum_{i,j=1}^n \d^2\big(\eta^k(i),\eta^\infty(j)\big)p_{ij}:\ \sum_l p_{il}=\sum_l p_{lj}=1\text{ for all }i,j\Big\}^{1/2}.\]
 For this `classical' transport problem, however, it is known that the infimum is attained (among others) on the set of extremal points within the set of doubly stochastic matrices.
 Hence,
 \[\d_2\Big((\eta^k)_*\big(\frac1n\sum_{i=1}^n\delta_i\big), (\eta^\infty)_*\big(\frac1n\sum_{j=1}^n\delta_j\big)\Big)=
 \inf\Big\{\frac1n\sum_{i=1}^n \d^2\big(\eta^k(i),\eta^\infty(\sigma_i)\big):\ \sigma\in S_n\Big\}^{1/2}.\]
 Moreover, the triangle inequality for $\d$ implies
 \begin{eqnarray*} &&2\inf\Big\{\frac1n\sum_{i=1}^n \d^2\big(\eta^k(i),\eta^\infty(\sigma_i)\big):\ \sigma\in S_n\Big\}^{1/2}\\
 &\ge&
 \inf\Big\{\frac1{n^2}\sum_{i,j=1}^n \Big|\d\big(\eta^k(i),\eta^k(j)\big)- \d\big(\eta^\infty(\sigma_j),\eta^\infty(\sigma_i)\big)\Big|^2:\ \sigma\in S_n\Big\}^{1/2}\\
 &=&
 \inf\Big\{\frac1{n^2}\sum_{i,j=1}^n \Big|d^k_{ij}- d^\infty_{\sigma_j\sigma_i}\Big|^2:\ \sigma\in S_n\Big\}^{1/2}\\
 &=&  d_{\Mn}\big(d^k,d^\infty\big).\end{eqnarray*}
 This finally implies $d_{\Mn}\big(d^k,d^\infty\big)\to0$ as $k\to\infty$ which is the claim.
 \end{remark}

\begin{challenge}
Prove or disprove that the injections
\[\Phi:\ \Mn \rightarrow \YY\]
and
\[\Phi:\ \Mnl \rightarrow \ol\XX\]
are isometric embeddings.
\end{challenge}

\begin{proposition}
			$\bigcup_{n\in\N} \XXn$ is dense in $\ol\XX$ and
$\bigcup_{n\in\N} \YYn$ is dense in $\YY$.
	\end{proposition}

\begin{proof}
The density assertion (w.r.t. $\DD$) concerning $\XX$ or $\ol\XX$ is an immediate consequence of the analogous density statement for $\XX$ w.r.t. $\D$ in \cite{St06}, Lemma~3.5, and the
	estimate $\DD\leq 2\D$ of Lemma~\ref{lp-d-dd}.

To see the density assertion concerning $\YY$, let a gauged measure space $\X$ be given. We always can choose a representative $(X,\f,\m)$ without atoms. The gauge function $\f\in L^2_s(X^2,\m^2)$ then can be approximated in $L^2$-norm by piecewise constant functions $\f^{(n)}\in L^2_s(X^2,\m^2)$, $n\in\N$. Even more, these functions $\f^{(n)}$ on $X\times X$ can be chosen to be constant on $X_i^{(n)}\times X_j^{(n)}$ for $1\le i < j\le n$ for a suitable partition of $X$ into sets $X_i^{(n)}$ of volume $\frac1n$ ($\forall i=1,\ldots,n$).
That is, for each $n\in\N$ the gauged measure space $(X,\f^{(n)},\m)$ is homomorphic to the $n$-point space
$\Big(\{1,\ldots,n\},f^{(n)}, \frac{1}{n}\sum_{i=1}^n \delta_{i}\Big)$ for
\[f_{ij}^{(n)}:=\f^{(n)}\Big|_{X_i^{(n)}\times X_j^{(n)}}\qquad(\forall 1\le i<j\le n).\]
\end{proof}

The spaces ${\sf M}^{(n)}$ also play a key role in the `reconstruction theorem' of Gromov \cite{Gro} and Vershik \cite{Versh} based on `random matrix distributions'.
For each $n\in\N$ and each gauged measure space $(X,\f,\m)$, let
$\nu_n^{(X,\f,\m)}$
denote the distribution of the matrix
\[\big(\f(x_i,x_j)\big)_{1\le i<j\le n}\in {\sf M}^{(n)}\]
under the measure $d\m^n(x_1,\ldots,x_n)$. Here $\m^n=\m^{\otimes n}$ denotes the $n$-fold product measure of $\m$. Let $\m^\infty=\m^{\otimes\N}$ denote the infinite product of $\m$ defined on $X^\infty=\{(x_i)_{i\in\N}:\, x_i\in X\}$, put
\[{\sf M}^{(\infty)}=\Big\{ \big(f_{ij}\big)_{1\le i<j<\infty}:\ f_{ij}\in\R\Big\}\]
and let $\nu_\infty^{(X,\f,\m)}$ denote the distribution of
\[\big(\f(x_i,x_j)\big)_{1\le i<j<\infty}\in {\sf M}^{(\infty)}\]
under the measure $d\m^\infty(x_1,x_2,\ldots)$.
\index{m@${\sf M}^{(\infty)}$}
\index{matrix distribution}

\begin{proposition}\label{reconstruct} For the following assertions, the implications $(i)\Rightarrow(ii)\Leftrightarrow(iii)$ hold true for all gauged measure spaces $(X,\f,\m)$ and $(X',\f',\m')$:
\begin{enumerate}
\item
$(X,\f,\m)$ and $(X',\f',\m')$ are homomorphic (as elements in $\YY$)
\item For each $n\in\N$:\
 $\nu_n^{(X,\f,\m)}$ and $\nu_n^{(X',\f',\m')}$ coincide (as probability measures on ${\sf M}^{(n)}$)
\item $\nu_\infty^{(X,\f,\m)}$ and $\nu_\infty^{(X',\f',\m')}$ coincide (as probability measures on ${\sf M}^{(\infty)}$).
\end{enumerate}
For  metric measure spaces $(X,\f,\m)$ and $(X',\f',\m')$, the assertions (i), (ii) and (iii) are equivalent.
\end{proposition}

\begin{proof}
\begin{description}
\item[$(i)\Rightarrow(ii)$] Assuming the spaces to be homomorphic amounts to assume that there exists a measure $\ol\m\in\Cpl(\m,\m')$ on $X\times X'$ such that $\f(x,y)=\f'(x',y')$ for $\ol\m^2$-a.e. $((x,x'),(y,y'))$.
    Thus
    \begin{eqnarray*}
    &&\text{distr. of } \Big(\f(x_i,x_j)\Big)_{1\le i<j\le n} \text{ under } d\m^n\big(x_1,\ldots,x_n\big)\\
    &=&
    \text{distr. of } \Big(\f(x_i,x_j)\Big)_{1\le i<j\le n} \text{ under } d\ol\m^n\big((x_1,x'_1),\ldots,(x_n,x'_n)\big)\\
    &=&\text{distr. of } \Big(\f'(x'_i,x'_j)\Big)_{1\le i<j\le n} \text{ under }
    d\ol\m^n\big((x_1,x'_1),\ldots,(x_n,x'_n)\big)\\
      &=&\text{distr. of } \Big(\f'(x'_i,x'_j)\Big)_{1\le i<j\le n} \text{ under } d\m'^n\big(x'_1,\ldots,x'_n\big).
\end{eqnarray*}
    \item[$(ii)\Leftrightarrow(iii)$:] Straightforward consequence of the fact that the Borel field in ${\sf M}^{(\infty)}$ is generated by pre-images under projections into ${\sf M}^{(n)}$, $n\in\N$.
\item[$(iii)\Rightarrow(i)$:] Reconstruction theorem \cite{Gro}, $3\frac12.5$.
\end{description}
\end{proof}

\section{The Space $\YY$ as a Riemannian Orbifold}

\subsection{The Symmetry Group}

Let Polish spaces $X_1,X_2,X_3$  with Borel probability measures $\m_1,\m_2,\m_3$ be given as well as couplings
$\mu'\in\Cpl(\m_1,\m_2)$ and $\mu''\in\Cpl(\m_2,\m_3)$.
Recall  the gluing construction from Lemma
\ref{glue1} which  yields a measure
$\widehat\mu=\mu'\boxtimes\mu''$ on $X_1\times X_2\times X_3$ with $(\pi_1,\pi_2)_*\widehat\mu=\mu'$ and $(\pi_2,\pi_3)_*\widehat\mu=\mu''$.
\begin{definition}
The \emph{melting}\index{melting}\index{$\boxtimes$}
of $\mu'$ and $\mu''$ is the probability measure $\mu\in\Cpl(\m_1,\m_3)$ defined as
\[\mu=(\pi_1,\pi_3)_*(\mu'\boxtimes\mu'').\]
It will be denoted by $\mu'\boxdot\mu''$.
\end{definition}
\index{$\boxdot$}

\begin{lemma} Let a gauged measure space $(X,\f,\m)$ be given.
\begin{enumerate}
\item
$\Cpl(\m,\m)$, the space of all \emph{self-couplings} of $\m$,  is a group with composition
$\boxdot$.
The neutral element is the \emph{diagonal coupling}
\[d\nu(x,y)=\d\delta_{x}(y)\,d\m(x).\]
The element inverse to $\mu$ is given by \index{${\mu}^{-1}$}
\[d{\mu}^{-1}(x,y)=d\mu(y,x).\]
\item
A norm is given on this group by
\[ \|\mu\|_f=\bigg(\int_X\int_X \Big|\f(x_0,y_0)-\f(x_1,y_1)\Big|^2\,d\mu(x_0,x_1)\,d\mu(y_0,y_1)\bigg)^{1/2}.\]
\end{enumerate}
\end{lemma}
\begin{proof}
(i) is obvious: the gluing of $\mu$ and $\mu^{-1}$ for instance is given by
$(\pi_1,\pi_2,\pi_1)_*\mu$. Projecting this onto the first and third factor yields
\[(\pi_1,\pi_1)_*\mu=(\pi_1,\pi_1)_*\m\]
which is the diagonal coupling.

(ii) The  inequality to be verified
\[ \|\mu'\boxdot\mu''\|_f \le  \|\mu'\|_f  +\|\mu''\|_f \]
follows exactly in the same way as  the triangle inequality for $\DD$.
\end{proof}

\begin{definition}
The \emph{symmetry group}\index{s@$\Sym(.)$} of $(X,\f,\m)$ is the subgroup of $\Cpl(\m,\m)$ of elements with vanishing norm:
\[ \Sym(X,\f,\m)=\Big\{\mu\in \Cpl(\m,\m)\spec \|\mu\|_f=0\Big\}.\]
In other words, $\Sym(X,\f,\m)$ is the set of all \emph{optimal} couplings of $(X,\f,\m)$ with itself.

We say, that $(X,\f,\m)$ has \emph{no symmetries} if $\Sym(X,\f,\m)$ only contains the neutral element (diagonal coupling).
\end{definition}

The symmetry group $\Sym(X,\f,\m)$ will depend on the choice of the representative within the equivalence class  $\auf X,\f,\m\zu$. For different choices of representatives, the groups
will be obtained from each other via conjugation and thus in particular will be isomorphic to each other.

\begin{lemma} Let two homomorphic gauged measure spaces $(X,\f,\m)$ and $(X',\f',\m')$ be given with $\nu\in\Opt(\m,\m')$ being a coupling which realizes the vanishing $\DD$-distance.
Then
\begin{eqnarray*}
\Sym(X',\f',\m')& =& \nu^{-1}\ \boxdot\ \Sym(X,\f,\m)\ \boxdot\ \nu\\
&=&\Big\{\mu'=\nu^{-1} \boxdot\mu\boxdot\nu\spec \mu\in \Sym(X,\f,\m)\Big\}.
\end{eqnarray*}
\end{lemma}

\begin{proof}
The fact that $\mu$ is in $\Sym(X,\f,\m)$ implies that
$\f(x_0,y_0)=\f(x_1,y_1)$ for $\mu^2$-a.e. $((x_0,x_1),(y_0,y_1))\in X^2\times X^2$. The fact that $\nu$ realizes the (vanishing) distance of $(X,\f,\m)$ and $(X',\f',\m')$ implies that
$\f(x_0,y_0)=\f'(x'_0,y'_0)$ for $\nu^2$-a.e. $((x_0,x'_0),(y_0,y'_0))\in (X\times X')^2$.
Thus
\[\f'(x'_0,y'_0)=\f(x_0,y_0)=\f(x_1,y_1)=\f'(x'_1,y'_1)\]
for $(\nu^{-1}\boxtimes\mu\boxtimes\nu)^2$-a.e. $((x'_0,x_0,x_1,x'_1),(y'_0,y_0,y_1,y'_1))\in
(X'\times X\times X\times X')^2$.
Projecting the measure $\nu^{-1}\boxtimes\mu\boxtimes\nu$ from $X'\times X\times X\times X'$ onto $X'\times X'$ yields the claim:
\[\f'(x'_0,y'_0)=\f'(x'_1,y'_1)\]
for $(\nu^{-1}\boxdot\mu\boxdot\nu)^2$-a.e. $((x'_0,x'_1),(y'_0,y'_1))\in
(X'\times  X')^2$.
\end{proof}

If the underlying space is not just a gauged measure space but a metric measure space, then
the symmetry group admits an equivalent representation in more familiar terms.

\begin{definition}
Given   a metric measure space $(X,\d,\m)$, let\index{s@$\sym(.)$}
\[ \sym(X,\d,\m)=\Big\{\phi:X^\base\to X^\base\spec \m=\phi_*\m,\ \d=\phi^*\d\Big\}\]
where $X^\base$ denotes the support of $\m$.
\end{definition}
Note that any $\phi$ which preserves the metric is Lipschitz continuous and thus in particular Borel measurable.
If moreover it is measure preserving, then according to the proof of (iii) $\Rightarrow$ (iv) in Lemma \ref{iso-null} it is necessarily bijective with Borel measurable inverse.

\begin{lemma} Let $(X,\d,\m)$ be a metric measure space.
\begin{enumerate}
\item $\sym(X,\d,\m)$ is a group (with composition of maps as group operation)
\item The groups  $\sym(X,\d,\m)$ and $\Sym(X,\d,\m)$ are isomorphic. For any $\phi\in\sym(X,\d,\m)$ the corresponding  measure in $\mu\in\Sym(X,\d,\m)$ is given by \[\mu:=(\Id,\phi)_*\m.\]
\item Let $(X',\d',\m')$ be another metric measure space, isomorphic to the first one with $\psi: X^\base\to{X'}^\base$ being a Borel measurable bijection which pushes forward the measure and pulls back the metric.
    Then
    \[\sym(X',\d',\m')=\psi\circ \sym(X,\d,\m)\circ\psi^{-1}.\]
    \end{enumerate}
\end{lemma}

\begin{proof} Most properties are obvious. Let us briefly comment on the inverse of the isomorphism in
(ii). Let a measure $\mu\in\Sym(X,\d,\m)$ be given. It is an optimal coupling of $(X,\d,\m)$ with itself with vanishing $\int\int |\d-\d|^2d\mu d\mu$. According to Lemma \ref{iso-null} this implies that there exists a bijective Borel map (with Borel inverse) $\phi:X^\base\to X^\base$ satisfying $\m=\phi_*\m$ and $\d=\phi^*\d$.
\end{proof}
\subsection{Geodesic Hinges}

A \emph{geodesic hinge} is a pair of geodesics $(\X_t)_{0\le t\le\tau}$ and
$(\X'_t)_{0\le t\le\tau'}$ emanating from  a common point $\X_0=\X_0'$ in $\YY$.
To simplify the presentation, we assume $\tau=\tau'=1$. (Since the geodesics are not required to have unit speed, this is no restriction.)

We fix representatives $(X_0,\f_0,\m_0)$, $(X_1,\f_1,\m_1)$ and $(X'_1,\f'_1,\m'_1)$ of the endpoints as well as optimal couplings
$\ol\m\in\Cpl(\m_0,\m_1)$ and
$\ol\m'\in\Cpl(\m_0,\m'_1)$. We are now looking for couplings of $\ol\m$ and $\ol\m'$, that is, for  $\mu\in\Cpl(\ol\m,\ol\m')$ being measures on $X=X_0\times X_1\times X_0\times X'_1$.
 The projections onto the respective factors will be denoted by $\pi_0,\pi_1,\pi'_0,\pi'_1$.
 Note that the factor $X_0$  shows up twice in the definition of $\mu$.

For $t\in(0,1]$, we define the functional
\begin{eqnarray*} C_t(\mu)&=&\frac1{t^2}\int_{X}\int_{X}
\bigg|(1-t)\Big[\f_0(x_0,y_0)-\f_0(x'_0,y'_0)\Big]\\
&&\qquad\qquad\qquad\quad
+t\Big[\f_1(x_1,y_1)-\f'_1(x'_1,y'_1)\Big]\bigg|^2\,
d\mu(x_0,x_1,x'_0,x'_1)\,d\mu(y_0,y_1,y'_0,y'_1)
\end{eqnarray*}
on  $\Cpl(\ol\m,\ol\m')$.
Moreover, we put
\begin{eqnarray*} C_0(\mu)&=&
\sup_{t>0}C_t(\mu).
\end{eqnarray*}

\begin{lemma}
\begin{enumerate}
\item For each $t\in(0,1]$, there exists a measure $\mu_t\in\Cpl(\ol\m,\ol\m')$ which minimizes $C_t(.)$, an 'optimal' coupling of $\ol\m$ and $\ol\m'$ w.r.t. the cost function $|\f_t-\f'_t|^2$.
\item
The quantity
\[ C_t^*=C_t(\mu_t)=\frac1{t^2}\DD^2(\X_t,\X'_t)\]
is non-increasing in $t$.
\item
For each $\mu$ with $\big(\pi_0,\pi'_0\big)_*\mu\in\Sym(X_0,\f_0,\m_0)$,
\[ t\mapsto C_t(\mu) \quad\text{is independent of }t\in (0,1]\]
and thus $C_0(\mu)=C_t(\mu)=C_1(\mu)$. In particular,
\begin{equation*}
 C_0(\mu)=
\int_{X}\int_{X}
\Big[\f_1(x_1,y_1)-\f'_1(x'_1,y'_1)\Big]^2\,
d\mu(x_0,x_1,x'_0,x'_1)\,d\mu(y_0,y_1,y'_0,y'_1)<\infty.
\end{equation*}
\item
For each $\mu$ with $\big(\pi_0,\pi'_0\big)_*\mu\not\in\Sym(X_0,\f_0,\m_0)$,
\[C_0(\mu)=\infty.\]

\item
The functional $C_0$ is lower semicontinuous on $\Cpl(\ol\m,\ol\m')$.
\item
Every accumulation point $\mu_0$ 
of $(\mu_t)_{t>0}$ satisfies
$\big(\pi_0,\pi'_0\big)_*\mu_0\in\Sym(X_0,\f_0,\m_0)$.
\end{enumerate}
\end{lemma}

\begin{proof}
(i) follows from the existence result in Proposition \ref{exist opt} and the fact that
\[ C_t^*=\inf\Big\{C_t(\mu)\spec \mu\in\Cpl(\ol\m,\ol\m')\Big\}=\frac1{t^2}\DD^2(\X_t,\X'_t).\]

(ii) is a general consequence of nonnegative curvature  in Alexandrov geometry.

(iii), (iv) are obvious:
If the condition $\big(\pi_0,\pi'_0\big)_*\mu\in\Sym(X_0,\f_0,\m_0)$ was not satisfied then obviously
$C_0(\mu)=\infty$.
On the other hand, the previously mentioned condition $\big(\pi_0,\pi'_0\big)_*\mu\in\Sym(X_0,\f_0,\m_0)$ implies
$C_t(\mu)=C_1(\mu)<\infty$ independent of $t$ and thus $C_0(\mu)=C_1(\mu)<\infty$.

(v)
According to Lemma \ref{obdA},
 we may assume without restriction that $X_0=X_1=X'_1=I$ and $\m_0=\m_1=\m'_1=\Leb^1$.
  With the same argument as in the proof of Lemma \ref{approx-lsc} (approximating $\f_0,\f_1,\f'_1\in L^2$ by bounded continuous $\f_{0,i},\f_{1,i},\f'_{1,i}$), $C_t(.)$ is proven  to be continuous on $\Cpl(\ol\m,\ol\m')$.
As a supremum of continuous functionals $C_t$, the functional  $C_0$ is lower semicontinuous.

(vi)
Assume that $\big(\pi_0,\pi'_0\big)_*\mu_0\not\in\Sym(X_0,\f_0,\m_0)$
for an accumulation point $\mu_0$ of the family $(\mu_t)_{t>0}$.
Then
\begin{equation*}
\int_{X}\int_{X}
\Big[\f_0(x_0,y_0)-\f_0(x'_0,y'_0)\Big]^2\,
d\mu_0(x_0,x_1,x'_0,x'_1)\,d\mu_0(y_0,y_1,y'_0,y'_1)\ge 2\epsilon>0
\end{equation*}
and thus
for a converging (sub)sequence $(\mu_{t_n})_n$,
\begin{equation*}
\int_{X}\int_{X}
\Big[\f_0(x_0,y_0)-\f_0(x'_0,y'_0)\Big]^2\,
d\mu_{t_n}(x_0,x_1,x'_0,x'_1)\,d\mu_{t_n}(y_0,y_1,y'_0,y'_1)\ge \epsilon
\end{equation*}
uniformly in $n$. This implies
\[C_{t_n}(\mu_{t_n})\nearrow \infty\]
which contradicts the minimality of $\mu_{t_n}$.
\end{proof}

\begin{proposition}\label{alex-sym}
Let a geodesic hinge $(\X_t)_{0\le t\le1}$ and
$(\X'_t)_{0\le t\le1}$  be given as above with speeds $R=\DD(\X_0,\X_1)$, $R'=\DD(\X_0,\X'_1)$ and representatives
$(X_0\times X_1,\f_0+t(\f_1-\f_0),\ol\m)$, $(X_0\times X'_1,\f_0+t(\f'_1-\f_0),\ol\m')$, resp.
Then there exists a probability measure $\ol\mu$ on $X:=X_0\times X_1\times X_0\times X'_1$ with
\begin{itemize}
\item
$\ol\mu\in\Cpl(\ol\m,\ol\m')$, more precisely,
$(\pi_0,\pi_1)_*\ol\mu=\ol\m$ and $(\pi'_0,\pi'_1)_*\ol\mu=\ol\m'$
\item $(\pi_0,\pi'_0)_*\ol\mu\in \Sym(X_0,\f_0,\m_0)$
\end{itemize}
and
\begin{eqnarray*}
C_0(\ol\mu)=\inf\Big\{C_0(\mu)\spec
\mu\in\Cpl(\ol\m,\ol\m'), \ (\pi_0,\pi'_0)_*\mu\in \Sym(X_0,\f_0,\m_0)
\Big\}.
\end{eqnarray*}
Equivalently,
\begin{eqnarray*}
\big\langle \f_1-\f_0,\f'_1-\f_0\Big\rangle_{L^2(X^2,\ol\mu^2)}&=&\sup\Big\{
\big\langle \f_1-\f_0,\f'_1-\f_0\Big\rangle_{L^2(X^2,\mu^2)}
\spec\mu\in\Cpl(\ol\m,\ol\m'),\\
&&\qquad\qquad\qquad\qquad \qquad\qquad
(\pi_0,\pi'_0)_*\mu\in \Sym(X_0,\f_0,\m_0)
\Big\}.
\end{eqnarray*}
Moreover,
\begin{eqnarray*}
\cos\measuredangle\Big(\X_\bullet,\X'_\bullet\Big)
&\ge&\frac1{RR'}\big\langle \f_1-\f_0,\f'_1-\f_0\Big\rangle_{L^2(X^2,\ol\mu^2)}.
\end{eqnarray*}
\end{proposition}

\begin{proof}
The existence of $\ol{\mu}$ follows from the lower semicontinuity of $C_0$ proven in the previous Lemma.
Moreover,
\[C_0(\ol\mu)=\lim_{t\to0}C_t(\ol\mu)\ge\lim_{t\to0}C_t^*.\]
On the other hand, nonnegative curvature of $\YY$ implies that the angle between the geodesics always exists. Indeed, it is a monotone limit
\begin{eqnarray*}
\cos\measuredangle\Big(\X_\bullet,\X'_\bullet\Big)&=&\lim_{t\to0}
\frac1{2RR'}\Big[R^2+R'^2-\frac1{t^2}\DD^2(\X_t,\X'_t)\Big]\\
&=&
\frac1{2RR'}\Big[R^2+R'^2-\lim_{t\to0}C_t^*\Big].
\end{eqnarray*}
Finally, since $\f_0(x_0,y_0)=\f_0(x_0',y_0')$ for $\mu^2$-a.e.
$\big((x_0,x_1,x'_0,x'_1),(y_0,y_1,y'_0,y'_1)\big)\in X^2$ we may
 rewrite the previous expressions for each coupling $\mu$ (with the required  properties of its pairwise marginals)  as follows
\begin{eqnarray*}
\lefteqn{R^2+R'^2-C_0(\mu)}\\
&=&
\int_{X}\int_{X}
\bigg(
\Big[\f_1(x_1,y_1)-\f_0(x_0,y_0)\Big]^2
+\Big[\f'_1(x'_1,y'_1)-\f_0(x'_0,y'_0)\Big]^2\\
&&\qquad\qquad\qquad\qquad\qquad\qquad
-\Big[\f_1(x_1,y_1)-\f'_1(x'_1,y'_1)\Big]^2\bigg)\,
d\mu(x_0,x_1,x'_0,x'_1)\,d\mu(y_0,y_1,y'_0,y'_1)\\
&=&2
\int_{X}\int_{X}
\Big[\f_1(x_1,y_1)-\f_0(x_0,y_0)\Big]\cdot
\Big[\f'_1(x'_1,y'_1)-\f_0(x'_0,y'_0)\Big]\,
d\mu(x_0,x_1,x'_0,x'_1)\,d\mu(y_0,y_1,y'_0,y'_1).
\end{eqnarray*}
This is the claim.
\end{proof}

\begin{conjecture}
For each geodesic hinge as above,
\begin{eqnarray*}
\cos\measuredangle\Big(\X_\bullet,\X'_\bullet\Big)
&=&\frac1{RR'}\big\langle \f_1-\f_0,\f'_1-\f_0\Big\rangle_{L^2(X^2,\ol\mu^2)}.
\end{eqnarray*}
\end{conjecture}

\subsection{Tangent Spaces and Tangent Cones}

\begin{definition} The \emph{tangent space}\index{t@$\TT_\X$} at $\X\in\YY$ is defined as
\[\TT_\X=\bigcup_{\auf X,\f,\m\zu =\X} L^2_s(X^2,\m^2)/\sim
\]
with union taken over all gauged measure spaces $(X,\f,\m)$ in the homomorphism class $\auf X,\f,\m\zu $.
Here
$g\in L^2_s(X^2,\m^2)$ and  $g'\in L^2_s(X'^2,\m'^2)$ are regarded as equivalent, briefly $g\sim g'$, if
they are defined on two representatives $(X,\f,\m)$ and $(X',\f',\m')$ of $\X$ for which
there exists a coupling $\mu\in \Cpl(\m,\m')$ such that $\f=\f'$ and $g=g'$ $\mu^2$-a.e. on $(X\times X')^2$.
More precisely, the latter means that
\[\f(x,y)=\f'(x',y')\quad\text{and} \quad g(x,y)=g'(x',y')\]
for $\mu^2$-a.e. $\big((x,x'),(y,y')\big)\in(X\times X')^2$.
\end{definition}

\begin{remarks}
\begin{enumerate}
\item This, indeed, is an equivalence relation: $g\sim g'$ and $g'\sim g''$ implies $g\sim g''$.
\item
For $g,h$ defined as symmetric $L^2$-functions on the same representative $(X,\f,\m)$ of $\X$, the above equivalence means that
$g=h$ $\mu^2$-a.e. for some $\mu\in\Sym(X,\f,\m)$.
\item
Given a gauged measure space $(X,\f,\m)$,  a probability space $(X',\m')$ (for consistence, with $X'$ being a Polish space) is called ``enlargement''  of $(X,\m)$ if there exists a measurable map $\phi:X'\to X$ with $\m=\phi_*\m'$.  In this case,  the map
\[\Phi:\ g\mapsto \phi^*g\]
defines an isometric embedding of the Hilbert space
$L^2_s(X^2,\m^2)$  into the Hilbert space $L^2_s(X'^2,\m'^2)$.
Put $\f'=\phi^*\f$.
Then \[g\sim \phi^*g\]
for each $g\in L^2_s(X^2,\m^2)$. Indeed, $\mu:=(\phi,\Id)_*\m'$ defines a coupling of $\m$ and $\m'$ with the property
$\f=\f', g=\phi^*g$ $\mu^2$-a.e.

Therefore, for all $g,h\in L^2_s(X^2,\m^2)$,
\[g\sim h\quad\Longleftrightarrow\quad \phi^*g\sim\phi^*h.\]
\item
For each gauged measure space $(X,\f,\m)$, the ``standard'' space  $(I,\Leb^1)$ together with some parametrization $\phi\in\Par(\m)$ can be regarded as an enlargement.
Hence, each tangent vector admits a representative in $L^2_s(I^2,\Leb^2)$. In other words,  the tangent space can be considered as subspace of $L^2_s(I^2,\Leb^2)$, see section 6.4 below.
\end{enumerate}
\end{remarks}

\begin{definition} A metric $d_\X^{\TT}$ will be defined on the tangent space $\TT_\X$ as follows: for  $g,h\in \TT_\X$, say $g\in L^2_s(X^2,\m^2)$, $h\in L^2_s(X'^2,\m'^2)$ with
$\auf X,\f,\m\zu=\auf X',\f',\m'\zu=\X$, we put
\[
d_\X^{\TT}(g,h)=\inf\Big\{
\| g-h\|_{L^2((X\times X')^2,\mu^2)}\spec
\mu\in \Cpl(\m,\m'),\quad \f=\f'\ \mu^2\text{-a.e. on }(X\times X')^2\Big\}.\]
\end{definition}

\begin{remarks}
\begin{enumerate}
\item
$d_\X^{\TT}$ is symmetric and satisfies the triangle inequality.
\item
$d_\X^{\TT}(g,h)=0$ if and only if $g\sim h$.
\item
Given $g,h\in \TT_{\X}$, say $g\in L^2_s(X^2,\m^2)$, $h\in L^2_s(X'^2,\m'^2)$ with
$\auf X,\f,\m\zu=\auf X',\f',\m'\zu=\X$, choose a common enlargement $(\ol X,\ol m)$ of $(X,\m)$ and $(X',\m')$
with embeddings $\phi: \ol X\to X$, $\phi': \ol X\to X'$. Since the spaces $(X,\f,\m)$ and $(X',\f',\m')$
are homomorphic we may assume without restriction that $\phi^*\f=\phi'^*\f'=:\ol\f$.
Put $\ol g=\phi^*g, \ol h={\phi'}^*h$. Then $g\sim \ol g$, $h\sim\ol h$ and
\[ d_\X^{\TT}(g,h)=d_\X^{\TT}(\ol g,\ol h)=\inf
\Big\{
\| \ol g- \ol h\|_{L^2({\ol X}^4,\mu^2)}\spec
\mu\in \Sym(\ol X,\ol \f,\ol m)\Big\}.\]
\end{enumerate}
\end{remarks}

\begin{lemma}
$d_\X^{\TT}$ is a cone metric on $\TT_\X$.
\end{lemma}

\begin{proof}
The claim will follow from the fact that for each
$g\in L^2_s(X^2,\m^2)$, $h\in L^2_s(X'^2,\m'^2)$ with $\|g\|_{L^2}=\|h\|_{L^2}=1$,
the quantity
\[\frac1{2st}\Big[d_\X^{\TT}(sg,th)^2-s^2-t^2\Big]\]
is independent of $s$ and $t\in (0,\infty)$. The latter can be seen as follows
\begin{eqnarray*}
\lefteqn{\frac1{2st}\Big[d_\X^{\TT}(sg,th)^2-s^2-t^2\Big]}\\
&=&
\inf\Big\{\frac1{2st}\Big[\|sg-th\|^2_{L^2((X\times X')^2,\mu^2)}-s^2-t^2\Big]
\spec \mu\in\Cpl(\m,\m'), \ \f=\f' \ \mu^2\text{-a.e.}
\Big\}
\\
&=&
-\sup\Big\{\langle g,h\rangle_{L^2((X\times X')^2,\mu^2)}
\spec \mu\in\Cpl(\m,\m'),\ \f=\f' \ \mu^2\text{-a.e.}
\Big\}.
\end{eqnarray*}
\end{proof}

\begin{definition}
The \emph{exponential map} $\EExp_\X:  \TT_\X\to\YY$ \index{e@$\EExp_\X$} is defined by
\[g\mapsto \auf X, f+g,\m \zu\]
for $g\in L^2_s(X^2,\m^2)$.
\end{definition}

\begin{remark}
This definition is consistent since $g\sim g'$ implies $\auf X,\f+g,\m\zu=\auf X',\f'+g',\m'\zu$.
Indeed, given $g\in L^2_s(X^2,\m^2)$, $g'\in L^2_s(X'^2,\m'^2)$ with
$\auf X,\f,\m\zu=\auf X',\f',\m'\zu$, we know that $g\sim g'$ if and only if there exists a measure
$\mu\in\Cpl(\m,\m')$ such that $\f=\f'$ and $g=g'$ $\mu^2$-a.e. This implies $\f+tg=\f'+tg'$ $\mu^2$-a.e. for every $t\in\R$ which in turn implies
\[ \auf X,\f+tg,\m\zu=\auf X',\f'+tg',\m'\zu\]
for every $t\in\R$. In other words,
\[\EExp_\X(tg)=\EExp_\X(tg')\]
for every $t$. Thus $\EExp$ is well-defined.
\end{remark}

\begin{definition} For $\X\in\YY$ we define the map $\tau_\X^\YY: \TT_\X\to [0,\infty]$ by
\[\tau_\X^\YY(g)=\sup\Big\{t\ge 0\spec \big(\EExp_\X(s g)\big)_{s\in [0,t]} \text{ is geodesic in }\YY\Big\}.\]
Analogously, for $\X\in\ol\XX$, we define
\[\tau_\X^{\ol\XX}(g)=\sup\Big\{t\ge 0\spec \big(\EExp_\X(s g)\big)_{s\in [0,t]} \text{ is geodesic in }\ol\XX\Big\}.\]
\end{definition}\index{tt@$\tau_\X^\YY, \tau_\X^{\ol\XX}$}

Recall that $T_\X\YY$, the \emph{tangent cone} at $\X$ in the sense of Alexandrov geometry (cf. section 4.3.), is defined as the cone over its unit sphere $T_\X^1\YY$ which in turn is the completion of the space of geodesic directions $\mathring T_\X^1\YY$. Equivalently,
$T_\X\YY$ can be considered as the completion of $\mathring T_\X\YY$
which in turn is the cone over $\mathring T_\X^1\YY$.
Denote the metric on $T_\X\YY$ by $d_\X^\YY$.

\begin{theorem} \begin{enumerate}
\item
The set $\big\{g\in \TT_\X\spec \tau^\YY_\X(g)>0\big\}$ can be identified with the cone  $\mathring T_\X\YY$ via $g\mapsto \big(\EExp_\X(s g)\big)_{s\in [0,\tau^\YY_\X(g))}$.
\item
For each $g\in \TT_\X$, say $g\in L^2_s(X^2,\m^2)$, with $\tau^\YY_\X(g)>0$,
\[d_\X^{\YY}(g,0)=\|g\|_{T_\X\YY}=\|g\|_{L^2(X^2,\m^2)}=
d_\X^{\TT}(g,0).\]
\item
For all $g,h\in \mathring T_\X\YY$,
\[
d_\X^{\YY}(g,h)\le
d_\X^{\TT}(g,h).\]
\end{enumerate}
\end{theorem}

\begin{proof} (i),(ii) By definition, for each $g\in \TT_\X$ with $\tau^\YY_\X(g)>0$,
$\big(\EExp_\X(s g)\big)_{s\in [0,\tau^\YY_\X(g))}$ is a geodesic in $\YY$. Hence, $g$ is an element of the cone
$\mathring T_\X\YY$.
Conversely, each geodesic $(\X_s)_{s\in [0,t]}$ in $\YY$ emanating from $\X=\X_0$ can be represented as
$\X_s=\EExp_\X(s g)$ for suitable $g\in L^2_s(X^2,\m^2)$ and suitable representative  $(X,\f,\m)$ of $\X$.

(iii)
To prove the inequality between the distances on $T_\X\YY$ and $\TT_\X$ it suffices to verify the analogous inequality between the induced distances on the respective unit spheres $T^1_\X\YY$ and $\TT^1_\X$ (since both spaces are cones over their respective unit spheres).
Let representatives $(X,\f,\m)$ and $(X',\f',\m')$ of $\X$ be given as well as
unit tangent vectors $g\in L^2_s(X^2,\m^2)$ and $g'\in  L^2_s(X'^2,\m'^2)$.
Then
\[ d^{\YY,1}_{\X}(g,g')=\measuredangle \Big( \big(\EExp_\X(sg)\big)_{s\ge0}, \big(\EExp_\X(tg')\big)_{t\ge0}\Big)\]
whereas
\[\cos d^{\TT,1}_\X(g,g')=\sup\Big\{\langle g,g'\rangle_{L^2((X\times X')^2,\mu^2)}
\spec
\mu\in \Cpl(\m,\m'),\quad \f=\f'\ \mu^2\text{-a.e.}
\Big\}
\]
According to Proposition \ref{alex-sym} (with $\f,\f'$ in the place of $\f_0,\f_0$ and $g,g'$ in the place of $\f_1-\f_0, \f_1'-\f_0$)
\[\cos d^{\TT,1}_\X(g,g')\le \cos d^{\YY,1}_\X(g,g').\]
This proves the claim.
\end{proof}

\begin{corollary} \begin{enumerate}
\item
The set $\big\{g\in \TT_\X\spec \tau^{\ol\XX}_\X(g)>0\big\}$ can be identified with the cone  $\mathring T_\X\ol\XX$.
\item
For all $g,h\in \mathring T_\X\ol\XX\subset \mathring T_\X\YY$,
\[
d_\X^{\ol\XX}(g,h)=
d_\X^{\YY}(g,h)
\le
d_\X^{\TT}(g,h).\]
\end{enumerate}
\end{corollary}

\subsection{Tangent Spaces  -- A Comprehensive Alternative Approach}

Recall the fact (sect. 5.2) that the space $\YY$ of gauged measure space is isometric to a quotient space $\LL$ of $L^2_s(I^2,\Leb^2)$. The tangent spaces $\TT_\X$ for $\X=\auf X,\f,\m\zu\in\YY$, therefore, will be in one-to-one correspondence with the tangent spaces $\TT_\f$ (to be defined below) for $\f\in\LL$.

Given a function $\f\in L^2_s(I^2,\Leb^2)$, we put
\[\Sym(\f)=\Big\{(\psi_0,\psi_1)\in\Inv(I,\Leb^1)^2\spec \psi_0^*\f=\psi_1^*\f\Big\}.\]
Note that this will be a group, isomorphic to the previously introduced $\Sym(I,\f,\Leb^1)$, provided we identify all
$(\psi_0,\psi_1)\in\Sym(\f)$ which satisfy $\psi_0=\psi_1$. (The latter should be understood as identity $\Leb^2$-a.e. as usual.)
We say that $\f$ has no symmetries if \[
\forall\psi_0,\psi_1\in\Inv(I,\Leb^1):\ \psi_0^*\f=\psi_1^*\f \quad\Longrightarrow\quad \psi_0=\psi_1.\]

\begin{definition}
The \emph{tangent space}
\[\TT_\f=L^2_s(I^2,\Leb^2)/\Sym(\f)\]
is the quotient space of $L^2_s(I^2,\Leb^2)$ with respect to the equivalence relation
\[g\sim h\quad\Longleftrightarrow\quad \exists (\psi_0,\psi_1)\in\Sym(\f)\spec \psi_0^* g=\psi_1^* h.\]
It is a metric space with metric
\[\d_\f(g,h)=\inf\Big\{\|\psi_0^* g-\psi_1^* h\|_{L^2(I^2,\Leb^2)}\spec (\psi_0,\psi_1)\in\Sym(\f)\Big\}.\]
If $\f$ has no symmetries then $\TT_\f=L^2_s(I^2,\Leb^2)$. In particular, then $\TT_\f$ is a Hilbert space.
\end{definition}
This definition justifies to regard the tangent spaces $\TT_\f$ (and thus also the previously defined tangent spaces $\TT_\X$) as infinite dimensional \emph{Riemannian orbifolds}.

\begin{definition}
The \emph{exponential map} $\EExp_\f:  \TT_\f\to\LL$ is defined by
\[g\mapsto \auf f+g\zu.\]
Equivalently, it may be considered as map $\EExp_\f:  \TT_\f\to\YY$ with
\[g\mapsto \auf I, f+g,\Leb^1\zu.\]
\end{definition}

Indeed, however, the measure space $(I,\Leb^1)$ does not play any particular role. It is just one of many possible enlargements of a given space. It can be replaced by any other standard Borel space without atoms. Thus for any gauged measure space $(X,\f,\m)$ without atoms we may define \index{t@$\TT_\f$, $\TT_{(X,\f,\m)}$}
\[\TT_{(X,\f,\m)}=L^2_s(X^2,\m^2)\big/ \Sym(X,\f,\m)\]
where two elements $g$ and $h$ in $L^2_s(X^2,\m^2)$ are identified if there exists a measure $\mu\in \Sym(X,\f,\m)$ -- a self-coupling of $\m$ which leaves $\f$ invariant --  such that
\[g(x,y)=h(x',y')\qquad\text{for $\mu^2$-a.e.}\big((x,x'),(y,y')\big)\in X^4.\]
For $g\in \TT_{(X,\f,\m)}$ we put \[\EExp_{(X,\f,\m)}(g)=\auf X,\f+g,\m\zu.\]
\index{e@$\EExp_{(X,\f,\m)}$, $\EExp_\f$}

\begin{corollary}
For each gauged measure space $(X,\f,\m)$ without atoms, the space
$\TT_{(X,\f,\m)}$ may be identified with the tangent space $\TT_\X$ where $\X$ denotes the homomorphism class of $(X,\f,\m)$. The exponential maps $\EExp_{(X,\f,\m)}$ and $\EExp_\X$ are defined consistently.
\end{corollary}

\subsection{Ambient Gradients}\label{sec:ambgrad}

\begin{definition} A function $\U:\YY\to\R$ is called \emph{strongly differentiable} at $\X\in\YY$
\begin{itemize}
\item
if the directional derivative \index{d@$D_h$}
\[D_h\U(\X):=\lim_{t\searrow0} \frac1t\big[\U(\EExp_\X(th))- \U(\X)\big]  \]
exists for every $h\in \TT_\X$ and
\item
if there exists a tangent vector $g\in\TT_\X$ such that
\begin{eqnarray*}
D_h\U(\X)=
\langle g,h\rangle_{L^2((X\times X')^2,\mu^2)}
\end{eqnarray*}
for every $h\in \TT_\X$ and every $\mu\in \Cpl(\m,\m')$ with $\f=\f'\ \mu^2\text{-a.e. on }(X\times X')^2$.

Here we assumed $g\in L^2_s(X^2,\m^2)$ and $h\in L^2_s(X'^2,\m'^2)$ with $(X,\f,\m)$ and $(X',\f',\m')$ being two representatives of $\X$.
\end{itemize}
The tangent vector $g\in\TT_\X$ is then called \emph{ambient gradient} of $\U$ at $\X$. It is denoted by \index{ambient gradient}
\index{g@$\Grad$}
\[g=\Grad\U(\X).\]
\end{definition}

\begin{lemma} For any function $\U:\YY\to\R$ which is strongly differentiable at $\X\in\YY$, the ambient gradient is unique and satisfies
\[\|\Grad\U(\X)\|=\sup\big\{ D_h\U(\X)\spec \|h\|=1\big\}.\]
\end{lemma}
Here $\|h\|=\|h\|_{L^2(X'^2,\m'^2)}=d_\X^{\TT}(h,0)$ for $h\in L^2_s(X'^2,\m'^2)$ and
$\|g\|=\|g\|_{L^2(X^2,\m^2)}=d_\X^{\TT}(g,0)$ for $g=\Grad\U(\X)\in L^2_s(X^2,\m^2)$ with representatives
 $(X,\f,\m)$ and $(X',\f',\m')$  of $\X$.

\begin{proof} \emph{Uniqueness.}
In order to be the  ambient gradient $\Grad\U(\X)$, a function $g\in L^2_s(X^2,\m^2)$ in particular has to satisfy
\[D_h\U(\X)=\langle g,h\rangle_{L^2(X^2,\m^2)}\qquad\text{for every }h\in L^2_s(X^2,\m^2).\]
(Indeed, choose $X'=X$ and $\mu$ to be diagonal coupling.) The latter property determines $g$ (if it exists) uniquely within $ L^2_s(X^2,\m^2)$.
Two ambient gradients  $g$ and $g'$ defined on two representatives of $\X$ may always be extended (via pull back) to functions on a common enlargement. Thus the ambient gradient is unique (if it exists).

\emph{Norm identity.}
For each $h$ and each coupling $\mu$ (which leaves $\f$ invariant) as above
\begin{eqnarray*}
D_h\U(\X)&=&
\langle g,h\rangle_{L^2((X\times X')^2,\mu^2)}\\
&\le&
\|g\|_{L^2((X\times X')^2,\mu^2)}\cdot\|h\|_{L^2((X\times X')^2,\mu^2)}\\
&=&
\|g\|_{L^2(X^2,\m^2)}\cdot\|h\|_{L^2(X'^2,\m'^2)}.
\end{eqnarray*}
Thus
$D_h\U(\X)\le\|g\|$
for each $h\in\TT_\X$ with $\|h\|=1$. On the other hand, assume without restriction that $\|g\|>0$ and choose $h=\frac1{\|g\|}g$ and $\mu=$  diagonal coupling of $\m$ to obtain
\[D_h\U(\X)=\frac1{\|g\|}\langle g,g\rangle_{L^2((X\times X)^2,\mu^2)}=\|g\|.\]
\end{proof}

\begin{theorem}
Let $\U:\YY\to\R$ be Lipschitz continuous, semiconcave and strongly differentiable in  $\X\in\YY$. Assume that the ambient gradient
$\Grad\U(\X)$  lies in $\mathring T_\X\YY$ or, in other words, assume that $\tau_\X^\YY(g)>0$ for $g=\Grad\U(\X)$. Then
\[\Grad\U(\X)=\nabla^\YY \U(\X).\]
Here $\nabla^\YY$ denotes the gradient in the sense of Alexandrov geometry as introduced e.g. in section 4.3, see
\cite{Pl}.
\end{theorem}

\begin{proof} Put $g=\Grad\U(\X)$.
According to the argumentation in the proof of the previous Lemma,
$g_1=\frac1{\|g\|}g$ is the maximizer of
\[h\mapsto D_h\U(\X)\]
in $\TT_\X^1$. Therefore, assuming that $g\in \mathring T_\X\YY$, the normalized $g_1$  in particular is the maximizer
of $h\mapsto D_h\U(\X)$
in $\mathring T^1_\X\YY$.
Up to a multiplicative constant, this already characterizes the gradient of $\U$ at $\X$ in the sense of Alexandrov geometry. The previous Lemma finally yields the equivalence of the norms (= lengths of tangent vectors) in both spaces.
\end{proof}

\begin{corollary}
Let $\U:\YY\to\R$ be defined on all of $\YY$ and assume that its restriction to $\ol\XX$ is Lipschitz continuous and semiconcave. Assume furthermore that $\U$ is strongly differentiable in  $\X\in\ol\XX$ and  that the ambient gradient
$\Grad\U(\X)$  lies in $\mathring T_\X\ol\XX$. Then
\[\Grad\U(\X)=\nabla^{\ol\XX} \U(\X).\]
\end{corollary}

\section{Semiconvex Functions on $\YY$ and their Gradients}

\subsection{Polynomials on $\YY$ and their Derivatives}

A striking consequence of the detailed knowledge of the geometry of $\YY$ is that for major classes of functions on $\YY$ one can explicitly calculate sharp bounds for derivatives of any order. Of particular interest will be bounds for first and second derivatives.

An important class of `smooth' functions on $\YY$ is given by \emph{polynomials} of order $n\in\N$. These are functions $\U:\YY\to\R$ of the form
\begin{equation}
\label{U}
\U(\X)=\int_{X^n}u\bigg( \Big( \d(x^i,x^j)\Big)_{1\le i<j\le n}\bigg)d\m^n(x)
\end{equation}
for $\X=\auf X,\d,\m\zu$
where $u:\R^{\frac{n(n-1)}2} \to\R$ is any  Borel function  which grows at most quadratically.  Mostly, $u$ will be differentiable with bounded derivatives of any order. For our purpose, derivatives of order 1 and 2 are sufficient. Here and in the sequel, $\m^n=\m\otimes\ldots\otimes \m$ denotes the $n$-fold product of $\m$ and $x=(x^1,\ldots,x^n)\in X^n$ whereas $\xi=(\xi_{ij})_{1\le i<j\le n}\in\R^{\frac{n(n-1)}2}$. Deviating from the convention of the previous Chapter, the gauge function will be denoted by $\d$. In most cases of application, indeed, it will be a pseudo metric.

Note that all these functions $\U$ are functions of homomorphism classes,
i.e. the definition of $\U(\X)$ does not depend on the choice of the representative $(X,\d,\m)$ of $\X$, see Proposition \ref{reconstruct}.
Moreover, it might be worthwhile to mention that the set of polynomials of any order  separates points in $\XX$ (\cite{Gro}, cf. also  \cite{gpw}, Prop. 2.6)

Recall that each geodesic $\geod{0}{1}$ in $\YY$ can  be represented as
\begin{equation}\label{gerep}
	\X_t= \auf X_0\times X_1, \d_0+ t(\d_1-\d_0),\ol\m\zu
\end{equation}for given representatives of  $\X_0$ and $\X_1$ and a suitable choice of $\ol\m\in\Opt(\m_0,\m_1)$.
Thus $\U$ is represented along the geodesic $\geod{0}{1}$ as
\begin{equation}\label{formulaU}
\U\Big(\X_t\Big)=\int_{(X_0\times X_1)^n}u\bigg( \Big( \d_0(x_0^i,x_0^j)+t\big(\d_1(x_1^i,x_1^j)-\d_0(x_0^i,x_0^j)\big)\Big)_{1\le i<j\le n}\bigg)d\ol\m^n(x_0,x_1)
\end{equation}
where $(x_0,x_1)$ now stands for the $n$-tuple $((x_0^i,x_1^i))_{1\le i\le n}$ of points $(x_0^i,x_1^i)\in X_0\times X_1$.

\begin{lemma}\label{derivf}
	Assume that $\U:\YY\rightarrow \R$ is given by formula \eqref{U} with
	$u\in C^2(\R^{\frac{n(n-1)}2},\R_+)$ with bounded derivatives.
				Then for each  geodesic $\geod{0}{1}$ in $\YY$ (represented	as in \eqref{gerep}):
			\begin{eqnarray*}
									\frac{d}{dt} \U(\X_t)&= &
						\sum_{1\le i<j\le n}\int_{(X_0\times X_1)^n} \frac{\partial}{\partial \xi_{ij}}
u\bigg( \Big( \d_0(x_0^p,x_0^q)+t\big(\d_1(x_1^p,x_1^q)-\d_0(x_0^p,x_0^q)\big)\Big)_{1\le p<q\le n}\bigg)
\cdot\\
&&\qquad\qquad\qquad\qquad\cdot\bigg(\d_1(x_1^i,x_1^j)-\d_0(x_0^i,x_0^j)\bigg)
d\ol\m^n(x_0,x_1)
\end{eqnarray*}
and
\begin{eqnarray*}
									\frac{d^2}{dt^2} \U(\X_t)&= &
					\sum_{1\le k<l\le n}	\sum_{1\le i<j\le n}\int_{(X_0\times X_1)^n} \frac{\partial}{\partial \xi_{kl}}\frac{\partial}{\partial \xi_{ij}}
u\bigg( \Big( \d_0(x_0^p,x_0^q)+t\big(\d_1(x_1^p,x_1^q)-\d_0(x_0^p,x_0^q)\big)\Big)_{1\le p<q\le n}\bigg)
\cdot\\
&&\qquad\qquad\qquad\qquad\cdot\bigg(\d_1(x_1^i,x_1^j)-\d_0(x_0^i,x_0^j)\bigg)\cdot\bigg(\d_1(x_1^k,x_1^l)-\d_0(x_0^k,x_0^l)\bigg)
d\ol\m^n(x_0,x_1)
\end{eqnarray*}
			for all $t\in (0,1)$ and as a right limit for $t=0$.
		\end{lemma}

\begin{proof}
	These formulae are straightforward consequences of the representation of \eqref{formulaU}: interchanging order of differentiation (w.r.t.\ $t$) and integration
	(w.r.t.\ $x_0^i,x_1^i$) and application of chain rule.
\end{proof}

Note that at $t=0$ the previous formulas simplify, e.g.
\begin{eqnarray*}
									\frac{d}{dt} \U(\X_t)\big|_{t=0}&= &
						\sum_{1\le i<j\le n}\int_{(X_0\times X_1)^n} \frac{\partial}{\partial \xi_{ij}}
u\bigg( \Big( \d_0(x_0^p,x_0^q)\Big)_{1\le p<q\le n}\bigg)
\cdot\bigg(\d_1(x_1^i,x_1^j)-\d_0(x_0^i,x_0^j)\bigg)
d\ol\m^n(x_0,x_1).
\end{eqnarray*}

\begin{theorem}\label{thm U} Let $n\in\N$ as well as numbers $\lambda,\kappa\in\R$ be given and let $u:\R^{\frac{n(n-1)}2}\to\R$ be continuous and bounded (or with at most quadratic growth).
\begin{enumerate}
\item
If $u$ is  $\lambda$-Lipschitz continuous on $\R^{\frac{n(n-1)}2}$ then $\U$ is  $\lambda'$-Lipschitz continuous on $\YY$ for $\lambda'=\lambda\cdot \frac{n(n-1)}2$.
\item
If $u$ is $\kappa$-convex  on $\R^{\frac{n(n-1)}2}$ then $\U$ is $\kappa'$-convex on $\YY$ for $\kappa'=\kappa\cdot \frac{n(n-1)}2$.
\end{enumerate}
\end{theorem}
\begin{proof}
(i) Approximating $u$ by $u_k\in\mathcal C^2$ (with bounded derivatives), we may apply the estimates of the previous Lemma. Thus for any geodesic $(\X_t)_t$ in $\YY$
\begin{eqnarray*}
\Big|\frac{d}{dt} \U(\X_t)\Big|&\le &
\lambda\cdot
\sum_{1\le i<j\le n}\int_{(X_0\times X_1)^n}\Big|\d_1(x_1^i,x_1^j)-\d_0(x_0^i,x_0^j)\Big|
d\ol\m^n(x_0,x_1)\\
&\le&
\lambda\cdot
\sum_{1\le i<j\le n}\bigg(\int_{X_0\times X_1}\int_{X_0\times X_1}\Big|\d_1(x_1^i,x_1^j)-\d_0(x_0^i,x_0^j)\Big|^2
d\ol\m(x_0^i,x_1^i)d\ol\m(x_0^j,x_1^j)\bigg)^{1/2}\\
&=&
\lambda\cdot \frac{n(n-1)}2\cdot \DD(\X_0,\X_1).
\end{eqnarray*}
Since $\DD(X_0,\X_1)$ is the speed of the geodesic $(\X_t)_t$, this implies
\[\Lip \U\le \lambda\cdot \frac{n(n-1)}2.\]

(i')
A more direct proof, avoiding any approximation argument, is based on the explicit representation formula
(\ref{formulaU}).
It immediately yields
\begin{eqnarray*}
|\U(\X_1)-\U(\X_0)|&\le&
\int_{(X_0\times X_1)^n}\bigg|u\bigg( \Big(\d_1(x_1^i,x_1^j)\Big)_{1\le i<j\le n}\bigg)-u\bigg(\Big(\d_0(x_0^i,x_0^j)\Big)_{1\le i<j\le n}\bigg)\bigg|\,d\ol\m^n(x_0,x_1)\\
&\le&
\lambda\cdot\bigg(\int_{(X_0\times X_1)^n}
\bigg| \Big(\d_1(x_1^i,x_1^j)-\d_0(x_0^i,x_0^j)\Big)_{1\le i<j\le n}\bigg|^2
\,d\ol\m^n(x_0,x_1)\bigg)^{1/2}\\
&=&
\frac{n(n-1)}2\lambda\cdot\bigg(\int_{(X_0\times X_1)^2}
\bigg| \d_1(x_1^1,x_1^2)-\d_0(x_0^1,x_0^2)\bigg|^2
\,d\ol\m^2(x_0,x_1)\bigg)^{1/2}\\
&=&
\frac{n(n-1)}2\lambda\cdot\DD\big(\X_1,\X_0\big).
\end{eqnarray*}

(ii) Recall that for smooth $u$, $\kappa$-convexity is equivalent to
\[
\sum_{1\le k<l\le n}	\sum_{1\le i<j\le n} \frac{\partial}{\partial \xi_{kl}}\frac{\partial}{\partial \xi_{ij}}
u(\xi)\cdot V_{ij}\cdot V_{kl}\ge
\kappa\cdot \sum_{1\le i<j\le n} |V_{ij}|^2\qquad(\forall \xi, V\in\R^{\frac{n(n-1)}2}).\]
Thus, similarly to the previous argumentation, Lemma \ref{derivf} in the case of $\kappa$-convex $u$ now yields
\begin{eqnarray*}
\frac{d^2}{dt^2} \U(\X_t)&\ge &
\kappa\cdot
\sum_{1\le i<j\le n}\int_{(X_0\times X_1)^n}\Big|\d_1(x_1^i,x_1^j)-\d_0(x_0^i,x_0^j)\Big|^2
d\ol\m^n(x_0,x_1)\\
&=&
\kappa\cdot \frac{n(n-1)}2\cdot \DD^2(\X_0,\X_1).
\end{eqnarray*}
This proves the claim.

(ii')
Again, a more direct proof (without approximation) is possible, based on (\ref{formulaU}). It implies
\begin{eqnarray*}
\lefteqn{\U(\X_t)-t\U(\X_1)-(1-t)\U(\X_0)}\\
&=&
\int_{(X_0\times X_1)^n}\bigg[
u\bigg(
\Big( t\,\d_1(x_1^i,x_1^j)+(1-t)\,\d_0(x_0^i,x_0^j)\Big)_{1\le i<j\le n}
\bigg)\\
&&\qquad\quad
-t\,u\bigg( \Big(\d_1(x_1^i,x_1^j)\Big)_{1\le i<j\le n}\bigg)-(1-t)\,u\bigg(\Big(\d_0(x_0^i,x_0^j)\Big)_{1\le i<j\le n}\bigg)\bigg]\,d\ol\m^n(x_0,x_1)\\
&\le&
-\frac\kappa2\cdot t(1-t)\cdot\int_{(X_0\times X_1)^n}
\bigg| \Big(\d_1(x_1^i,x_1^j)-\d_0(x_0^i,x_0^j)\Big)_{1\le i<j\le n}\bigg|^2
\,d\ol\m^n(x_0,x_1)\\
&=&-\frac\kappa2\cdot t(1-t)\cdot
\frac{n(n-1)}2\cdot\int_{(X_0\times X_1)^2}
\bigg| \d_1(x_1^1,x_1^2)-\d_0(x_0^1,x_0^2)\bigg|^2
\,d\ol\m^2(x_0,x_1)\\
&=&
-\frac\kappa2\cdot \frac{n(n-1)}2\cdot t(1-t)\cdot\DD^2\big(\X_1,\X_0\big).
\end{eqnarray*}
This proves the $\kappa'$-convexity of $\U$ for $\kappa'=\kappa\cdot \frac{n(n-1)}2$.
\end{proof}

\begin{remark}
The formulas in Lemma \ref{derivf} for derivatives of $t\mapsto\U(\X_t)$ not only hold for geodesics $(\X_t)_{t\in[0,1]}$ but for all curves $(\X_t)_{t\ge0}$ in $\YY$ induced by exponential maps:
\[\X_t=\EExp_\X(tg)\qquad\text{for some }g\in\TT_\X.\]
For instance, the directional derivative of $\U$ at $\X=\auf X,\d,\m\zu$ in direction $g\in L^2_s(X^2,\m^2)$  is given by
	\begin{eqnarray}
									D_g\U(\X)&= &
						\sum_{1\le i<j\le n}\int_{X^n} \frac{\partial}{\partial \xi_{ij}}
u\Big( \big( \d(x^p,x^q)\big)_{1\le p<q\le n}\Big)
\cdot g(x^i,x^j)\,
d\m^n(x)
\end{eqnarray}
This  leads to an explicit representation formula for the ambient gradient of $\U$ at $\X$.
\end{remark}
To this end, given $u$ and $(X,\d,\m)$ as above,
put
\begin{equation}
u^\d_{ij}(x)=\frac{\partial}{\partial \xi_{ij}}
u\Big( \big( \d(x^p,x^q)\big)_{1\le p<q\le n}\Big),
\end{equation}
for $x=(x^1,\ldots,x^n)\in X^n$.

\begin{theorem}\label{grad U}
The ambient gradient $\Grad\U(X)$ of the function $\U$  at the point $\X=\auf X,\d,\m\zu\in\YY$  is the function $f\in L^2_s(X^2,\m^2)$ given by
$f(y,z)=\frac12	\tilde	f(y,z)+\frac12\tilde f(z,y)$ with
 \begin{eqnarray*}
			\tilde f(y,z)&=&				
						\sum_{1\le i<j\le n}\int_{X^{n-2}}
u_{ij}^\d\Big(
x^1,\ldots,x^{i-1},y,x^{i+1},\ldots, x^{j-1},z,x^{j+1},\ldots,x^{n}\Big)
\\
&&\qquad\qquad\qquad\qquad\qquad
d\m^{n-2}(x^1,\ldots,x^{i-1},x^{i+1},\ldots, x^{j-1},x^{j+1},\ldots,\ldots,x^n).
\end{eqnarray*}
Moreover, $\Grad(-\U)(X)=-\Grad\U(X)$.
\end{theorem}

\begin{proof}
For a given representative $(X,\d,\m)$ of $\X$ put
$f$ as above. Now in addition, let  $g\in \TT_\X$ be given. Let us first consider the particular case that $g$ is given on the same representative, i.e. $g\in L^2_s(X^2,\m^2)$.
Then
\begin{eqnarray*}
D_g \U(\X)&= &				
\sum_{1\le i<j\le n}\int_{X^n}
\frac{\partial}{\partial \xi_{ij}}
u\Big( \big( \d(x^p,x^q)\big)_{1\le p<q\le n}\Big)
\cdot g(x^i,x^j)\,
d\m^n(x)\\
&= &	
\sum_{1\le i<j\le n}
\int_{X^n}
 u_{ij}^\d(x^1,\ldots,x^n)
\cdot g(x^i,x^j)\,
d\m^n(x)\\
&= &	
\sum_{1\le i<j\le n}\int_{X^2}\int_{X^{n-2}}
u_{ij}^\d\Big(
x^1,\ldots,x^{i-1},y,x^{i+1},\ldots, x^{j-1},z,x^{j+1},\ldots,x^{n}\Big)\cdot g(y,z)
\\
&&\qquad\qquad\qquad\qquad\qquad
d\m^{n-2}(x^1,\ldots,x^{i-1},x^{i+1},\ldots, x^{j-1},x^{j+1},\ldots,\ldots,x^n)\,d\m^2(y,z)\\
&= &	
\int_{X^2}
f(y,z)
\cdot g(y,z)\,
d\m^2(y,z)\ =\ \langle  f, g\rangle_{L^2( X^2,\m^2)}.
\end{eqnarray*}
Now let us consider the general case: \  $g\in L^2_s(X'^2,\m'^2)$ for some representative $(X',\d',\m')$ of $\X$. Put $\ol X=X\times X'$ and let $\ol\m$ be \emph{any} coupling of $\m$ and $\m'$ such that $\d=\d'$ $\ol\m^2$-a.e. on $\ol X^2$. Choose $\ol\d$ on $\ol X^2$ which coincides a.e. with $\d$ (and $\d'$) and define $\ol f,\ol g\in L^2_s(\ol X^2,\ol\m^2)$ by
$\ol g(\ol y,\ol z)=g(y',z')$ for $\ol y=(y,y'), \ol z=(z,z')\in X\times X'$,
\begin{eqnarray*}
\ol f(\ol y,\ol z)&=&
\sum_{1\le i<j\le n}\int_{\ol X^{n-2}}\frac12\Big[
u_{ij}^{\ol\d}\Big(
\ol x^1,\ldots,\ol x^{i-1},\ol y,\ol x^{i+1},\ldots, \ol x^{j-1},\ol z,\ol x^{j+1},\ldots,\ol x^{n}\Big)\\
&&\qquad\qquad\qquad\quad
+u_{ij}^{\ol\d}\Big(
\ol x^1,\ldots,\ol x^{i-1},\ol z,\ol x^{i+1},\ldots, \ol x^{j-1},\ol y,\ol x^{j+1},\ldots,\ol x^{n}\Big)\Big]\\
\\
&&\qquad\qquad\qquad\qquad\qquad
d\ol{\m}^{n-2}(\ol x^1,\ldots,\ol x^{i-1},\ol x^{i+1},\ldots, \ol x^{j-1},\ol x^{j+1},\ldots,\ldots,\ol x^n).
\end{eqnarray*}
Then
$\ol f(\ol y,\ol z)=f(y,z)$ for $\ol y=(y,y'), \ol z=(z,z')\in X\times X'$
since $\ol\d=\d$ $\ol\m^2$-a.e. on $\ol X^2$.
Repeating the previous calculation with $\ol X,\ol f,\ol g$ and $\ol \m$ in the place of $X,f,g$ and $\m$ yields
\begin{eqnarray}
D_g \U(\X)&= &	\langle \ol f,\ol g\rangle_{L^2(\ol X^2,\ol\m^2)}.
\end{eqnarray}
\end{proof}
\begin{remark}
\begin{itemize}
\item
Polynomials of degree 2 are of the form $\int_X\int_X u(\d(x,y))\,d\m(x)\,d\m(y)$. They had been used e.g. to define the $L^p$-size of $\X=\auf X,\d,\m\zu$.
\item
Polynomials of degree 3 can be used to determine whether a space
$\X\in\YY$ satisfies the triangle inequality, at least in a certain weak sense. For instance,
\[\U(\X)=\int_X\int_X\int_X
\Big[\d(x,z)-\d(x,y)-\d(y,z)\Big]^-
d\m(x)\,d\m(y)\,d\m(z)\]
vanishes if and only if
$\X\in\YY$ satisfies the triangle inequality $\m^3$-a.e., cf. Remark \ref{m2 vs m3}.

\item
Polynomials of degree 4 allow to determine whether a given curvature bound (either from above or from below) in the sense of Alexandrov is satisfied. This will be achieved through the functionals $\G_K$ and $\H_K$ to be considered below.
\end{itemize}
\end{remark}

\subsection{Nested Polynomials}

Besides polynomials, there are many other functions on $\YY$ for which derivatives (of any order) can be calculated explicitly. Among them are
functions
\[
	\U:\YY\rightarrow \R_+
\]
of the form
\begin{equation}\label{formulaF}
	\U(\X)= \int_X U\left( \int_X \eta(\d(x,y))d \m(y)\right) d\m(x)
\end{equation}
for given functions $U:\R_+\rightarrow \R_+$ and $\eta:\R_+\rightarrow \R_+$.
Any functional of this type will be called \emph{nested polynomial} of order 2.
The $\F$-functional to be considered in the next chapter will be of this type.

Note, however,  that analogous Lipschitz continuity and semiconvexity results can be easily obtained along the same lines of reasoning for more general classes of nested polynomials including for instance
\[
	\U(\X)= \int_X U\left( \int_X \eta(\d(x,y))d \m(y), \int_X \vartheta (\d(x,z))d \m(z) \right) d\m(x)
\]
or
\begin{align*}
	\U(\X)= \int_X \int_X U\left( \int_X \int_X \right.
	\theta\Big( \d(x,y),\d(x,z),\d(x,w),\d(y,z),\d(y,w),\d(z,w)\Big)
				\left.\vphantom{\int_X} d \m(w) d\m(z)\right) d\m(y) d\m(x).
\end{align*}

\begin{lemma}\label{derivf00}
	Assume that $\U:\YY\rightarrow \R_+$ is given by formula \eqref{formulaF} with
	$U\in C^2(\R_+,\R_+)$ and $\eta\in C^1(\R_+,\R_+)$, both with bounded derivatives.
	\begin{enumerate}
		\item\label{derivf1}
			Then for each  geodesic $\geod{0}{1}$ in $\YY$ and 	
			represented	as in \eqref{gerep}:
			\begin{equation*}
				\begin{split}
					\frac{d}{dt} \U(\X_t)= &
						\int_{X_0\times X_1} \bigg[ U' \bigg( \int_{X_0\times X_1}
						\eta\Big(\d_0(x,z)+t(\d_1(x,z)-\d_0(x,z))\Big) 	
						d\ol\m(z)\bigg)\\
						&\qquad\qquad\cdot \int_{X_0\times X_1}
						 \eta' \Big( \d_0(x,y)+t(\d_1(x,y)-\d_0(x,y)\Big) 	\cdot \Big(\d_1(x,y)-\d_0(x,y)\Big)
																	d\ol\m(y)\bigg] d\ol\m(x)
				\end{split}
			\end{equation*}
			for all $t\in (0,1)$ and as a right limit for $t=0$.
		\item\label{derivf2}
			Moreover,
			\begin{equation*}
				\begin{split}
					\frac{d^2}{dt^2} \U(\X_t)= &
						\int_{X_0\times X_1} \left[ U'' \left( \int_{X_0\times X_1}
						\eta\Big(\d_0(x,z)+t(\d_1(x,z)-\d_0(x,z))\Big) 	
						d\ol\m(z)\right)\right.\\
						&\phantom{\int_{X_0}}
						\left.\cdot \left(\int_{X_0\times X_1}
						\eta' \Big( \d_0(x,y)+t(\d_1(x,y)-\d_0(x,y)\Big)
						\cdot \Big(\d_1(x,y)-\d_0(x,y)\Big) d\ol\m(y)\right)^2 \right] d\ol\m(x)\\
						& + \int_{X_0\times X_1} \left[ U' \left( \int_{X_0\times X_1}
						\eta\Big(\d_0(x,z)+t(\d_1(x,z)-\d_0(x,z))\Big) 	
						d\ol\m(z)\right)\right.\\
						&\phantom{+\int_{X_0}}
						\left.\cdot \int_{X_0\times X_1}  \eta''
						\Big( \d_0(x,y)+t(\d_1(x,y)-\d_0(x,y)\Big)
							\cdot \Big(\d_1(x,y)-\d_0(x,y)\Big)^2 d\ol\m(y) \right] d\ol\m(x),
				\end{split}
			\end{equation*}
			again for all $t\in (0,1)$ and as a right limit at $t=0$.
	\end{enumerate}
\end{lemma}

\begin{proof} As in the case of polynomials, these formulae are straightforward consequences of the representations  \eqref{formulaF} and
	\eqref{gerep} which provide an explicit formula for the dependence of $\U(\X_t)$ on $t$:
\begin{equation*}									 \U(\X_t)=
						\int_{X_0\times X_1}  U \bigg( \int_{X_0\times X_1}
						\eta\Big(\d_0(x,y)+t(\d_1(x,y)-\d_0(x,y))\Big) 	
						d\ol\m(y)\bigg)\, d\ol\m(x).
						\end{equation*}
Now again, interchanging the order of differentiation  and integration
	 and applying the chain rule leads to the asserted formulas for the directional derivatives.
\end{proof}

\begin{remarks}
\begin{enumerate}
\item
	In the case $t=0$, using the abbreviation $w_0(x)=\int_{X_0} \eta(\d_0(x,z))d\m_0(z)$,
the previous formulas yield
	\begin{equation}\label{derivt0}
				\begin{split}
					\frac{d}{dt} \U(\X_t)\Big\vert_{t=0} = &
						\int_{X_0\times X_1}\int_{X_0\times X_1} U' (w_0(x)) \cdot
						 \eta' ( \d_0(x,y))\cdot \Big(\d_1(x,y)-\d_0(x,y)\Big)
						 	d\ol\m(y) d\ol\m(x),
				\end{split}
	\end{equation}
	\begin{equation*}
		\begin{split}
					\frac{d^2}{dt^2} \U(\X_t)\Big\vert_{t=0}= &
						\int_{X_0}  U''(w_0(x))\left[ \int_{X_0\times X_1}
						\eta' ( \d_0(x,y))	\cdot \Big(\d_1(x,y)-\d_0(x,y)\Big) d\ol\m(y)\right]^2 d\ol\m(x)\\
						& + \int_{X_0\times X_1}\int_{X_0\times X_1} U'(w_0(x))\cdot  \eta''(\d_0(x,y))\cdot
						\Big(\d_1(x,y)-\d_0(x,y)\Big)^2 d\ol\m(y) d\ol\m(x),		
		\end{split}
	\end{equation*}
\item More generally, for each $\X_0=\auf X_0,\d_0,\m_0\zu=\auf X_1,\d_1,\m_1\zu\in\YY$, each
$g\in L^2_s(X_1^2,\m_1^2)$ and each $\ol\m\in\Cpl(\m_0,\m_1)$ with $\d_0=\d_1$ $\ol\m^2$-a.e.
\begin{eqnarray*}
D_g\U(\X)&=&\int_{X_0\times X_1}\int_{X_0\times X_1} U' (w_0(x_0)) \cdot
						 \eta' ( \d_0(x_0,y_0))\cdot g(x_1,y_1)\,
						 	d\ol\m(y_0,y_1) d\ol\m(x_0,x_1).
\end{eqnarray*}
\end{enumerate}
\end{remarks}

\begin{corollary}\label{cor47a}
The ambient gradient of $\U$ at the point $\X=\auf X,\d,\m\zu$ is given by the function $f=\Grad\U(\X) \in L^2(X^2,\m^2)$ defined as
\begin{equation}\label{gradfyy}
		f(x,y)= \frac{1}{2}\Big(U'(w(x))+U'(w(y))\Big)
		\cdot \eta'(\d(x,y))
	\end{equation}
where $w(.):= \int_{X} \eta(d(.,z)) d\m(z)$.
In particular,
\[\|\Grad\U(\X)\| = \frac{1}{2}
		\left[ \int_{X} \int_{X} \Big[U'(w(x))+U'(w(y))\Big]^2\cdot \eta'(\d(x,y))^2
								d\m(y) d\m(x)\right]^{\frac{1}{2}}.
	\]
\end{corollary}

\begin{theorem}
	\begin{enumerate}
		\item\label{lip}
			If $U$ and $\eta$ are Lipschitz functions on $\R_+$, then $\U$ is a Lipschitz function
			on $(\YY,\DD)$ with
			\[
				\Lip(\U)\leq \Lip(U)\cdot \Lip(\eta).
			\]
		\item\label{kconv}
			Assume that $U,\eta\in C^2(\R_+)$ with
			\begin{align*}
				U'\geq -L,\quad U''\geq -\lambda
				\qquad \text{and} \qquad
				|\eta'|\leq C_1, \quad \eta''\leq C_2
			\end{align*}
			for some numbers $L,\lambda,C_1,C_2\in\R_+$.
			Then $\U$ is $\kappa$-convex on $(\YY,\DD)$ with
			\[
				\kappa\geq -\lambda\cdot C_1^2-L\cdot C_2.
			\]
	\end{enumerate}
\end{theorem}

\begin{proof}
	(\ref{lip}) For Lipschitz continuous $U$ and $\eta$, the formula in Lemma~\ref{derivf00}
	(\ref{derivf1}), yields
	\begin{equation*}
	\begin{split}
		\bigg|\frac{d}{dt} \U(\X_t)\bigg|& \leq \Lip(U)\cdot \Lip(\eta)\cdot
					\int_{X_0\times X_1} \int_{X_0\times X_1} |\d_1(x,y)-\d_0(x,y)| d\ol\m(y) d\ol\m(x) \\
				& \leq \Lip(U)\cdot \Lip(\eta) \cdot \DD(\X_0,\X_1)
	\end{split}
	\end{equation*}
	and thus
	\[
		\Lip(\U)\leq \Lip(U)\cdot \Lip(\eta).
	\]
	(Indeed, a more direct estimation is possible without any $t$-differentiation.)\\
	(\ref{kconv}) The given bounds on derivatives of $U$ and $\eta$ allow to estimate the right hand
	side in Lemma~\ref{derivf00} (\ref{derivf2}) as follows:
	\begin{equation*}
				\begin{split}
					\frac{d^2}{dt^2} \U(\X_t) \geq &
						- \lambda \cdot C_1^2 \cdot \int_{X_0\times X_1} \left( \int_{X_0\times X_1}
						|\d_1(x,z)-\d_0(x,z)| d\ol\m(y)\right)^2  d\ol\m(x)\\
						& - L\cdot C_2 \cdot \int_{X_0} \int_{X_0}
						|\d_1(x,z)-\d_0(x,z)|^2	d\m(y) d\m(x)\\
						\geq & -(\lambda\cdot C_1^2+L\cdot C_2) \cdot \DD(\X_0,\X_1)^2.
				\end{split}
	\end{equation*}
	That is, $\frac{d^2}{dt^2} \U(\X_t) \geq  \kappa \cdot \DD(\X_0,\X_1)^2$ for each geodesic
	$\geod{0}{1}$ in $\YY$. This is the $\kappa$-convexity of $\U$ on the geodesic space $(\YY,\DD)$.
\end{proof}

A straightforward generalization yields analogous assertions for functionals $\bar{\U}:\YY\rightarrow \R_+$ of the form
\[
	\bar{\U}(\X)=\int_0^\infty \U_r(\X)\rho_r dr
\]
for some probability density $\rho$ on $\R_+$ and a one-parameter family of functionals $\U_r$, $r\in \R_+$, of the form \eqref{formulaF} with appropriate $U_r$ and $\eta_r$ (depending in a measurable way on $r\in \R_+$):
\[
	\U_r(\X)=\int_X U_r \left( \int_X \eta_r(\d(x,y))d\m(y)\right) d\m(x).
\]

\begin{corollary}
	\begin{enumerate}
		\item
			If $U_r$ and $\eta_r$ are Lipschitz ($\forall r\geq 0$) then so is $\bar{\U}$ with
			\[
				\Lip(\bar{\U})\leq \int_0^\infty \Lip(U_r) \Lip(\eta_r)\rho_r dr.
			\]
		\item
			If $U_r$ and $\eta_r$ are $C^2$ ($\forall r\geq 0$) then $\bar{\U}$ is $\kappa$-convex for
			\[
				\kappa=-\int_0^\infty \left[ \Vert (U_r'')\Vert_\infty \cdot \Vert \eta_r'\Vert_\infty^2
				+ \Vert (U_r')\Vert_\infty \cdot \Vert \eta_r''\Vert_\infty\right] \rho_r dr
			\]
	\end{enumerate}
\end{corollary}

\subsection{The $\mathcal{G}$-Functionals}
Throughout this section, let
\[\zeta(r)=\left\{\begin{array}{ll}
-2r-1,&r\le-1\\
r^2,&-1\le r\le0\\
0,&0\le r.
\end{array}\right.\]
Given a number $K>0$ and a gauged measure space $(X,\d,\m)$, we say that $\m^3$-a.e. triangle in $(X,\d)$ has  \emph{perimeter $\le 2\pi/\sqrt K$} if
\[\d(x_1,x_2)+\d(x_2,x_3)+\d(x_3,x_1)\le 2\pi/\sqrt K\]
for $\m^3$-a.e. $(x_1,x_2,x_3)\in X^3$. Put
\[\YY_K^{per}=\Big\{\X=\auf X,\d,\m\zu\in \YY\spec
\text{$\m^3$-a.e. triangle in $(X,\d)$ has  perimeter $\le 2\pi/\sqrt K$}\Big\}.\]
For $K\le0$ we put $\YY_K^{per}=\YY$.
\index{y@$\YY_K^{per}$}

\begin{lemma}
For each $K\in\R$, $\YY_K^{per}$ is a closed convex subset of $\YY$.
\end{lemma}

\begin{proof}
\emph{Convexity:} the inequalities $\d_0(x_1,x_2)+\d_0(x_2,x_3)+\d_0(x_3,x_1)\le 2\pi/\sqrt K$ and $\d_1(x_1,x_2)+\d_1(x_2,x_3)+\d_1(x_3,x_1)\le 2\pi/\sqrt K$ carry over from given spaces $(X_0,\d_0,\m_0)$ and $(X_1,\d_1,\m_1)$, resp., to the product space (equipped with any coupling measure) and they are preserved under convex combinations.

\emph{Closedness:} the inequalities $\d_n(x_1,x_2)+\d_n(x_2,x_3)+\d_n(x_3,x_1)\le 2\pi/\sqrt K$ on a sequence of spaces carry over to the limit space. In detail, this stability result is based on the same arguments as the stability of the triangle inequality, see proof of Corollary \ref{closedness of tr-in}.
\end{proof}

\index{g@$\G_K$}
\begin{definition}\label{def G}\begin{enumerate}
\item
The $\G_0$-functional is defined on $\YY$ by
\[\G_0(\X)=\int_{X^4}\zeta\bigg(3\sum_{1\le i\le 3} \d^2(x_0,x_i)-\sum_{1\le i<j\le 3}\d^2(x_i,x_j)\bigg)\,d\m^4(x_0,x_1,x_2,x_3).\]
\item
For any $K\in(0,\infty)$ we define the $\G_K$-functional by
\[\G_K(\X)=\int_{X^4}\zeta\bigg(-\frac1K\bigg[\sum_{1\le i\le 3} \cos\Big(\sqrt{K}\d(x_0,x_i)\Big)\bigg]^2+\frac3K+\frac2K\sum_{1\le i<j\le 3}\cos\Big(\sqrt{K}\d(x_i,x_j)\Big)\bigg)\,d\m^4(x_0,x_1,x_2,x_3)\]
provided $\X\in\YY_K^{per}$ and $\G_K(\X)=\infty$ otherwise.
\item
For any $K\in(-\infty,0)$ we define the $\G_K$-functional by
\begin{eqnarray*}
\G_K(\X)&=&\int_{X^4}\zeta\bigg(-\frac{18}K\log\Big[\frac13\sum_{1\le i\le 3} \cosh\Big(\sqrt{-K}\d(x_0,x_i)\Big)\Big]\\
&&\qquad\qquad+\frac9K\log\Big[\frac13+\frac29\sum_{1\le i<j\le 3}\cosh\Big(\sqrt{-K}\d(x_i,x_j)\Big)\Big]\bigg)\,d\m^4(x_0,x_1,x_2,x_3).\end{eqnarray*}
\end{enumerate}
\end{definition}
Note that $\G_K(\X)\to\G_0(\X)$ for $K\nearrow0$ as well as for $K\searrow0$.
\begin{theorem}\label{theoG}
\begin{enumerate}
\item
For each $K\in\R$ the function $\G_K$ is semiconvex and locally Lipschitz continuous  on $\YY^{per}_K$.
If $K\not=0$ it is globally Lipschitz continuous; if $K=0$ it satisfies $\|\Grad\G_K(\X)\|\le 36\cdot\size(\X)$.

\item
Moreover, $\Grad\G_K$ is given explicitly, e.g. for $K=0$ at the point $\X\in\YY$ as the symmetrization of the function $f\in L^2(X^2,\m^2)$ defined by
\begin{eqnarray*}
f(z,z')&=& 6\d(z,z')\cdot\int_{X^2}\bigg[ 3\zeta'\bigg(3\Big(
\d^2(z,z')+\d^2(z,y)+\d^2(z,y')\Big)-\Big(
\d^2(z',y)+\d^2(z',y')+\d^2(y,y')\Big)\bigg)\\
&&\quad-2\zeta'
\bigg(\Big(
\d^2(y,z)+\d^2(y,z')+\d^2(y,y')\Big)-\Big(
\d^2(y',z)+\d^2(y',z')+\d^2(z,z')\Big)\bigg)\bigg]\,d\m^2(y,y')
.\end{eqnarray*}
\item
For each $K\in\R$ and $\X\in \XX^{geo}$:
\[\G_K(\X)=0\quad\Longleftrightarrow\quad \X\text{ has curvature $\ge K$ in the sense of Alexandrov}.\]
\end{enumerate}
\end{theorem}
Here an isomorphism class $\X$ of mm-spaces is said to have curvature $\ge K$ (or $\le K$) in the sense of Alexandrov if for some (hence any) of its representatives $(X,\d,\m)$ the metric space
$(\supp(\m),\d)$ has curvature $\ge K$ (or $\le K$, resp.) in the sense of Alexandrov.

\begin{proof}
(i), (ii)
Differentiability (weakly up to order two) and semiconvexity  follow from the previous Theorem \ref{thm U}
applied to suitable  functions $u$ on $\R^{6}$. In the case $K=0$, the appropriate choice  is
\[u\Big(\xi_{01},\ldots,\xi_{23}\Big)=\zeta\bigg(3\sum_{1\le i\le 3} \xi_{0i}^2-\sum_{1\le i<j\le 3}\xi_{ij}^2\bigg).\]
Approximating $\zeta$ by
\[\zeta_\epsilon=\Phi_\epsilon(\zeta):=\frac\zeta{1+\epsilon\sqrt\zeta}\]
and analogously $u$ by $u_\epsilon=\Phi_\epsilon(u)$ we obtain Lipschitz continuous, semiconvex functions
$u_\epsilon$   on $\R^{6}$ which approximate $u$ (which itself is locally Lipschitz and semiconvex).
According to Theorem \ref{grad U}, this also yields the formula for the gradient $\Grad\G$.

The formula for $\Grad\G(\X)$ together with the estimate $-2\le\zeta'\le 0$ implies
\[|\Grad\G(\X)(z,z')|\le 36\cdot |\d(z,z')|\]
and thus $\|\Grad\G(\X)\|\le 36\cdot \size(\X)$.

The general case of $K\in\R$ is treated analogously. For instance, in the case $K=-1$ one has to choose
\[u\Big(\xi_{01},\ldots,\xi_{23}\Big)=\zeta\bigg(18\log\Big(\frac13\sum_{1\le i\le 3} \cosh\xi_{0i}\Big)-9\log\Big(\frac13+\frac29\sum_{1\le i,j\le 3}\cosh\xi_{ij}\Big)\bigg).\]
Again it is easily verified that this function is Lipschitz continuous  and semiconvex on $\R^{6}$.

\medskip

(iii) We first discuss the case $K=0$.
Obviously, $\G_0(\X)=0$ is equivalent to
\begin{equation}\label{nnc}
3\sum_{1\le i\le 3} \d^2(x_0,x_i)\ge\sum_{1\le i<j\le 3}\d^2(x_i,x_j)
\end{equation}
for $\m^4$-a.e. quadruple $(x_0,x_1,x_2,x_3)\in X^4$. Since $\d$ is continuous the latter is equivalent to (\ref{nnc}) for \emph{all} quadruples $(x_0,x_1,x_2,x_3)\in X^4$. According to a recent characterization by Lebedeva and Petrunin \cite{lp}, for a geodesic mm-space this in turn is equivalent to nonnegative curvature in the sense of Alexandrov.

Analogously,
in the case $K<0$ the condition  $\G_K(\X)=0$ is obviously equivalent to the condition
\begin{equation}\label{nnc}
\bigg(\sum_{1\le i\le 3} \cosh\Big(\sqrt{-K}\d(x_0,x_i)\Big)\bigg)^2\ge3+2\sum_{1\le i<j\le 3}\cosh\Big(\sqrt{-K}\d(x_i,x_j)\Big)
\end{equation}
for all  quadruples $(x_0,x_1,x_2,x_3)\in X^4$.
In the case $K>0$ it is equivalent to the facts that all triangles in $X$ have perimeter $\le{2\pi}/{\sqrt K}$ and that
\begin{equation}\label{nnc}
\bigg(\sum_{1\le i\le 3} \cos\Big(\sqrt{K}\d(x_0,x_i)\Big)\bigg)^2\le3+2\sum_{1\le i<j\le 3}\cos\Big(\sqrt{K}\d(x_i,x_j)\Big)
\end{equation}
for all  quadruples $(x_0,x_1,x_2,x_3)\in X^4$.

Again in both cases, within geodesic mm-spaces, the latter characterizes the spaces of curvature $\ge K$ in the sense of Alexandrov \cite{lp}.
\end{proof}

\subsection{The  $\mathcal{H}$-Functionals}

\begin{definition}\begin{enumerate}
\item The $\H_0$-functional is defined on $\YY$ by
\begin{eqnarray*}
\H_0(\X)&=&\int_{X^4}\zeta\bigg(\d^2(x_1,x_2)+\d^2(x_2,x_3)+\d^2(x_3,x_4)+\d^2(x_4,x_1)\\
&&\qquad\qquad\qquad\qquad\qquad\qquad\qquad-\d^2(x_1,x_3)-\d^2(x_2,x_4)\bigg)\,d\m^4(x_1,x_2,x_3,x_4)\end{eqnarray*}
with $\zeta$ as before in Definition \ref{def G}.

\index{h@$\H_K$}
\item For  $K\in(0,\infty)$ we define the $\H_K$-functional by
\begin{eqnarray*}\H_K(\X)&=&\int_{X^4}\zeta\bigg(-\frac2K\sum_{i=1}^4 \cos^*\Big(\sqrt{K}\d(x_i,x_{i+1})\Big)\\
&&\qquad\qquad+\frac8K\cos\Big(\frac12\sqrt{K}\d(x_2,x_4)\Big)
\cdot\cos\Big(\frac12\sqrt{K}\d(x_1,x_3)\Big)
\bigg)\,d\m^4(x_1,x_2,x_3,x_4)\end{eqnarray*}
with $x_5:=x_1$
and $\cos^*(r):=\cos(r)$ for $r\in[-\pi/2,\pi/2]$ and $\cos^*(r)=-\infty$ else.
\item For any $K\in(-\infty,0)$ we define the $\H_K$-functional by
\begin{eqnarray*}
\H_K(\X)&=&\int_{X^4}\zeta\bigg(-\frac8K\log\Big[\frac14\sum_{i=1}^4 \cosh\Big(\sqrt{-K}\d(x_i,x_{i+1})\Big)\Big]\\
&&\qquad\quad+\frac8K\log
\Big[
\cosh\Big(\frac12\sqrt{-K}\d(x_2,x_4)\Big)
\cosh\Big(\frac12\sqrt{-K}\d(x_1,x_3)\Big)
\Big]\bigg)\,d\m^4(x_1,x_2,x_3,x_4).\end{eqnarray*}
\end{enumerate}
\end{definition}
Note that $\H_K(\X)\to\H_0(\X)$ for $K\nearrow0$ as well as for $K\searrow0$.

\begin{theorem}
\begin{enumerate}
\item
For each $K\in\R$ the function $\H_K$ is semiconvex and locally  Lipschitz continuous  on $\YY$.
It is globally Lipschitz if $K\not=0$.

\item
Moreover, $\Grad\H_K$ is given explicitly, e.g. for $K=0$ at the point $\X\in\YY$ as the symmetric function $f\in L^2(X^2,\m^2)$ defined by
\begin{eqnarray*}
f(z,z')&=& 4\d(z,z')\cdot\int_{X^2}\bigg[ 2\zeta'\Big(
\d^2(z,z')+\d^2(z',y)+\d^2(y,y')+\d^2(y',z)-
\d^2(z,y)-\d^2(z',y')\Big)\\
&&\quad-\zeta'
\Big(
\d^2(z,y)+\d^2(y,z')+\d^2(z',y')+\d^2(y',z)-\d^2(z,z')-\d^2(y,y')\Big)\bigg]\,d\m^2(y,y')
.\end{eqnarray*}
\item
For each $\X\in \XX^{geo}$ and each $K\in\R$:
\[\X\text{ has globally curvature $\le K$ in the sense of Alexandrov}\quad\Longrightarrow\quad  \H_K(\X)=0.\]
In particular, in the case $K=0$
\[ \X\text{ has globally  curvature $\le 0$ in the sense of Alexandrov}\quad\Longleftrightarrow\quad \H_0(\X)=0.\]
\end{enumerate}
\end{theorem}

\begin{proof} (i), (ii) The proof of (local/global) Lipschitz continuity and semiconvexity is almost identical to the previous one for $\G_K$. Also the formula for $\Grad\H$ is derived in completely the same way.

(iii)
Obviously, $\H_0(\X)=0$ is equivalent to
\begin{equation}\label{npc}
\d^2(x_1,x_2)+\d^2(x_2,x_3)+\d^2(x_3,x_4)+\d^2(x_4,x_1)-\d^2(x_1,x_3)-\d^2(x_2,x_4)\ge0
\end{equation}
for all quadruples $(x_1,x_2,x_3,x_4)\in X^4$.  According to a recent characterization by Berg and Nikolaev \cite{bn}, for a geodesic mm-space this in turn is equivalent to globally nonpositive curvature in the sense of Alexandrov.
The claim for general $K\in\R$ follows from the next lemma.
\end{proof}

\begin{lemma}
Let $(X,\d)$ be a geodesic metric space globally of curvature $\le K$  in the sense of Alexandrov for some $K\in\R\setminus\{0\}$.
Then  if $K<0$
 \begin{eqnarray*}
\lefteqn{ 4\,\cosh\Big(\frac12\sqrt{-K}\d(x_2,x_4)\Big)\cdot \cosh\Big(\frac12\sqrt{-K}\d(x_1,x_3)\Big)}\\
&\le&
\cosh\Big(\sqrt{-K}\d(x_1,x_2)\Big)+\cosh\Big(\sqrt{-K}\d(x_2,x_3)\Big)+
 \cosh\Big(\sqrt{-K}\d(x_3,x_4)\Big)+\cosh\Big(\sqrt{-K}\d(x_4,x_1)\Big)
 \end{eqnarray*}
for every quadruple $(x_1,x_2,x_3,x_4)\in X^4$. Analogously, if $K>0$
 \begin{eqnarray*}
\lefteqn{ 4\,\cos\Big(\frac12\sqrt{K}\d(x_2,x_4)\Big)\cdot \cos\Big(\frac12\sqrt{K}\d(x_1,x_3)\Big)}\\
&\ge&
\cos\Big(\sqrt{K}\d(x_1,x_2)\Big)+\cos\Big(\sqrt{K}\d(x_2,x_3)\Big)+
 \cos\Big(\sqrt{K}\d(x_3,x_4)\Big)+\cos\Big(\sqrt{K}\d(x_4,x_1)\Big)
 \end{eqnarray*}
 for every quadruple $(x_1,x_2,x_3,x_4)\in X^4$ with $\d(x_i,x_{i+1})\le\frac\pi{2\sqrt K}$ for each $i=1,\ldots,4$.
\end{lemma}

\begin{proof} To simplify notation, we first assume $K=1$.
Let a quadruple $(x_1,\ldots,x_4)\in X^4$ be given with $\d(x_i,x_j)\le\frac\pi{2\sqrt K}$ for all $i,j$ and let $z$ be a midpoint of $x_1$ and $x_3$.
Then by global triangle comparison, applied to the triangle $(x_1,x_2,x_3)$
\[\cos\Big(\d(z,x_2)\Big)\cdot\cos\Big(\frac12\d(x_1,x_3)\Big)\ge \frac12 \cos\Big(\d(x_1,x_2)\Big)+\frac12\cos\Big(\d(x_3,x_2)\Big).\]
Considering the triangle $(x_1,x_4,x_3)$ we obtain similarly
\[\cos\Big(\d(z,x_4)\Big)\cdot\cos\Big(\frac12\d(x_1,x_3)\Big)\ge \frac12 \cos\Big(\d(x_1,x_4)\Big)+\frac12\cos\Big(\d(x_3,x_4)\Big).\]
Since $r\mapsto \cos(r)$ is decreasing and concave on the interval $[0,\pi/2]$,
\[ \cos\Big(\frac12\d(x_2,x_4)\Big)\ge
\cos\Big(\frac12\d(x_2,z)+\frac12\d(z,x_4)\Big)\ge
\frac12\cos\Big(\d(x_2,z)\Big)+\frac12\cos\Big(\d(z,x_4)\Big).\]
Altogether this implies
\begin{eqnarray*}
\lefteqn{ \cos\Big(\frac12\d(x_2,x_4)\Big)\cdot \cos\Big(\frac12\d(x_1,x_3)\Big)}\\
&\ge&
 \frac14 \cos\Big(\d(x_1,x_2)\Big)+\frac14\cos\Big(\d(x_3,x_2)\Big)+
 \frac14\cos\Big(\d(x_1,x_4)\Big)+\frac14\cos\Big(\d(x_3,x_4)\Big).
 \end{eqnarray*}
 In the case $K=-1$, the same formulas hold true with all $\cos$ replaced by $\cosh$ and all $\ge$ replaced by $\le$. The general case follows by re-scaling.
\end{proof}

\section{The $\mathcal{F}$-Functional}

\subsection{Balanced spaces}

Given a gauged  measure space $(X,\d,\m)$, we define its \emph{volume growth function}\index{volume growth}  $v: \R_+\times X\rightarrow \R_+$ by
\[
	v_r(x):=\m(B_r(x))
\]
where $B_r(x)=\{y\in X \spec |\d(x,y)|<r\}$.
\begin{definition}
A gauged  measure space $(X,\d,\m)$ is called \emph{balanced}\index{balanced} if there exists a function $v^\star:\R_+\to\R_+$ such that for every $r>0$
\begin{equation*}\label{defbalanced}
	v_r(x)=v^\star_r\quad\text{for $\m$-a.e. }x\in X.
\end{equation*}
\end{definition}
\index{v@$v_r$, $v^\star_r$}

\begin{remarks}
\begin{enumerate}
\item
Being balanced is invariant under homomorphisms of gauged measure spaces (see Proposition 5.6).
\item
A metric  measure space $(X,\d,\m)$ is balanced if and only if for all $r\in\R_+$
\begin{equation*}
	x\mapsto v_r(x) \quad\text{does not depend on $x\in \supp(\m)\subset X$}.
\end{equation*}
\end{enumerate}
\end{remarks}
\begin{proof} (i) as well as the  ``if''-implication in (ii) are obvious. For the converse, note that $v^\star$ has at most countably many discontinuities.
Choose  $r>0$ in which $v^\star$ is continuous.
By the triangle inequality, for all $x$ and all $y\in B_\epsilon(x)$
\[v_{r-\epsilon}(x)\le v_r(y)\le v_{r+\epsilon}(x).\]
Hence,
\[y\mapsto v_r(y)\quad\text{is continuous on }\supp(\m).\]
Thus
\begin{equation}\label{v=v on supp}
v_r(y)=v_r^\star\quad\text{for all }y\in \supp(\m).
\end{equation}
Recall that this holds for all $r>0$ in which $v^\star$ is continuous.
Since $r\mapsto v_r(y)$ is left continuous for each $y\in X$,  (\ref{v=v on supp}) extends to all $r>0$.
\end{proof}

\begin{proposition}
Assume that a gauged measure space $(X,\d,\m)$ is \emph{homogeneous} in the sense that for each pair $(x,y)\in X^2$ there exists a map $\psi:X\to X$ which sends $x$ to $y$
and which preserves measure and gauge. Then $(X,\d,\m)$ is balanced.
\end{proposition}

\begin{proof} The fact that $\psi$ preserves measure and gauge implies $v_r(x)=v_r(\psi(x))$.
\end{proof}

\begin{example}
\begin{description}
\item[\it Discrete Circles.] For $n\in\N$, let
$X=\big\{e^{k2\pi i/n}\subset\mathbb C:\ k=1,\ldots,n\big\}$,
let $\m$  be the uniform distribution on the $n$ points of $X$ and let $\d$ be the graph distance on $X$ (which -- up to a multiplicative constant -- coincides with the induced distance within the unit circle of $\mathbb C$). Then $(X,\d,\m)$ is balanced.

The volume growth $v^\star$ is a step function with values in $\frac{2k-1}n$ for $k=1,\ldots, \lfloor \frac{n+1}2\rfloor$ and (in addition if $n$ is even) 1.
\item[\it Platonic Solids.] Each platonic solid (regarded as a metric measure space with uniform distribution on the vertices and induced graph distance or, alternatively, with distance of ambient Euclidean space) is a balanced space.
\item[\it Discrete Continuum.] The discrete continuum (see \ref{discrete-sphere}) is balanced with volume growth
\[v_r^\star=\left\{\begin{array}{cc}
0,\quad&r\le1\\
1,\quad&r>1.
\end{array}\right.\]
\end{description}
\end{example}

\begin{proposition}
For $\X\in\XX^{length}$ the following are equivalent:
\begin{enumerate}
\item $\X$ is balanced with $v_r^\star=r\wedge1$
\item $\X$ is the circle of length 2 (with uniform distribution).
\end{enumerate}
\end{proposition}

\begin{proof} Without restriction, assume that $\m$ has full support.
A first consequence of the volume growth is the doubling property for $\m$ and thus the compactness of $X$ (cf. \cite{Gro}, \cite{bbi}).
Since $X$ was assumed to be a length space, we conclude that it is a geodesic space.

Let $\gamma:[0,1]\to X$ be a geodesic of length $L=\d(\gamma_0,\gamma_1)\le1$.
Then for each $n\in\N$
\[\m\left(B_{\frac L{2n}}(\gamma)\right)\ge
\m\left(\bigcup_{i=1}^n B_{\frac L{2n}}\left(\gamma_{\frac in}\right)\right)=
\sum_{i^=1}^n \m\left(B_{\frac L{2n}}\left(\gamma_{\frac in}\right)\right)=
n\cdot v^\star_{\frac L{2n}}=\frac L2.\]
Thus $\m(\gamma)\ge\frac L2$.

According to the volume growth, the diameter  is 1. Thus there exists a pair $(x,y)\in X^2$ such that $\d(x,y)=1$.
Let $\gamma$ be a connecting geodesic. Then \[\m(\gamma)\ge\frac12.\]
(Hence, there exist at most two such geodesics which are `disjoint' in the sense that the restrictions to the open interval $(0,1)$ are disjoint.
If there exist two `disjoint' geodesics then we are done: they will support all the mass.)

 Let $z=\gamma_{1/2}$ be the midpoint of $\gamma$. Then
\[\frac12\le \m\left(\gamma\right)\le \m\left(B_{\frac12}(z)\right)=\frac12.\]
Thus within $B_{\frac12}(z)$ all the mass is supported by $\gamma$. There is no branching.
But at $x$ and $y$, the boundary points of $B_{\frac12}(z)$, other geodesics $\alpha,  \beta$ (of length 1) must start. Otherwise, $v_r(x)=r/2$ and $v_r(y)=r/2$ for all  $r\in(0,1)$.
The diameter bound  requires that $\gamma$ composed with these geodesics $\alpha$ and $\beta$ emanating from $x$ and $y$, resp., constitute a closed curve.
This yields the claim.
\end{proof}

\begin{example}
Let $X={\mathcal I}^\infty$ be the infinite dimensional torus, i.e. the infinite product of
$\mathcal I=\R/\Z$, the circle of length 1.
The 1-dimensional Lebesgue measure $\Leb^1$ on ${\mathcal I}$ induces a Borel probability measure
 $\m=\Leb^\infty$ on the Polish space  $X$.
Given a sequence of positive real numbers $(a_n)_{n\in\N}$, we define a metric $\d$ on $X$ by
\[\d(x,y)=2\sup_{n\in\N}\frac{ \d_1(x_n,y_n)}{a_n}\]
where $\d_1$ denotes the standard  metric on ${\mathcal I}$, i.e. $\d_1(s,t)=\inf_{k\in\Z}|s-t+k|$.
\begin{enumerate}
\item
 Then $(X,\d,\m)$ is balanced with
\[v^\star_r=\prod_{n\in\N} (r\,a_n\wedge1).\]
\item
If $a_n=1$ for all $n$ then $m(B_r(x))=0$ for all $x\in X$ and all $r\in [0,1)$. That is, $(X,\d,\m)$ is balanced with $v^\star_r=0$ for $r<1$ (and of course $v^\star_r=1$ for $r\ge1$).
\item
If $a_n=e^n$ then $(X,\d,\m)$ is balanced with
\[v^\star_r=r^{-\frac12\log r+O(1)}\qquad\text{as }r\to0.\]
Indeed, for each $x\in X$ and $r>0$
\[m(B_r(x))= \prod_{a_n< 1/r}(r\cdot a_n)=\exp\left(\sum_{n<-\log r}(\log r +n)\right)
=\exp\left(-(\log r)^2+\frac12(\log r)^2+O(\log r)\right).\]
\end{enumerate}
\end{example}

Now let us have a closer look on Riemannian spaces which are balanced. We will consider the volume growth $(r,x)\mapsto v_r(x)$ for triples $(X,\d,\m)$ where $X=M$ is a  Riemannian manifold (which always is assumed to be smooth, complete and connected)
equipped with its Riemannian distance $\d$ and its Riemannian volume measure $\m$.
To avoid confusing normalization constants, for the rest of this section we will not require that the measure $\m$ is normalized. Even more, we will not require that it is finite (i.e. we will also allow spaces of infinite volume). The manifold $M$ will be called \emph{balanced} if its volume growth function $\m(B_r(x))$ is independent of $x$.

The favorite examples here are the simply connected $n$-dimensional Riemannian manifolds $\M^{n,{K}}$ of constant sectional curvature ${K}\in\R$.
The model spaces $\M^{n,{K}}$ for ${K}>0$ are rescaled versions of the standard $n$-sphere $\s^n=\M^{n,1}$ whereas for ${K}<0$ they are rescaled versions of the hyperbolic space $\mathbb H^n=\M^{n,-1}$. The space form for ${K}=0$ is the Euclidean space $\R^n=\M^{n,0}$.

\begin{example} For each $n\in\N$ and ${K}\in\R$, the space $\M^{n,{K}}$ is balanced with volume growth
\begin{equation}\label{volspaceform}
v_r^\star=\frac{2\pi^{n/2}}{\Gamma(n/2)}\int_0^r\left(\frac{\sinh(\sqrt{-{K}}t)}{\sqrt{-{K}}}\right)^{n-1}dt
\end{equation}
if ${K}<0$; if ${K}>0$, $\frac{\sinh(\sqrt{-{K}}t)}{\sqrt{-{K}}}$ must be replaced
by
$\frac{\sin(\sqrt{{K}}t\wedge \pi)}{\sqrt{{K}}}$ and if ${K}=0$ by $t$.
\end{example}

Besides model spaces, there exist many other Riemannian examples of balanced spaces.

\begin{example}
\begin{itemize}
			\item
		\emph{Product of spheres}, e.g.\ $M=\s^2\times \s^2$:\\
		Here $v_r(x)=v^\star_r=(2\pi)^2\cdot\big(1-\cos (r\wedge \pi)\big)^2$ for all $x\in M$ and  $r>0$.
\item \emph{Torus} $M=\R^n/\Z^n={\mathcal I}^n$ with ${\cal I}=\R/\Z$ circle of length $1$:\\
		Here $v_r(x)=v^\star_r$ for all $(r,x)\in\R_+\times M$ for some function $v^\star:\R_+\rightarrow \R_+$
		satisfying
		\[
			v^\star_r= \begin{cases}
					c_n r^n & \text{for }0\leq r\leq \frac{1}{2}\\
					1 &			\text{for }r\geq 1.
				\end{cases}
		\]
\end{itemize}
\end{example}

\begin{lemma}[Gray, Vanhecke \cite{gv}] For any
$n$-dimensional Riemannian manifold $(M,g)$ -- equipped with its Riemannian distance $\d$ and its (non-normalized) Riemannian volume measure $\m$ -- the volume growth function admits  the following asymptotic expansion
\begin{equation}
v_r(x)=c_nr^n\cdot\Big(1+b_2(x)r^2+b_4(x)r^4+b_6(x)r^6+{\mathcal O}(r^8)\Big)
\end{equation}
as $r\searrow0$ locally uniformly in  $x\in X$ with $c_n=\frac{\pi^{n/2}}{\Gamma(n/2+1)}$ and explicitly given coefficients $b_2, b_4, b_6$.
In particular,
\begin{itemize}
\item $b_2(x)=-\frac{\scal(x)}{6(n+2)}$ where $\scal(x)$ denotes the scalar curvature at $x\in M$
\item $b_4(x)=\frac1{360 (n+2)(n+4)}
\Big( -3 \|\Riem\|^2(x)+8 \|\Ric\|^2(x)+5\scal^2(x)-18 \Delta\scal(x)\Big)$
\end{itemize}
with $\Riem$ denoting the Riemannian curvature tensor and $\Ric$ the Ricci tensor.
\end{lemma}

In dimension $n=2$, the coefficient $b_4$ is explicitly given as
\[b_4(x)=\frac1{1440}\Big(\scal^2(x)-3\Delta\,\scal (x)\Big).\]
In $n=3$, it is given as
\[b_4(x)=\frac1{6300}\Big(4\scal^2(x)-2\|\Ric\|^2(x)-9\Delta\,\scal (x)\Big).\]

In dimensions $n\ge 3$, the coefficient $b_4$ can also be expressed as
\begin{equation}
b_4(x)=\frac1{360 (n+2)(n+4)}
\Big( -3 \|\sf W\|^2(x)+C'_n \,\|\mathring\Ric\|^2(x)+C''_n\,\scal^2(x)-18 \Delta\scal(x)\Big)
\end{equation}
with $C'_n=8-\frac3{(n-2)^2}$ and $C''_n=5-\frac3{[2n(n-1)]^2}+\frac8{n^2}$
in terms of the \emph{traceless Ricci tensor}
\begin{equation}
\mathring\Ric=\Ric-\frac\scal n g
\end{equation}
and the \emph{Weyl tensor}
\begin{equation}
\sf W=\Riem -\frac1{n-2} \mathring{\Ric} \circ g-\frac{\scal}{2n(n-1)}g\circ g.
\end{equation}
Indeed (see \cite{Petersen}),
\[\|\Ric\|^2=\|\mathring\Ric\|^2+\frac1{n^2}\scal^2\]
and
\[\|\Riem\|^2=\|\sf W\|^2+\frac1{(n-2)^2}\|\mathring\Ric\|^2+\frac1{[2n(n-1)]^2}\scal^2.\]

\begin{itemize}\item
If $M$ is  conformally flat then the Weyl tensor vanishes \cite{Petersen}.
\item
Conversely, if the Weyl tensor vanishes and $n\ge4$ then $M$ is conformally flat.
\item
If the traceless Ricci tensor vanishes and $n\ge 3$ then $M$ is Einstein (i.e. $\Ric=\lambda g$ for some $\lambda\in\R$).
\end{itemize}

\begin{corollary}
Every balanced Riemannian manifold has constant scalar curvature.
\end{corollary}

\begin{proof}
If $v_r(x)$ is independent of $x$ then so are the coefficients $b_k(x)$ in the above asymptotic expansion. For $k=2$ this is the claim.
\end{proof}

The converse implication is not true. Even worse: \emph{constant sectional curvature does not imply that $M$ is balanced.}

\begin{example}
Consider the Riemannian manifold
		\[
			M=\mathbb{H}/G
		\]
		obtained as the quotient space of $\mathbb{H}$ under the action of a discrete subgroup $G$
		of isometries of $\mathbb{H}$, acting freely on it. Then $M$ has constant curvature $-1$.\\
		Hence, for each $x\in M$ there exists $R>0$ such that
		\[
			v_r(x)= v^\star_r \qquad \text{for all } r\in [0,R].
		\]
On the other hand, if $M$ is non-compact
		for each $r>0$
		\[
			v_r(x) \rightarrow 0 \quad \text{as }x\rightarrow \infty.
		\]
Note that there also exist such examples $M=\mathbb{H}/G$ which are non-compact but have finite volume,
		see e.g.\ Example 5.7.4 in \cite{Dav}.
\end{example}

\begin{conjecture}

Let $v^\star$ be the volume growth of a given model space $\M^{n,{K}}$ and let $v$ denote the volume growth of another, arbitrary Riemannian manifold $M$.
\begin{enumerate}
\item[\bf (I)] Gray, Vanhecke (1979):
\[\forall x\text{ as }r\to0:\ v_r(x)=v^\star_r+{o}(r^{n+4})\qquad\Longleftrightarrow\qquad M\text{ has sectional curvature ${K}$ and dimension $n$}\]
\item[\bf (II)] Moreover:
\[\forall x, \forall r>0: \ v_r(x)=v^\star_r\qquad\Longleftrightarrow\qquad M=\M^{n,{K}}.
\phantom{has sectional curvature}
\]
\end{enumerate}
\end{conjecture}

\begin{theorem}
The Conjectures (I) and (II) are true in each of the following cases
\begin{enumerate}
\item $n\le 3$
\item $M$ is conformally flat
\item $M$ is an Einstein manifold
\item $M$ satisfies the uniform lower bound $\Ric_x\ge (n-1){K}$ 
\item $M$ satisfies the uniform upper bound $\Ric_x\le (n-1){K}$.
\end{enumerate}
\end{theorem}

\begin{proof} Conjecture (I) has been proven by Gray and Vanhecke in \cite{gv}.
Being unaware of this result, an independent proof of it as well as a proof of Conjecture (II) has been proposed to the author by Andrea Mondino (personal communication, May 2012).
For the convenience of the reader, we sketch the arguments for both conjectures.

According to the asymptotic formula for the volume growth (up to order 2), the assumption on the local coincidence of the volume growth of $M$ and $\M^{{K},n}$ implies
\begin{itemize}
\item $\dim_M=n$
\item $\scal(x)= \scal^\star=n (n-1)\,{K}$ for all $x$.
\end{itemize}
Taking into account the $4^{th}$-order term of the volume growth, it yields
\[-3 \|\Riem\|^2(x)+8 \|\Ric\|^2(x)=
-3 \|\Riem^\star\|^2+8 \|\Ric^\star\|^2\] or equivalently
\begin{equation}\label{W,R}
-3 \|\sf W\|^2(x)+\left(8-\frac3{(n-2)^2}\right) \|\mathring\Ric\|^2(x)=
-3 \|\sf W^\star\|^2+\left(8-\frac3{(n-2)^2}\right) \|\mathring\Ric^\star\|^2.
\end{equation}
Now assume (iii), (iv) or (v). Since $\scal(x)= n (n-1)\,{K}$, each of these assumptions implies that
$\mathring\Ric=0$. Anyway, $\sf W^\star=0$ and $\mathring\Ric^\star=0$. Hence, according to (\ref{W,R}), $\sf W=0$ and thus $R=R^\star=\frac{\scal^\star}{2n(n-1)}g\circ g$.

Next assume (ii), i.e. $M$ is  conformally flat. Then $\sf W=0$. Since $\sf W^\star=0$ and $\mathring\Ric^\star=0$, it implies $\mathring\Ric=0$ and thus  $R=R^\star=\frac{\scal^\star}{2n(n-1)}g\circ g$.

The case (i) follows from the explicit formulas for the coefficient $b_4$ in dimensions 2 and 3.

\medskip

To prove the validity of Conjecture (II) in all these cases, finally, assume that $M$ has constant sectional curvature ${K}$ and dimension $n$. Then by the Bishop-Gromov volume comparison theorem
 \[v_r(x)\le v^\star_r\]
 for all $r$ and $x$. Moreover, equality (for all $r$ and $x$) holds true if and only if $M$ is the model space $\M^{n,{K}}$.
\end{proof}
\begin{remark}
Within the larger frame of Finsler manifolds $M$, Conjectures (I) and (II) are wrong. In fact, every $n$-dimensional normed space equipped with a multiple of the $n$-dimensional Lebesgue measure is balanced -- and after appropriate choice of the normalizing constant -- has the same volume growth as the Euclidean space $\R^n$.
\end{remark}

\subsection{The $\F$-Functional and its Gradient Flow}

Now let us fix a balanced space $\X^\star\in\ol\XX$ (with volume growth $v^\star$) as well as  a Borel function $\rho:\R_+\rightarrow \R_+$ with $\rho_r>0$ for all $r$ and
$\int_0^\infty (r^2+r^4)\rho_r dr <\infty$.
We regard $\X^\star$ as a  ``model space'' within the category of pseudo metric measure spaces. The downward gradient flow for the $\F$-functional to be defined below -- either on $\ol\XX$ or on $\YY$ -- will push any other space $\X$ towards $\X^\star$.

Define $\F:\YY\rightarrow\R_+$ by
\[
	\F(\X)= \frac12\int_0^\infty \int_X \left[\int_0^r\Big(v_t(x)-v^\star_t\Big)\,dt\right]^2 d\m(x) \rho_r dr
\]
where
$v_r(x)=\m(B_r(x))$ for $\X=\auf X,\d,\m\zu$. Recall that  $B_r(x)=\{y\in X:\ |\d(x,y)|<r\}$.

\begin{theorem}\label{theoF}
	\begin{enumerate}
		\item\label{globminbal}
			Each global minimizer $\X$ of $\F$ is balanced with
			\[
				\m(B_r(x))=v^\star_r \qquad \text{for all }r\in [0,\infty) \text{ and  $\m$-a.e. }x\in X.
			\]
		\item\label{flipsemiconv}
			The function $\F:\YY\rightarrow \R_+$ is Lipschitz continuous and semiconvex.
			More precisely, it is $\kappa$-convex with
			$
				\kappa= -\sup_{r>0}\big[r\,\rho_r\big]
			$
			and Lipschitz continuous with
			$
				\Lip(\F) \leq  \int_0^\infty r\, \rho_r dr.
			$
\item The ambient gradient of $-\F$ at a point $\X=\auf X,\d,\m\zu\in\ol\XX$ is given by $\Grad(-\F)(\X)=\f\in L^2_s(X^2,\m^2)$ with
	\[
		\f(x,y) =\int_0^\infty \left( \frac{v_r(x)+v_r(y)}{2}-v^\star_r\right) \bar{\rho}\big(r\vee\d(x,y)\big)dr
	\]
where $\bar{\rho}(a)= \int_a^\infty  \rho_r dr$.
		\end{enumerate}
\end{theorem}

\begin{proof}
	(\ref{globminbal}) Since we assumed that $v^\star_r$ is the volume growth of
	$\X^*$, the function $\F$ will attain its global minimum $0$ at least at the point $\X^*$. For any
	other minimizer $\X$, it immediately follows that
	\[
		w_r(x)=w^\star_r
	\]
	for $\m$-a.e.\ $x\in X$ and a.e.\ $r\geq 0$
where
\[w^\star_r:= \int_0^r  v^\star_t dt ,\qquad
w_r(x)=\int_0^r  v_t(x) dt.\]
Indeed, this actually holds for each $r>0$ since for every $x\in X$ the function  $r\mapsto w_r(x)$ is continuous. (It is obtained as the anti-derivative of a function $r\mapsto v_r(x)$ which itself is non-decreasing and left continuous.)
With the same argument, $r\mapsto w^\star_r$ is seen to be continuous.
\\
	(\ref{flipsemiconv}) Note that $\F$ can be written as
	$
		\F(\X)= \int_0^\infty \F_r(\X) \rho_r dr
	$
	with $\F_r(\X)=\int_X U_r \left( \int_X \eta_r(\d(x,y))d\m(y) \right) d\m(x)$ as in \eqref{formulaF}
	if one chooses
	\[
		U_r(a)=\frac12(a-w^\star_r)^2, \quad \eta_r(a)= \left(r-|a|\right)^+.
	\]
	For each geodesic $\geod{0}{1}$ emanating from $\X$
	\[
		\frac{d}{dt}\F(\X_t)\Big|_{t=0} =  \int_0^\infty \int_X \int_X \Big(w_r(x)-w^\star_r\Big)\cdot
\Big(1_{(-r,0]}(\d(x,y))  -1_{[0,r)}(\d(x,y))\Big)\cdot
								\Big[\d_1(x,y)-\d(x,y)\Big]\,
								 d\m(y)\, d\m(x)\, \rho_r\, dr
	\]
	and thus
	\begin{align*}
		|\nabla\F(\X)| &\leq  \left( \int_X \int_X \left[ \int_{|\d(x,y)|}^\infty |w_r(x)-w^\star_r|\,
								\rho_r dr\right]^2  d\m(y) d\m(x) \right)^\frac{1}{2}\\
				&\leq  \int_0^\infty r\, \rho_r dr.
	\end{align*}
	For the last inequality, note that $|w_r(x)-w_r^\star|\leq r$  (since
	$0\le v_r(x)\leq 1$ and $0\le v^\star_r\leq 1$) for all gauged measure spaces.\\
	A similar calculation yields
	\begin{align*}
					\frac{d^2}{dt^2} \F(\X_t)\Big|_{t=0}&=
						 \int_0^\infty \int_{X} \left[\int_{X}
\Big(1_{(-r,0]}(\d(x,y))  -1_{[0,r)}(\d(x,y))\Big)
						\cdot\Big(\d_1(x,y)-\d(x,y)\Big) d\m(y)\right]^2 d\m(x) \rho_r dr\\
						& \phantom{ ={}} + \int_0^\infty \int_X \int_X \big(w_r(x)-w^\star_r\big)\cdot 	\big(\d_1(x,y)-\d(x,y)\big)^2\,
						 d\m(y)d\m(x) \rho_r \, d\big(\delta_{\d(x,y)}+\delta_{-\d(x,y)}-2\delta_0\big)(r)\\
						&\geq -\sup_{r>0}\big[r\,\rho_r\big] \cdot \int_X \int_X
						\big(\d_1(x,y)-\d(x,y)\big)^2	d\m(y)d\m(x)\\
						&=\kappa\cdot \DD(\X_1,\X)^2
	\end{align*}
provided $\kappa$ is chosen as in the claim.\\
(iii) According to
Corollary 7.8
\[
\Grad(-\F_r)(\X)(x,y)=	-\frac12\Big[
U'(w_r(x))+U'(w_r(y))\Big]\, \eta'_r(\d(x,y)).\]
Since $\X\in\ol\XX$ we may assume that $\d(x,y)\ge0$. Integrating  w.r.t. $\rho_r\,dr$ yields
\begin{align*}
\Grad(-\F)(\X)(x,y)&=-
\int_0^\infty	\frac12\Big[
U'(w_r(x))+U'(w_r(y))\Big]\, \eta'_r(\d(x,y)) \rho_r dr\\
&=\int_0^\infty \left( \frac{w_r(x)+w_r(y)}{2}-w^\star_r\right)\cdot
1_{[0,r)}(\d(x,y))
\, \rho_r dr\\
	&=\int_0^\infty \int_0^\infty \left( \frac{v_t(x)+v_t(y)}{2}-v^\star_t\right)
		\cdot 1_{\{t\le r\}}\cdot 1_{\{\d(x,y)<r\}}dt\,\rho_r\, dr\\
			& =\int_0^\infty \left( \frac{v_t(x)+v_t(y)}{2}-v^\star_t\right) \bar{\rho}(t\vee\d(x,y))dt
	\end{align*}
	with $\bar{\rho}(a)= \int_a^\infty  \rho_r dr$.
\end{proof}

\begin{corollary}
\begin{enumerate}
\item
	For each $\X_0\in \ol\XX$ the gradient flow equation
	\begin{equation}\label{gradflow}
		\dot{\X}_t= \nabla(-\F)(\X_t)
	\end{equation}
	has a unique solution $\X_\bullet: [0,\infty)\rightarrow \ol\XX$ starting in $\X_0$. For all $\X_0,\X_0'\in \ol\XX$ and all $t>0$
\begin{equation}\label{exp-cont}
\DD(\X_t,\X'_t)\le e^{|\kappa|\, t}\cdot\DD(\X_0,\X'_0)
\end{equation}
with $\kappa$ from assertion (ii) of the above Theorem.
\item
Similarly, for each $\X_0\in \YY$, there exists a unique solution to  the gradient flow equation (\ref{gradflow}) in $\YY$. It also satisfies the Lipschitz estimate
(\ref{exp-cont}).
\end{enumerate}
\end{corollary}

\begin{remark}
\begin{enumerate}
\item
	The  concept of ambient gradients (see Section 6.5) allows a quite intuitive interpretation of
	the evolution driven by \eqref{gradflow}. According to this calculus, $\Grad(-\F)(\X)$ is the
	function $\f\in L^2(X^2,\m^2)$ given by
	\begin{equation}\label{amb-rho}
		\f(x,y) =\int_0^\infty \left( \frac{v_r(x)+v_r(y)}{2}-v^\star_r\right) \bar{\rho}(r\vee\d(x,y))dr.
	\end{equation}
This fact
	should be interpreted as follows:\\
	the function $\f$ is positive for those pairs of points $(x,y)\in X^2$ for which - in average
	w.r.t.\ the distribution $ \bar{\rho}(r\vee\d(x,y))\,dr$ of the radius - the volume of the balls
	$B_r(x)$ and $B_r(y)$ is too large compared with the volume $v^\star_r$ of balls in the model space;
	and vice versa, if the volume of $B_r(x)$ and $B_r(y)$ is too small (in average w.r.t.\ $r$) then
	$\f(x,y)<0$.\\
	The infinitesimal evolution of $\X$ under the gradient flow is given by
	\[
		\d_t(x,y)=\d(x,y)+ t\f(x,y)+O(t^2)
	\]
	with $\f$ as above.
	That is, $\d(x,y)$ will be enlarged if the volume of balls centered at $x$ and $y$ is too large, and
	$\d(x,y)$ will be reduced if the volume of balls is too small.
\item The gradient flow for $-\F$ gets stuck if it enters the set of critical points. Obviously, $\X$ is critical for $-\F$ if and only if
    \[\Grad(-\F)(\X)=0.\]
    In view of (\ref{amb-rho}) this yields:\
    {\it
    $\X$ is critical if and only if for $\m^2$-a.e. $(x,y)\in X^2$
    \[\frac{v_r(x)+v_r(y)}{2}=v^\star_r\]
    in average w.r.t. the measure $\bar{\rho}(r\vee\d(x,y))dr$.}
\item
The above identification of the ambient gradient leads to an even more intuitive formula if we dispense with smoothing the volume growth, i.e. if in the definition of $\U$ we replace the functions $w_r$ and $w^\star_r$ by the original $v_r$ and $v^\star_r$, resp.
Let
\[\tilde\F(\X)=
\frac12 \int_0^\infty \int_X (v_r(x)-v^\star_r)^2 d\m(x) \rho_r dr
\]
for a Borel  function  $\rho\ge0$ on $\R_+$ as above.
Then a direct calculation as above yields
\[\Grad(-\tilde\F)(\X)(x,y)=\Big[\frac{v_{\d(x,y)}(x)+ v_{\d(x,y)}(y)}2- v^\star_{\d(x,y)}\Big]\cdot\rho_{\d(x,y)}.\]
\end{enumerate}
\end{remark}

\begin{remark}
For each $n\in\N$, the $\F$-functional induces a functional
\[\F^{(n)}=\F\circ\Phi:\, \M^{(n)}\to\R_+\]
on the space $\M^{(n)}$ of symmetric $(n\times n)$-matrices $(\d_{ij})_{1\le i<j \le n}$ with vanishing diagonal entries via the injection $\Phi: \M^{(n)}\to \YY$, see section 5.4.
This functional $\F^{(n)}$ again is Lipschitz continuous and $\kappa$-convex (with the same bounds as $\F$). It admits a unique downward gradient flow in $\M^{(n)}$. This flow $(\d(t))_{t\ge0}$ can be characterized in a very explicit way as follows:
\begin{itemize}
\item As long as $\d_t$ does not reach points $\d\in \M^{(n)}$ with non-trivial symmetries, the flow is simply given by the first order ODE in $\R^{\frac{n(n-1)}2}$
    \[\frac d{dt}\d(t)=-\nabla\F^{(n)}\big(\d(t)\big)\]
    with
\[\nabla_{ij}\F^{(n)}\big(\d\big)=\int_0^\infty\left(\frac{v_r(i)+v_r(j)}2-v^\star_r\right)\ol\rho(r\vee\d_{ij})\,dr
\qquad\text{for }1\le i <j\le1\]
and
$v_r(i)=\frac1n\sharp\big\{k=1,\ldots,n:\ \d_{ik}<r\big\}$.
\item
If $\d$ admits symmetries, say $\sigma_1^*\d=\d$, \ldots, $\sigma_l^*\d=\d$ for $\sigma_1,\ldots,\sigma_l\in S_n$, then smoothness of $\F^{(n)}$ on $\R^{\frac{n(n-1)}2}$, invariance under actions of $S_n$, and uniqueness of $\nabla\F^{(n)}$ imply that the evolution remains within the subspace $\M^{(n)}_{\sigma_1,\ldots,\sigma_l}$ of elements in $\M^{(n)}$ which are invariant under all these permutations $\sigma_1,\ldots,\sigma_l$.
Within this linear subspace, the downward gradient flow for $\F^{(n)}$ again solves a first order ODE until it reaches a point with additional symmetries.
\end{itemize}
\end{remark}

The functional $\F$ is closely related to the famous Einstein-Hilbert functional of Riemannian geometry. To explore this link, for given $n\in\N$ let us consider a one-parameter family $\F^{(\varepsilon)}$, $\varepsilon>0$, of such functionals
\[
	\F^\be(\X)=\frac12\int_0^\infty \int_X (w_r(x)-w^\star_r)^2 d\m(x) \rho_r^\be dr
\]
defined in terms of weight functions $\rho_\bullet^\be:\R_+\rightarrow \R_+$ satisfying as before
\begin{align*}
	\sup_{r>0}[r\,\rho^\be_r]<\infty,\qquad\int_0^\infty r\rho_r^\be dr & <\infty  &(\forall\varepsilon>0)
\intertext{
and now in addition with $c_n'= \left[\frac{c_n}{6(n+2)(n+3)}\right]^2$ (where $c_n=\frac{\pi^{n/2}}{\Gamma(n/2+1)}$, see Lemma 8.9)
}
c_n'\cdot\int_0^\infty r^{2n+6}\cdot \rho_r^\be dr & \rightarrow 1 &\text{as }\varepsilon\rightarrow 0
\intertext{
and
}
	\int_0^\infty r^{2n+8}\cdot\rho_r^\be dr& \rightarrow 0 &\text{as }\varepsilon\rightarrow 0.
\end{align*}

\begin{theorem}\label{thm2.15}
	Let $n\in\N$ be given and let $r\mapsto v^\star_r$ be the volume growth of some balanced Riemannian
	manifold of dimension $n$ and volume 1. Let $\scal^\star$ be its scalar curvature.
	Then for each compact Riemannian manifold of dimension $n$ and volume $1$, regarded as a metric
	measure space $(X,\d,\m)$
	\[
		\lim_{\varepsilon\searrow 0} \F^\be(\X)= \frac12 \int_X (\scal(x)-\scal^\star)^2 d\m(x)
	\]
	where $\scal(x)$ denotes the scalar curvature at $x\in X$.
\end{theorem}

\begin{proof}
	The asymptotic expansion
	\[
		v_r(x)= c_n\cdot r^n \left(1- \frac{\scal(x)}{6(n+2)}r^2+ O(r^4)\right).
	\]
of the volume growth  implies
	\[
		w_r(x)=\frac{c_n}{n+1}\cdot r^{n+1} \left(1- \frac{\scal(x)(n+1)}{6(n+2)(n+3)}r^2+O(r^4)\right)
	\]
	and thus 
	\begin{align*}
		\F_r(\X) &:= \frac12\int_X (w_r(x)-w^\star_r)^2 d\m(x)\\	
			&\phantom{:}=\frac12 c_n' \cdot r^{2n+6} \int_X \left(\scal(x)-\scal^\star+O(r^2)\right)^2 d\m(x)\\
			&\phantom{:}= \frac12c_n' \cdot r^{2n+6} \left[ \int_X (\scal(x)-\scal^\star)^2 d\m(x) + O(r^2) \right].
	\end{align*}
	Integrating w.r.t.\ $\rho_r^\be dr$ therefore yields
	\begin{align*}
		\F^\be (\X)&= \int_0^\infty \F_r(\X) \rho_r^\be dr \\
			&= \frac12\int_X (\scal(x)-\scal^\star)^2 d\m(x) \cdot c_n' \cdot \int_0^\infty r^{2n+6} \rho_r^\be dr
			+ c_n' \cdot \int_0^\infty O(r^2) r^{2n+6} \rho_r^\be dr.
	\end{align*}
		This proves the claim.
\end{proof}

\begin{remark} In Riemannian geometry, the canonical interpretation (and construction) of gradient flows for
	\[
		\F(X)=\frac12\int_X (\scal(x)-\scal^\star)^2 d\m(x)
	\]
is to regard it as a functional on  the space ${\mathfrak Met}(X)$ of metric tensors on a given Riemannian manifold $X$, cf.
\cite{ChowKnopf}.
The downward gradient flow then  is
	characterized as the evolution of  metric tensors determined by
	\[
		\frac{d}{dt}g(x)=  (\scal(x)-\scal^\star) \cdot \Ric_g(x).
	\]
This evolution is different from the evolution governed by the downward gradient flow  induced by the $L^2$-distortion distance on the space of pseudo metric measure spaces and also different from the induced flow within the space of Riemannian manifolds.
\end{remark}

Finally, we will study combinations of the $\F$- and the $\G$-functionals. Let $n\in\N$ be given and choose ${K}>0$ such that the model space $\M^{n,{K}}$ has volume 1. This amounts to ${K}=[(n+1)c_{n+1}]^{2/n}$.
Put $X^\star=\M^{n,{K}}$,
\[v_r^\star=\int_0^{{\sqrt{K}} r\wedge\pi}\sin^{n-1}(t)\,dt\Big/\int_0^{\pi}\sin^{n-1}(t)\,dt,\]
choose any strictly positive weight function $\rho:\R_+\to\R_+$  and define
the $\F$-functional on $\YY$ as before in terms of these quantities by
\[
	\F(\X)=\frac12 \int_0^\infty \int_X \left[\int_0^r  \left(v_t(x)-v^\star_t\right) dt  \right]^2 d\m(x) \rho_r dr.
\]
Moreover, let $\G_{K}$ as introduced in Definition 7.12  and put
\[\U=\F+\G_{K}:\YY\to\R_+.\]

\begin{theorem}\begin{enumerate}
\item
The functional $\U$ is Lipschitz continuous and semiconvex. It admits a unique downward gradient flow in $\YY$ as well as in $\ol\XX$.
\item
For all $\X\in\XX^{geo}$
\[\U(\X)=0\quad\Longleftrightarrow\quad \X=\M^{n,{K}}.\]
\end{enumerate}
\end{theorem}
\begin{proof}
(i) follows from Theorems \ref{theoG} and  \ref{theoF}.

(ii) Let $X$ be a representative of $\X$ with full support. According to Theorem \ref{theoG}, $\U(\X)=0$ implies that $X$ has curvature $\ge K$ in the sense of Alexandrov,
and according to Theorem \ref{theoF} it implies that the volume growth of $X$ is given by $v^\star$.
Thus in particular, $X$ has Hausdorff dimension $n$. The lower curvature bound implies a Bishop-Gromov volume comparison estimate with equality if and only if $X$ coincides with the model space $\M^{n,{K}}$, \cite{bbi}, Thm. 10.6.8 and Exercise 10.6.12.
\end{proof}

\section{Addendum: The $L^{p,q}$-Distortion Distance}

As an addendum to the previous exposition we present a generalization of the fundamental distance used so far on the space of (isomorphism classes of) mm-spaces. The $L^{p}$-distortion distance turns out to be a particular case of the $L^{p,q}$-distortion distance for the choice $q=1$.
The metric and geometric properties of this more general distortion distance will be essentially the same as those for $q=1$. Some properties will become slightly less intuitive (for instance, the embedding via $\d\mapsto \d^q$ into the space $\YY$). 
Instead of overloading notations and proofs with additional technicalities we tried to keep the presentation as simple as possible by restricting to the most simple case $q=1$. The modifications for general $q\ge1$ will be summarized in the two subsequent sections.

Whereas the case $q=1$ is the most simple one and also the most natural one from the point of view of transportation theory and image analysis,  the case $q=2$ is the most relevant one from the point of view of metric geometry and Riemannian calculus.
Indeed, here geodesic interpolations are obtained by interpolations of \emph{squared distances}
\begin{itemize}
\item
which should be regarded as the metric counterpart to \emph{linear} interpolations of metric tensors in Riemannian geometry
\item
and which preserves  nonnegative pre-curvature as well as  nonpositive pre-curvature.
\end{itemize}

\subsection{Metric Properties}

\begin{definition}
	For any $p,q\in[1,\infty)$, the \emph{$L^{p,q}$-distortion distance}\index{distortion distance $\DD_p$} between two  metric measure spaces $(X_0,\d_0,\m_0)$ and $(X_1,\d_1,\m_1)$ is
	defined as
	\begin{eqnarray*}
	\lefteqn{\DD_{p,q}((X_0,\d_0,\m_0),(X_1,\d_1,\m_1))}\\
&=& \inf \Bigg\{  \bigg( \int_{X_0\times X_1} \int_{X_0\times X_1}
		\left| \d^q_0(x_0,y_0)-\d^q_1(x_1,y_1) \right|^p d\ol{\m}
		(x_0,x_1)d\ol{\m}(y_0,y_1)\bigg)^{1/p} \spec \ol{\m}\in \Cpl(\m_0,\m_1)
		\Bigg\}.
	\end{eqnarray*}
Analogously, the $L^{\infty,q}$-distortion distance $\DD_{\infty,q}((X_0,\d_0,\m_0),(X_1,\d_1,\m_1))$ can be defined.
 However, this will be of minor interest and the case $p=\infty$ will be excluded  from the subsequent discussions.
\end{definition}

Similarly as in Lemma \ref{optcoupl1} one verfies that for each $p,q\in [1,\infty)$ and for each pair of metric measure spaces $(X_0,\d_0,\m_0)$ and $(X_1,\d_1,\m_1)$, the infimum in the
			definition of $\DD_{p,q}$ will be attained. That is, there exists a measure $\ol\m\in\Cpl(\m_0,\m_1)$
-- again called \emph{optimal coupling} -- such
			that
\begin{equation}\label{optimal}
				\DD_{p,q}((X_0,\d_0,\m_0),(X_1,\d_1,\m_1))=  \bigg( \int_{X_0\times X_1} \int_{X_0\times X_1}
				\left| \d_0^q(x_0,y_0)-\d_1^q(x_1,y_1) \right|^p  d\ol{\m}(x_0,x_1) d\ol{\m}(y_0,y_1)
				\bigg)^{1/p}.
		\end{equation}

The existence of optimal couplings was the key argument in the proof of Lemma 1.10. Thus we may conclude in the same manner as before:

\begin{lemma}
For any $p,q\in[1,\infty)$ and any metric measure spaces $(X_0,\d_0,\m_0)$ and $(X_1,\d_1,\m_1)$ the assertion
\[\DD_{p,q}((X_0,\d_0,\m_0),(X_1,\d_1,\m_1))=0\]
is equivalent to any of the assertions (ii), (iii) and (iv) of Lemma 1.10. 
In particular,  equivalence classes with respect to
$\DD_{p,q}$
neither depend on $p$ nor on $q$. 
\end{lemma}

The $L^{p,q}$-distortion distance obviously satisfies the triangle inequality (same proof as for Lemma 1.9) and it is finite between any pair of mm-spaces which have finite $L^{pq}$-size. According to the previous lemma, it vanishes only for pairs of mm-spaces which are isomorphic.

\begin{corollary}\label{d-is-metric-q}
For all $p,q\in[1,\infty)$,
$\DD_{p,q}$ is a   metric on $\XX_{pq}$.
\end{corollary}

As in the case $q=1$, we will see that convergence w.r.t.  $\DD_{p,q}$ is {\lq essentially independent\rq}  of $p$.
For this purpose, let us introduce the 
 $L^{0,q}$-distortion distance 
\[
	\DD_{0,q}(\X_0,\X_1)= \inf \bigg\{ \epsilon >0\spec
{\ol\m}\otimes{\ol\m} \bigg( \Big\{(x_0,x_1,y_0,y_1)\spec |{\d^q_0}(x_0,y_0)-\d^q_1(x_1,y_1)|>\epsilon\Big\}
\bigg)
\le\epsilon, \ \ol\m\in\Cpl(\m_0,\m_1)
\bigg\}.
\]

\begin{lemma}
For all $p,q\in[1,\infty)$, every point $\X_\infty$ and every sequence $(\X_n)_{n\in\N}$ in $\XX_{pq}$ the following statements are equivalent:
\begin{enumerate}
\item
$\DD_{p,q}(\X_n,\X_\infty)\to0$ as $n\to\infty$;
\item
$\DD_{0,q}(\X_n,\X_\infty)\to0$  and
$\size_{pq}(\X_n)\to \size_{pq}(\X_\infty)$ as $n\to\infty$.
\end{enumerate}
\end{lemma}
The {\it proof} is exactly the same as that of Proposition 2.1.
There also exist quantitative estimates between $L^{p,q}$-distortion distances for varying parameters $p$ and $q$.

\begin{lemma}\label{lp-d-dd-q} For all $1\leq p\leq p' <\infty$ and  all $1\leq q\leq q' <\infty$
	\begin{enumerate}
		\item
			\quad $ \DD_{p} \leq  \DD_{p,q}^{1/q}\leq\DD_{p,q'}^{1/q'}$
		\item\quad
			 $\DD_{0,q}^{1+1/p}\le \DD_{p,q}\leq \DD_{p',q}$.
		\item
			Restricted to the space $\{\X\in \XX \spec \diam(\X)\leq L\}$ for a given $L\in \R_+$:\\
			\quad $(qL^q)^{-(p'-p)} \cdot\DD_{p',q}^{p'}\leq  \DD_{p,q}^p \leq \Big(1+(qL^{q})^p\Big)\cdot\DD_{0,q}$ \quad 
and\quad  $\frac q{q'}L^{1-q/q'}\DD_{p,q'}\leq  \DD_{p,q}\leq   q L^{q-1}\cdot\DD_{p}$
	\end{enumerate}
\end{lemma}

\begin{proof}
(i) This follows from the basic inequality $|\d_0-d_1|\le |\d_0^q-d_1^q|^{1/q}\leq |\d_0^{q'}-d_1^{q'}|^{1/q'}$ valid for all positive real numbers $\d_0,d_1$ and all $1\le q\le q'$.\quad 
(ii) Simple applications of Markov and Jensen inequalities (same as for (ii) in Prop. 2.6, now with $\d^q$ in the place of $\d$ and $p'$ in the place of $p$). \quad
(iii) For the first estimate, we follow the argumentation of Prop. 2.6 and note that the diameter bound implies that $|\d_0^q-d_1^q|\le q L^q$. The second inequality follows from the fact that
$|\d_0^q-d_1^q|\le q L^{q-1}|\d_0-d_1|$.
\end{proof}

\begin{corollary}
For every sequence $(\X_n)_{n\in\N}$ in $\XX_0$ with uniformly bounded diameters, for every $\X_\infty\in \XX_0$  and all $p,q\in[1,\infty)$, the following are equivalent:
\begin{enumerate}
\item
$\X_n\to\X_\infty$ w.r.t.
$\DD_{0}$;

\item
$\X_n\to\X_\infty$ w.r.t.
$\DD_{0,q}$;

\item
$\X_n\to\X_\infty$ w.r.t.
$\DD_{p,q}$;

\item
$\X_n\to\X_\infty$ w.r.t.
$\DD_{p}$.

\end{enumerate}
    \end{corollary}

\begin{remark}
For all $q<q'$, since $ \DD_{p,q}^{1/q}\leq\DD_{p,q'}^{1/q'}$, the space $\XX_{pq'}$ is a subset of $\XX_{pq}$. Conversely, the space
$\big(\XX_{pq}, \DD_{p,q}\big)$ is isometrically embedded into $\big(\XX_{pq'}, \DD_{p,q'}\big)$ via 
\begin{equation*} \iota_q: \  [X,\d,\m] \ \mapsto \ [ X,\d^{q/q'},\m]. 
\end{equation*}
Indeed, for each complete, separable metric $\d$ on $X$ also $\d^{q/q'}$ is a complete separable metric on $X$ and 
\[\DD_{p,q'}\Big((X_0,\d_0^{q/q'},\m_0),(X_1,\d^{q/q'}_1,\m_1)\Big)=\DD_{p,q}\Big((X_0,\d_0,\m_0),(X_1,\d_1,\m_1)\Big).\]
\end{remark}

\subsection{Geometric Properties of the Space $\big(\XX_{pq}, \DD_{p,q}\big)$ for $p=2$}

\begin{theorem}\label{thmoptcoupl-q} For all $p,q\in[1,\infty]$, $\big(\XX_{pq}, \DD_{p,q}\big)$ is a geodesic space. 

For  each pair of mm-spaces $\X_0, \X_1\in \XX_p$ and each optimal coupling
  $\ol{\m}$ of them the family of metric
			measure spaces ${\X}_t=[{X_0\times X_1},{\d}_t,\ol{\m}]$, $ t\in (0,1)$,
					with
			\[
				{\d}_t\left((x_0,x_1),(y_0,y_1)\right):=\Big( (1-t)\d_0^q(x_0,y_0)+t \d_1^q(x_1,y_1)\Big)^{1/q}
			\]
			defines a geodesic $({\X}_t)_{0\leq t \leq 1}$ w.r.t. $\DD_{p,q}$ connecting $\X_0$ and $\X_1$.
	
If $p>1$ then each geodesic $({\X}_t)_{0\leq t\leq 1}$ w.r.t. $\DD_{p,q}$ is of this form.
			That is,  there exists an optimal coupling $\ol{\m}$ of the measures $\m_0, \m_1$, defined on the product space of  representatives of the endpoints, such that for each $t\in (0,1)$ a representative of the
			 isomorphism class
			$\X_t$ is given by $(X_0\times X_1,\d_t,\m)$ with $\d_t:=\big((1-t)\d^q_0+t\d^q_1\big)^{1/q}$.

\end{theorem}

\begin{proof}
The claim can be proven exactly as Theorem 3.1, now with $\d^q_0$ and $\d^q_1$ in the place of $\d_0$ and $\d_1$, resp. The only place where  the fact was used that the latter are metrics is to verify that $\d_t$ is a metric.
But also for general $q\ge1$,  $\d_t$ as defined above is  a metric on $X_0\times X_1$. Indeed, it is the $l_q$-product of the metrics $(1-t)^{1/q}\d_0$ on $X_0$ and $t^{1/q}\d_1$ on $X_1$.
\end{proof}

\begin{corollary} In the case $p>1$, if the initial point $\X_0$ of a geodesic $({\X}_t)_{t\in [0,1]}$ w.r.t. $\DD_{p,q}$  has no atoms then each inner point ${\X}_t$, $t\in (0,1)$, of the geodesic has no atoms.
\end{corollary}

More detailed geometric properties can be derived in the case $p=2$. All these results are proven in exatly the same way as for $q=1$.

\begin{theorem}
For each $q\in [1,\infty)$,	$(\XX_{2q},\DD_{2,q})$ is a geodesic space of nonnegative curvature in the sense of Alexandrov:
both the triangle comparison and the quadruple comparison property are satisfied. 

The metric completion	$(\ol\XX_{2,q},\DD_{2,q})$ of $(\XX_{2q},\DD_{2,q})$ is a complete length space of nonnegative curvature in the sense of Alexandrov.
\end{theorem}

\begin{theorem} For each $q\in [1,\infty)$,
 the space $\big(\XX_{2q}, \DD_{2,q}\big)$
			 is the cone over its unit sphere
 $(\XX^1_{2,q},\DD_{2,q}^{(1)})$. The latter is a geodesic space  with  curvature $\geq 1$: both the triangle and the quadruple comparison property are satisfied.

The completion  $\ol\XX_{2,q}$ is the cone over its unit sphere $\ol\XX^1_{2,q}$ (which is the completion of $\XX^1_{2,q}$).
 The latter
is a complete length space  with  curvature $\geq 1$
			in the sense of Alexandrov.
\end{theorem}

In order to identify the elements in $\ol\XX_{2,q}$ as pseudo mm-spaces, let us consider for given $q\ge1$ the map
\begin{equation} \iota_q: \  [X,\d,\m] \ \mapsto \ \auf X,\d^{q},\m\zu. \label{q-iso}
\end{equation}

It is an  isometric embedding of 
$\big(\XX_{2q}, \DD_{2,q}\big)$ into $\big(\YY, \DD\big)$ and thus also estends to an embedding of the completion.

\begin{theorem} For each $q\in [1,\infty)$,
 the space  $\iota_q\big(\ol\XX_{2,q}\big)$  conincides with the space of homomorphism classes of gauged measure spaces
$\auf X,\f,\m\zu $ for which $\f^{1/q}$ satisfies the triangle inequality $\m^2$-almost everywhere.

In other words, each element in $\ol\XX_{2,q}$ can be represented as a pseudo metric measure space $(X,\d,\m)$
with 
 $\auf X,\d^{q},\m\zu\in \YY$.
\end{theorem}

\begin{proof}
We follow the argumentation in the proofs of Lemma 5.14, Corollary 5.16, and Theorem 5.19.
Surjectivity is proven as before. It remains to check why the $\m^2$-a.e. triangle inequality is preserved under convergence.
The crucial point is that in the proof of Lemma 5.14 (ii), finally, convergence of $\d^q_n\to\d^q_\infty$ w.r.t. 
$L^2_s(X^2_\infty,\m^2_\infty)$
will be reduced to $\m_\infty^2$-a.e. convergence which makes the whole argument independent of $q$.
\end{proof}
The embedding \eqref{q-iso} is the key for a detailed understanding of the space $\ol\XX_{2,q}$, using all the properties of the space $\YY$ derived in Chapters 5 and 6, and it allows to use the tangential structure of the latter to study gradient flows as performed in Chapters 7 and 8.

\subsection{Geometric Properties of Intermediate Spaces 
w.r.t.  $\DD_{p,q}$ for $q=2$}

With respect to the parameter $p$, only the value $p=2$ plays a specific role in the analysis of the $L^{p,q}$-distortion distance. It is the only value of $p$ for which $\big(\XX_{pq}, \DD_{p,q}\big)$ becomes a space of nonnegative curvature in the sense of Alexandrov -- independent of $q$.

With respect to $q$, two values are of interest. The value $q=1$ is  the most natural one from the point of view of transportation theory and image analysis. And, of course, it also leads to the most simple formulas.

  The value $q=2$ is the only value for which 
geodesic interpolations of spaces with nonnegative (or nonpositive) pre-curvature are again spaces with
nonnegative (or nonpositive, resp.) pre-curvature -- independent of $p$.
Moreover,
 geodesic interpolations of distances in the case $q=2$ may be regarded as the metric counterpart to
 linear interpolation of metric tensors in Riemannian geometry.
Let us  illustrate the latter aspect by two fundamental examples.

\begin{example}
Let mm-spaces $(X_i,\d_i,\m_i)$ for $i=0,1$ be given with $X_i$ a compact subset of $\R^{n_i}$, $\m_i=c_i\Leb^{n_i}$, 
and let $\d_i$ be induced by a symmetric positive semidefinite matrix $g_i\in\R^{n_i\times n_i}$, that is,
\[ \d_i(x,y)=\sqrt{(x-y)\cdot g_i\cdot (x-y)} .\]
Then the $t$-intermediate metric $\d_t=\sqrt{(1-t)\d_0+t\d_1}$ on $\R^n$ with $n=n_0+n_1$ is induced by the 
symmetric positive semidefinite matrix \[g_t=(1-t)g_0 + t g_1\in\R^{n\times n}\]
being a block matrix with $(1-t)g_0$ in the upper left corner, $tg_1$ in the lower right corner and 0's elsewhere.
This holds true independent of the choice of the optimal coupling $\ol\m$.
Typically, $\supp(\ol\m)$ will be low dimensional subset of $\R^n$. For instance, if $n_0=n_1$ then for `nice' measures $\m_0, \m_1$ we expect that $\supp(\bar\m)$ is $n_0$-dimensional.
\end{example}

\begin{example}\label{riemannian}
Let  $n$-dimensional Riemannian manifolds $M_0$ and $M_1$ be given with Riemannian tensors $g_0$ and $g_1$. For $i=0,1$, let $\m_i$ and $\d_i$ denote the induced Riemannian volume measure and Riemannian distance, resp., and assume that the distances $\d_0$ and $\d_1$ are complete. Moreover, assume that an optimal coupling $\ol\m$ of the mm-spaces
 $(M_0,\d_0,\m_0)$ and  $(M_1,\d_1,\m_1)$
is given as $\ol\m=(\Id,\phi)_*\m_0$ in terms of a diffeomorphism $\phi: M_0\to M_1$.
Then for each $t\in(0,1)$, a representative of the $t$-intermediate space is given by $(M_0, \d_t,\m_0)$ with
\begin{equation}
 d_t=\sqrt{(1-t)\d^2_0+t \phi^*\d_1^2}.
\end{equation}
Thus the \emph{length metric} $\d_t^*$    induced by the metric $\d_t$ on $M_0$ coincides with the \emph{Riemannian distance} induced by the metric tensor
\begin{equation}
 g_t=(1-t)g_0+t \phi^*g_1.
\end{equation}
Here $\phi^*g_1$ denotes the pull back of the metric tensor $g_1$ by means of $\phi$ (which  is a metric tensor on $M_0$).
\end{example}

Besides its analogy to linear interpolation of metric tensors in the Riemannian case, the other main reason to regard geodesic interpolations w.r.t. the $L^{p,q}$-distance for $q=2$ is that this interpolation preserves suitable {\lq curvature bounds\rq} (rough versions of nonnegative and nonpositive curvature in the sense of Alexandrov).

\begin{definition}
(i) Given $\kappa\in\R$, we say that a metric space $(X,\d)$ has pre-curvature $\ge\kappa$ iff
each point $x\in X$ has a neighborhood $X_x$ which can be isometrically embedded into some complete length space $(\hat X_x,\hat\d_x)$ which has
curvature $\ge\kappa$ in the sense of Alexandrov.

(ii) Similarly, we say that $(X,\d)$ has pre-curvature $\le\kappa$ iff every point has a neighborhood which admits an isometric embedding
into some complete length space of curvature
 $\le\kappa$ in the sense of Alexandrov.

(iii) We say that $(X,\d)$ has pre-curvature locally bounded from below (or locally bounded from above) if
 every point has a neighborhood which admits an isometric embedding
into some complete length space of curvature
 bounded from below (or from above, resp.) in the sense of Alexandrov.

(iv) A metric measure space $(X,\d,\m)$  is said to have pre-curvature $\ge0$ (or $\le0$ or locally bounded from above or locally bounded from below) iff this is true for the metric space
$(\supp(m),\d)$.
\end{definition}

\begin{theorem}
For any $p\in[1,\infty)$, within the geodesic space $\big(\XX_{p2}, \DD_{p,2}\big)$, the following sets are (strongly) convex:
\begin{itemize}
\item
the set of mm-spaces with pre-curvature $\ge0$ 
\item
the set of mm-spaces with pre-curvature $\le0$ 
\item
the set of mm-spaces with pre-curvature locally bounded from below 
\item
the set of mm-spaces with pre-curvature locally bounded from above.
\end{itemize}
\end{theorem}

The claim of the theorem is a straightforward consequence of the following lemma.

\begin{lemma}
Let mm-spaces $(X_0,\d_0,\m_0)$ and  $(X_1,\d_1,\m_1)$
be given with pre-curvature $\ge\kappa_0$ and $\ge\kappa_1$, resp.
Then for each $t\in(0,1)$ and each $p\in[1,\infty)$, any $t$-intermediate space w.r.t. $\DD_{p,2}$ has  pre-curvature $\ge\kappa_t$ 
with $\kappa_t:=\min\{\frac1{1-t}\kappa_0,\frac1t\kappa_1,0\}$.

Similarly, if the spaces $(X_0,\d_0,\m_0)$ and  $(X_1,\d_1,\m_1)$ have pre-curvature $\le\kappa_0$ and $\le\kappa_1$, resp., then
any $t$-intermediate space w.r.t. $\DD_{p,2}$ has  pre-curvature $\le\kappa^*_t$ 
with $\kappa^*_t:=\max\{\frac1{1-t}\kappa_0,\frac1t\kappa_1,0\}$.
\end{lemma}

\begin{proof}
Let mm-spaces $(X_0,\d_0,\m_0)$ and  $(X_1,\d_1,\m_1)$
with pre-curvature $\ge\kappa_0$ and $\ge\kappa_1$, resp., be given and fix
 $t\in(0,1)$ as well as $p\in[1,\infty)$. Without restriction, assume that the measures $\m_0$ and $\m_1$ have full support.
Moreover, choose an $t$-intermediate space, say represented as $(X_0\times X_1, \d_t,\ol\m)$, and fix a 
point $x=(x_0,x_1)$ in the support of $\ol\m$.
Choose neighborhoods $X_0'$ and $X_1'$ of $x_0$ and $x_1$, resp., which can be embedded isometrically into 
 suitable complete length spaces $(\hat X_0,\hat\d_0)$ and $(\hat X_0,\hat\d_0)$, resp., with
curvature $\ge\kappa_0$ and $\ge\kappa_1$, resp., in the sense of Alexandrov.
Rescaling the metric by the factor $\sqrt{1-t}$ or $\sqrt t$ thus yields that
$(X_0', {\sqrt{1-t}}\d_0)$ can be isometrically embedded into $(\hat X_0, {\sqrt{1-t}}\hat\d_0)$ 
which is a complete length space of curvature $\ge  \frac1{1-t}\kappa_0$ in the sense of Alexandrov.
Similarly,
$(X_1', {\sqrt{t}}\d_1)$ can be isometrically embedded into $(\hat X_1, {\sqrt{t}}\hat\d_1)$ 
which is a complete length space of curvature $\ge  \frac1{t}\kappa_1$ in the sense of Alexandrov.

Thus the product space $(X_0'\times X_1',\d_t)=(X_0', {\sqrt{1-t}}\d_0)\times (X_1', {\sqrt{t}}\d_1)$ 
can be isometrically embedded into  a complete length space of curvature $\ge  \kappa_t=\min\{\frac1{1-t}\kappa_0,\frac1t\kappa_1,0\}$
 in the sense of Alexandrov.
Since $X_0'\times X_1'$ is a neighborhood of the given point $x=(x_0,x_1)$, this proves the  claim concerning lower bounds for the pre-curvature.

Exactly the same argumentation also proves the claim concerning upper bounds for the pre-curvature.
\end{proof}

\begin{remarks}
\begin{enumerate}
\item
If $(X,\d)$ is a complete length space then obviously it has pre-curvature $\ge\kappa$ (or $\le\kappa$) if and only if it has curvature $\ge\kappa$ (or $\le\kappa$, resp.) in the sense of Alexandrov.
\item If  a  metric space $(X,\d)$ has pre-curvature $\ge\kappa$ then for each point in $X$ there exists a neighborhood $X'$ such that for each quadruple of points 
$x_0,x_1,x_2,x_3\in X'$
\begin{equation}\label{lebpet+}
\bigg(\sum_{1\le i\le 3} \cosh\Big(\sqrt{-\kappa}\d(x_0,x_i)\Big)\bigg)^2\ge3+2\sum_{1\le i<j\le 3}\cosh\Big(\sqrt{-\kappa}\d(x_i,x_j)\Big)
\end{equation}
provided $\kappa<0$.
In the case $\kappa=0$ the latter has to be read as 
\begin{equation*}
3\sum_{1\le i\le 3} \d^2(x_0,x_i)\ge\sum_{1\le i<j\le 3}\d^2(x_i,x_j)
\end{equation*} 
and 
in the case $\kappa>0$ as 
\begin{equation*}
\bigg(\sum_{1\le i\le 3} \cos\Big(\sqrt{\kappa}\d(x_0,x_i)\Big)\bigg)^2\le3+2\sum_{1\le i<j\le 3}\cos\Big(\sqrt{\kappa}\d(x_i,x_j)\Big).
\end{equation*}
If $(X,\d)$ is a complete length space also the converse holds true. Indeed, according to Lebedeva and Petrunin \cite{lp} every complete length space 
which locally satisfies \eqref{lebpet+} has  curvature $\ge\kappa$ in the sense of Alexandrov.
\item
 If  a  metric space $(X,\d)$ has pre-curvature $\le0$ then for each point in $X$ there exists a neighborhood $X'$ such that for each quadruple of points 
$x_1,x_2,x_3,x_4\in X'$
\begin{equation}\label{beni}
\d^2(x_1,x_2)+\d^2(x_2,x_3)+\d^2(x_3,x_4)+\d^2(x_4,x_1)\ge\d^2(x_1,x_3)+\d^2(x_2,x_4).
\end{equation}
If $(X,\d)$ is a complete length space also the converse holds true. Indeed, according to  Berg and Nikolaev \cite{bn} every complete length space 
which locally satisfies \eqref{beni} has  curvature $\le0$ in the sense of Alexandrov.

\item
In general, the property of being a length space of curvature $\ge0$ (or $\le0$) is not preserved by geodesic interpolations simply because the $t$-intermediate points 
$(X_0\times X_1, \d_t,\ol\m)$
typically are not length spaces.
The most natural strategy to overcome this, might be to replace the metric $\d_t$ by the length metric $\d_t^*$ induced by $\d_t$.
This, however, will change both, the property of being an $t$-intermediate point interpolating between $\d_0$ and $\d_1$  w.r.t.  $\DD_{p,2}$
and the bound on the pre-curvature for $\d_t$.
\end{enumerate}
\end{remarks}

\subsection{Existence of Transport Maps}
In this section, finally, we address the challenging uniqueness problem for optimal couplings and the challenging  existence problem for transport maps. 
Due to the quadratic dependency of the cost functional on the coupling, the minimization problem for the $ \DD_{p,q}$-distance is far more complicated than the minimization problem for the classical $W_p$-distance. So far, only partial results are available. However, under symmetry assumptions, a full characterization is possible.
Throughout the sequel, we will restrict ourselves to the case $p=q=2$. 
 Our main result is:
 
\wichtig{Let $(X_0,\d_0,\m_0)$ and  $(X_1,\d_1,\m_1)$ be two rotationally invariant metric measure spaces, both equipped with the Euclidean distance and with probability measures which are absolutely continuous w.r.t~the $n$-dimensional Lebesgue measure. Then every optimal coupling of them w.r.t.~the $L^{2,2}$-distortion distance $\DD_{2,2}$ is given by a transport map. It is obtained by identifying the two barycenters of the spaces and then monotonically re-arranging the radial distributions. The map (and thus the optimal coupling) is unique up to composition with rotations of the second space.}

To get started, we formulate some auxiliary results.

\begin{lemma} Let $\m_0,\m_1\in \prob(\R^n)$ be given with barycenters $z_i:=\int x\,d\m_i(x)\in\R^n$ and put $\m_i':=(T_i)_*\m_i$ with $T_i(x):=x-z_i$. 
Then for a  probability measure $\bar \m\in \prob(\R^{2n})$ the following are equivalent:
\begin{itemize}
\item[(i)] the measure $\bar \m$ is a minimizer of 
$$\int_{\R^{2n}}\int_{\R^{2n}}\Big| |x_0-y_0|^2- |x_1-y_1|^2\Big|^2\, d\bar m(x_0,x_1)\, d\bar \m(y_0,y_1)
$$
among all couplings of $\m_0$ and $\m_1$; 
\item[(ii)] 
 the measure $\bar \m':=\big(T_0,T_1\big)_*\bar \m$ is a minimizer of 
$$\int_{\R^{2n}}\int_{\R^{2n}}\Big| |x_0-y_0|^2- |x_1-y_1|^2\Big|^2\, d\bar \m'(x_0,x_1)\, d\bar m'(y_0,y_1)
$$
among all couplings of $\m_0'$ and $\m_1'$;
\item[(iii)] 
 the measure $\bar \m':=\big(T_0,T_1\big)_*\bar \m$ is a maximizer of 
 $$\int_{\R^{2n}}|x_0|^2\cdot |x_1|^2\, d\bar m'(x_0,x_1)
 +
2 \int_{\R^{2n}}\int_{\R^{2n}}\langle x_0,y_0\rangle\cdot \langle x_1,y_1\rangle\, d\bar \m'(x_0,x_1)\, d\bar \m'(y_0,y_1)
$$
among all couplings of $\m_0'$ and $\m_1'$.
\end{itemize}
\end{lemma}

\begin{proof} The equivalence of (i) and (ii) is obvious. To proceed, observe first that (ii) is equivalent to
\begin{itemize}
\item[(ii')] 
 the measure $\bar \m'$ is a maximizer of 
\begin{equation}\label{cost22}\int_{\R^{2n}}\int_{\R^{2n}} |x_0-y_0|^2 \cdot |x_1-y_1|^2\, d\bar \m'(x_0,x_1)\, d\bar \m'(y_0,y_1)
\end{equation}
among all couplings of $\m_0'$ and $\m_1'$ 
\end{itemize}
(which simply follows by squaring out the integrand in (ii) and putting aside the terms which depend only on $x_0,y_0$ or only on $x_1,y_1$).
Squaring out the integrand in \eqref{cost22} provides nine terms. Two of them lead to a constant contribution: $|x_0|^2\cdot |y_1|^2$ and $|x_1|^2\cdot |y_0|^2$; four of them lead to  vanishing integrals (since the measures $\m_0'$ and $\m_1'$ are centered): 
$-2 |x_0|^2\, \langle x_1,y_1\rangle$, $-2 |y_0|^2\, \langle x_1,y_1\rangle$, 
$-2 |x_1|^2\, \langle x_0,y_0\rangle$, and
$-2 |y_1|^2\, \langle x_0,y_0\rangle$. The remaining three terms yield twice the integrands in (iii).
\end{proof}
\begin{lemma} Assume that $\m_0$ and $\m_1$ are centered, rotationally invariant probability measures on $\R^n$. Denote the respective radial distributions by $\nu_0$ and $\nu_1$ such that
$\nu_i\big([0,r]\big)=\m_i\big( \bar B_r(0)\big)$ for all $r$, and 
denote by $\sigma$ the normalized uniform distribution on ${\mathbb S}^n$. 
For $i=0,1$, write $x_i=r_i\,\varphi_i, y_i=s_i\,\psi_i$ with $\varphi_i,\psi_i\in {\mathbb S}^n$ and $r,s\in\R_+$.

(i)
Then
$$d\m_i(x_i)=d\nu_i(r)\,d\sigma(\varphi_i),$$
and every coupling $\bar \m$ of $\m_0$ and $\m_1$ can be represented as 
$$d\bar \m( x_0,x_1)=\, d\bar\nu_{\sigma_0,\sigma_1}(r_0,r_1)d\bar\sigma(\varphi_0,\varphi_1)$$
where $\bar\sigma$ denotes some coupling of $\sigma$ with itself,  and for each $\sigma_0,\sigma_1$ the measure $\bar\nu_{\sigma_0,\sigma_1}$ denotes some coupling of $\nu_0$ and $\nu_1$.

(ii) The measure $\bar \m$ is a maximizer of 
$$\int_{\R^{2n}}|x_0|^2\cdot |x_1|^2\, d\bar m(x_0,x_1)$$
among all couplings of $\m_0$ and $\m_1$ if and only if for $\bar\sigma$-almost every $\varphi_0,\varphi_1$ the measure   $\bar\nu_{\sigma_0,\sigma_1}$ is a maximizer of  
$$\int_{\R^2_+} r_0^2\,r_1^2\, d\bar\nu_{\sigma_0,\sigma_1}(r_0,r_1)$$
among all couplings of $\nu_0$ and $\nu_1$.

(iii) For any $a>0$,  the functional 
$$\int_{\R^2_+} r_0^a\,r_1^a\, d\bar\nu(r_0,r_1)$$
has a unique maximizer among all couplings $\bar\nu$ of $\nu_0$ and $\nu_1$, given by monotone re-arrangement. If $\nu_0$ is absolutely continuous, then this is given by a map:
$$\bar\nu=\big( {\rm Id}, T\big)_*\nu_0$$
with $T:=F_1\circ F_0^{-1}$ where $F_i(r):=\nu_i([0,r])$.

(iv) The measure $\bar \m$ is a maximizer of 
$$\int_{\R^{2n}}\int_{\R^{2n}}\langle x_0,y_0\rangle\cdot \langle x_1,y_1\rangle\, d\bar \m(x_0,x_1)\, d\bar \m(y_0,y_1)
$$
among all couplings of $\m_0$ and $\m_1$ if and only if 
$$\bar\sigma=\big( {\rm Id}, U\big)_*\sigma$$
for some orthogonal map $U:{\mathbb S}^n\to {\mathbb S}^n$, and if for $\bar\sigma$-almost every $\varphi_0,\varphi_1$ the measure   $\bar\nu_{\sigma_0,\sigma_1}$ is a maximizer of  
$$\int_{\R^2_+} r_0\,r_1\, d\bar\nu_{\sigma_0,\sigma_1}(r_0,r_1)$$
among all couplings of $\nu_0$ and $\nu_1$.
\end{lemma}

\begin{proof} (i) and (ii) are obvious. In (iii),  by monotone re-parametrization $r\mapsto r^{1/a}$, the general case can be reduced to the case $a=1$. In this case, maximizers of the given integral are minimizers of the quadratic cost $\int |r_0-r_1|^2d\bar\nu(r_0,r_1)$. 
For this transport problem, it is well-known that there exists a unique optimizer and that this is given by monotone re-arrangement. Moreover, it is given in terms of the map $T:=F_1\circ F_0^{-1}$ if the first marginal is absolutely continuous, \cite{Vi03}.
 
To see (iv), consider
\begin{eqnarray*}\lefteqn{
\int_{\R^{2n}}\int_{\R^{2n}}\langle x_0,y_0\rangle\, \langle x_1,y_1\rangle\, d\bar m(x_0,x_1)\, d\bar \m(y_0,y_1)}\\
&=&\int_{({\mathbb S}^n)^2}\int_{({\mathbb S}^n)^2} \bigg[\int_{\R_+^2} r_0 r_1d\bar\nu_{\varphi_0,\varphi_1}(r_0,r_1)\bigg]\cdot
\bigg[\int_{\R_+^2} s_0 s_1d\bar\nu_{\psi_0,\psi_1}(s_0,s_1)\bigg]
\\
&&\qquad\qquad\qquad\qquad\cdot 
\langle \varphi_0,\psi_0\rangle\, \langle \varphi_1,\psi_1\rangle\, d\bar \sigma(\varphi_0,\varphi_1)\, d\bar\sigma(\psi_0,\psi_1)\\
&=&\int_{({\mathbb S}^n)^2}\int_{({\mathbb S}^n)^2} \bigg[\int_{\R_+^2} r_0 r_1d\bar\nu_{\varphi_0,\varphi_1}(r_0,r_1)\bigg]\cdot
\bigg[\int_{\R_+^2} s_0 s_1d\bar\nu_{\psi_0,\psi_1}(s_0,s_1)\bigg]
\\
&&\qquad\qquad\cdot\frac12\bigg[\langle \varphi_0,\psi_0\rangle^2+\langle \varphi_1,\psi_1\rangle^2-\big|
\langle \varphi_0,\psi_0\rangle-  \langle \varphi_1,\psi_1\rangle
\Big|^2\bigg] 
\, d\bar \sigma(\varphi_0,\varphi_1)\, d\bar\sigma(\psi_0,\psi_1).
\end{eqnarray*}
Now according to (iii),
\begin{equation}\label{mixed7}\int_{({\mathbb S}^n)^2}\int_{({\mathbb S}^n)^2} \bigg[\int_{\R_+^2} r_0 r_1d\bar\nu_{\varphi_0,\varphi_1}(r_0,r_1)\bigg]\cdot
\bigg[\int_{\R_+^2} s_0 s_1d\bar\nu_{\psi_0,\psi_1}(s_0,s_1)\bigg]\cdot
\langle \varphi_0,\psi_0\rangle^2\, d\bar \sigma(\varphi_0,\varphi_1)\, d\bar\sigma(\psi_0,\psi_1)\end{equation} is maximal if and only if 
$\bar\nu_{\varphi_0,\varphi_1}=\bar\nu$
for $\bar\sigma$-a.e.~$(\varphi_0,\varphi_1)\in ({\mathbb S}^n)^2$
as well as
$\bar\nu_{\psi_0,\psi_1}=\bar\nu$
for $\bar\sigma$-a.e.~$(\psi_0,\psi_1)\in ({\mathbb S}^n)^2$
where $\bar\nu:=\big( {\rm Id}, T\big)_*\nu_0$ as defined in (iii). 
In this case, independent of the choice of the coupling $\bar\sigma$, the quantity in \eqref{mixed7} coincides with,
$$\bigg[\int_{\R_+^2} r_0 r_1d\bar\nu(r_0,r_1)\bigg]^2\cdot \int_{({\mathbb S}^n)^2}\int_{({\mathbb S}^n)^2}\langle \varphi_0,\psi_0\rangle^2\, d\bar \sigma(\varphi_0,\varphi_1)\, d\bar\sigma(\psi_0,\psi_1)=\bigg[\int_{\R_+^2} r_0 r_1d\bar\nu(r_0,r_1)\bigg]^2\cdot \frac1n\,.
$$
Analogously, for the term with $\langle \varphi_1,\psi_1\rangle^2$  in the place of $\langle \varphi_0,\psi_0\rangle^2$.
It remains to consider the integral
\begin{eqnarray}\nonumber\lefteqn{-\frac12 \int_{({\mathbb S}^n)^2}\int_{({\mathbb S}^n)^2}\bigg[\int_{\R_+^2} r_0 r_1d\bar\nu_{\varphi_0,\varphi_1}(r_0,r_1)\bigg]\cdot
\bigg[\int_{\R_+^2} s_0 s_1d\bar\nu_{\psi_0,\psi_1}(s_0,s_1)\bigg]}\\
&&\qquad\qquad\qquad\qquad\cdot
\bigg|
\langle \varphi_0,\psi_0\rangle-  \langle \varphi_1,\psi_1\rangle
\bigg|^2
\, d\bar \sigma(\varphi_0,\varphi_1)\, d\bar\sigma(\psi_0,\psi_1).
\label{mixed9}
\end{eqnarray}
Since each of the terms $\int r_0 r_1d\bar\nu_{\varphi_0,\varphi_1}(r_0,r_1)$ and $\int s_0 s_1d\bar\nu_{\psi_0,\psi_1}(s_0,s_1)$ is positive, the maximum in \eqref{mixed9} is attained if and only if 
$$\langle \varphi_0,\psi_0\rangle= \langle \varphi_1,\psi_1\rangle\qquad \text{for }\bar\sigma\otimes\bar\sigma\text{-a.e.~}(\varphi_0,\varphi_1,\psi_0,\psi_1)\in ({\mathbb S}^n)^4.
$$ 
This obviously is equivalent to the fact that  
$$|\varphi_0-\psi_0|^2= |\varphi_1-\psi_1|^2\qquad \text{for }\bar\sigma\otimes\bar\sigma\text{-a.e.~}(\varphi_0,\varphi_1,\psi_0,\psi_1)\in ({\mathbb S}^n)^4
$$ 
which in turn is equivalent to the fact that  $\bar\sigma$ is a minimizer of 
$$\int_{({\mathbb S}^n)^2}\int_{({\mathbb S}^n)^2} \Big| |\varphi_0-\psi_0|^2- | \varphi_1-\psi_1|^2\Big|^2\, d\bar \sigma(\varphi_0,\varphi_1)\, d\bar\sigma(\psi_0,\psi_1)$$
 among all couplings of $\sigma$ with itself. According to Lemmas  9.2 and 1.10, this vanishing of the $L^{2,2}$-distortion distance is the case if and only if the coupling is induced by an isometry of  ${\mathbb S}^n$.
\end{proof}

Putting together these results, we have proven our main result in this section.
\begin{theorem} 
 Let absolutely continuous probability measures $\m_0$ and $\m_1$ on  $\R^n$ be given, each of them being rotationally invariant around its barycenter $z_0$ or $z_1$, resp. That is,
 $(U_i)_*\m_i=\m_i$
 for each $U\in O(n)$ and $i=0,1$ where $U_i(x):=U(x-z_i)+z_i$.
 
 (i) Then every $\bar \m\in\Cpl(\m_0,\m_1)$ which minimizes 
$$\int_{\R^{2n}}\int_{\R^{2n}} \Big| |x_0-y_0|^2- |x_1-y_1|^2\Big|^2d\bar \m(x_0,x_1) \, d\bar \m(y_0,y_1)$$   
 is given as
$$\bar \m =\big({\rm Id},\Phi\big)_*\m_0$$
in terms of a \wichtig{transport map} $\Phi: \R^n\to\R^n$.

(ii) The transport map $\Phi$ (and thus also the optimal coupling $\bar \m$) is \wichtig{unique} up to composition with rotations. That is, if 
$\bar \m' =\big({\rm Id},\Phi'\big)_*\m_0$ is another optimal coupling then 
$$\Phi'=U_1\circ \Phi$$
for some  $U\in O(n)$ and $U_1(x):=U(x-z_1)+z_1$.

(iii) The transport map is \wichtig{constructed} as  follows: For $i=0,1$, let $\nu_i$ be the radial distribution of $\m_i$ around $z_i$,
and let $F_i$ be the respective distribution function, i.e.
$$F_i(r):=\nu_i\big([0,r]\big):=\m_i\big(\bar B_r(z_i)\big).$$
Then the map $T:=F_1\circ F_0^{-1}: \R_+\to\R_+$ pushes forward $\nu_0$ to $\nu_1$ (``monotone re-arrangement''):
$$T_*\nu_0=\nu_1.$$
Moreover, for every $U\in O(n)$, a transport map $\Phi:\R^n\to\R^n$ with the requested properties is given by
$$\Phi\big(z_0+r\varphi):=z_1+ T(r)\, U(\varphi)\qquad (\forall r\in\R_+, \varphi\in{\mathbb S}^n).$$
\end{theorem}

\printindex
\vfill

\bibliographystyle{amsalpha}
\bibliography{refsos}

\end{document}